\documentclass[a4paper,11pt]{amsart}
\usepackage{amsmath,amssymb,a4wide}
\newtheorem{theorem}{Theorem}[section]
\newtheorem{lemma}[theorem]{Lemma}
\newtheorem{proposition}[theorem]{Proposition}
\newtheorem{corollary}[theorem]{Corollary}
\newtheorem{claim}[theorem]{Claim}
\theoremstyle{definition}

\theoremstyle{remark}
\newtheorem{remark}{Remark}[section]
\numberwithin{equation}{section}
\DeclareMathOperator{\sech}{sech}
\DeclareMathOperator{\diag}{diag}

\DeclareMathOperator{\supp}{supp}
\newcommand{\R}{\mathbb{R}}
\newcommand{\C}{\mathbb{C}}
\newcommand{\N}{\mathbb{N}}
\newcommand{\Z}{\mathbb{Z}}
\newcommand{\eps}{\epsilon}
\renewcommand{\a}{\alpha} 
\newcommand{\hxi}{\hat{\xi}}
\newcommand{\heta}{\hat{\eta}}
\newcommand{\tc}{\tilde{c}}

\newcommand{\barr}{\bar{r}}
\newcommand{\tu}{\tilde{u}}
\newcommand{\baru}{\bar{u}}
\newcommand{\tf}{\tilde{f}}
\newcommand{\barf}{\bar{f}}
\newcommand{\tv}{\tilde{v}}
\newcommand{\tx}{\tilde{x}}
\newcommand{\ty}{\tilde{y}}
\newcommand{\la}{\langle}
\newcommand{\ra}{\rangle}
\newcommand{\pd}{\partial}
\newcommand{\spann}{\operatorname{span}}
\newcommand{\mL}{\mathcal{L}}
\newcommand{\mF}{\mathcal{F}}

\begin{document}
\title{Asymptotic Linear Stability of the Benney-Luke equation in 2D}
\author{Tetsu Mizumachi}
\address{Division of Mathematical and Information Sciences\\
Hiroshima University\\
1-7-1 Kagamiyama, Higashi-Hiroshima 739-8521, Japan}
\email{tetsum@hiroshima-u.ac.jp}
\author{Yusuke Shimabukuro}
\address{ Institute of Mathematics\\
Academia Sinica\\
6F, Astronomy-Mathematics Building, No. 1, Sec. 4, Roosevelt Road, Taipei 10617, Taiwan}
\email{shimaby@gate.sinica.edu.tw}
\begin{abstract} 
  In this paper, we study transverse linear stability of line solitary
  waves to the $2$-dimensional Benney-Luke equation which arises in
  the study of small amplitude long water waves in $3$D. In the case
  where the surface tension is weak or negligible, we find a curve of
  resonant continuous eigenvalues near $0$.  Time evolution of these
  resonant continuous eigenmodes is described by a $1$D damped wave
  equation in the transverse variable and it gives a linear approximation
  of the local phase shifts of modulating line solitary waves. In
  exponentially weighted space whose weight function increases in the
  direction of the motion of the line solitary wave, the other part of
  solutions to the linearized equation decays exponentially as $t\to\infty$.
\end{abstract}

\maketitle
\tableofcontents
\section{Introduction}
\label{sec:intro}
In this paper, we study transverse linear stability of line solitary waves
for the Benney-Luke equation
\begin{equation}
  \label{eq:BL}
\pd_t^2\Phi-\Delta\Phi+a\Delta^2\Phi-b\Delta\pd_t^2\Phi
+(\pd_t\Phi)(\Delta\Phi)+\pd_t(|\nabla\Phi|^2)=0
\quad\text{on $\R\times\R^2$.}
\end{equation}
The Benney-Luke equation is an approximation model of small amplitude
long water waves with finite depth originally derived by Benney and
Luke \cite{BL64} as a model for 3D water waves.  Here
$\Phi=\Phi(t,x,y)$ corresponds to a velocity potential of water waves.
We remark that \eqref{eq:BL} is an isotropic model for propagation of
water waves whereas KdV, BBM and KP equations are unidirectional
models.  See e.g. \cite{BCG13,BCS02} for the other bidirectional
models of 2D and 3D water waves.
\par
The parameters $a$, $b$ are positive and satisfy $a-b=\hat{\tau}-1/3$, where
$\hat{\tau}$ is the inverse Bond number. 
If we think of waves propagating in one direction, slowly evolving in
time and having weak transverse variation, then the Benney-Luke equation
can be formally reduced to the KP-II equation if $0<a<b$ and
to the KP-I equation if $a>b>0$. 
More precisely, the Benney-Luke equation \eqref{eq:BL} is reduced to
$$2f_{\tx\tilde{t}}+(b-a)f_{\tx\tx\tx\tx}+3f_{\tx}f_{\tx\tx}+f_{\ty\ty}=0$$
in the coordinate $\tilde{t}=\eps^3t$, $\tx=\eps(x-t)$ and
$\ty=\eps^2 y$ by taking terms only of order $\eps^5$,
where $\Phi(t,x,y)=\eps f(\tilde{t},\tx,\ty)$.
See e.g. \cite{MilKel96} for the details.
In this paper, we will assume $0<a<b$, which corresponds to the case where
the surface tension is weak or negligible.
\par

The solution $\Phi(t)$ of the Benney-Luke equation \eqref{eq:BL} formally
satisfies the energy conservation law
\begin{equation}
  \label{eq:conserve}
E(\Phi(t),\pd_t\Phi(t))=E(\Phi_0,\Psi_0)
\quad\text{for $t\in\R$,}
\end{equation}
where 
$$
E(\Phi,\Psi):=\int_{\R^2}
\left\{|\nabla\Phi|^2+a(\Delta\Phi)^2+\Psi^2+b|\nabla\Psi|^2\right\}\,
dxdy\,,$$
and \eqref{eq:BL} is globally well-posed in the energy class
$(\dot{H}^2(\R^2)\cap \dot{H}^1(\R^2))
\times H^1(\R^2)$ (see \cite{Q13}).
The Benney Luke equation \eqref{eq:BL} has a 3-parameter family of line solitary wave
solutions
\begin{equation}
\label{eq:lsw}
\Phi(t,x,y)=\varphi_c(x\cos\theta+y\sin\theta-ct+\gamma)\,,
\quad \pm c>1\,,\quad \gamma\in\R\,,\quad \theta\in[0,2\pi)\,,
\end{equation}
where 
$$\varphi_c(x)=\frac{2(c^2-1)}{c\a_c}\tanh(\frac{\a_c}{2}x)\,,
\quad \a_c=\sqrt{\frac{c^2-1}{bc^2-a}}\,,$$
and $$q_c(x):=\varphi_c'(x)=\frac{c^2-1}{c}\sech^2\bigl(\frac{\a_cx}{2}\bigr)$$
is a solution of
\begin{equation}
  \label{eq:qc}
  (bc^2-a)q_c''-(c^2-1)q_c+\frac{3c}{2}q_c^2=0\,.
\end{equation}
Stability of solitary  waves to the $1$-dimensional Benney-Luke equation
are studied by \cite{Q03} for the strong surface tension case
$a>b>0$ and by \cite{MPQ13} for the weak surface tension case $b>a>0$. 
If $a>b>0$, then \eqref{eq:BL} has a stable ground state for $c$
satisfying $0<c^2<1$ (\cite{PQ99,Q05}). 
See also \cite{Marics} for the algebraic decay property of the ground state.
In view of \cite{RT1,RT2},
line solitary waves for the $2$-dimensional Benney-Luke equation
are expected to be unstable in this parameter regime.
On the other hand if $0<a<b$ and $c:=\sqrt{1+\eps^2}$ is close to $1$ (the sonic speed), 
then $\varphi_c(x-ct)$ is expected to be transversally stable because
$q_c(x)$ is similar to a KdV $1$-soliton and
line solitons of the KP-II equation is transversally stable
(\cite{KP,Miz13,Miz15}).
\par
The dispersion relation for the linearization of \eqref{eq:BL} around $0$ is 
$$\omega^2=(\xi^2+\eta^2)\frac{1+a(\xi^2+\eta^2)}{1+b(\xi^2+\eta^2)}$$
for a plane wave solution $\Phi(t,x,y)=e^{i(x\xi+y\eta-\omega t)}$.
If $b>a>0$, then $|\nabla\omega|\le 1$ and line solitary waves travel
faster than the maximum group velocity of linear waves.  Measuring the
size of perturbations with an exponentially weighted norm biased in
the direction of motion of a line solitary wave, we can observe that
perturbations which are decoupled from the line solitary wave decay as
$t\to\infty$.
In the $1$-dimensional case, small solitary waves are exponentially linearly
stable in the weighted space and  $\lambda=0$ is an isolated eigenvalue
of the linearized operator (see \cite{MPQ13}).
In our problem, however, the value $\lambda=0$ is not an isolated eigenvalue.
This is because line solitary waves do not decay in the transverse
direction.  Indeed, for any size of line solitary waves of
\eqref{eq:BL}, there appears a curve of continuous spectrum that goes
through $\lambda=0$ and locates in the stable half plane
(Theorem~\ref{thm:1}).  The curve of continuous eigenvalues has to do
with perturbations that propagate toward the transverse direction
along the crest of the line solitary wave
(Theorem~\ref{thm:linear-dynamics}).  If line solitary waves are small,
the rest of the spectrum locates in a stable half plane
$\{\lambda\in\C\mid \Re\lambda\le -\beta<0\}$ 
(Theorem~\ref{thm:spectrum}).  For the KP-II equation,
the spectrum of the linearized operator around a $1$-line
soliton near $\lambda=0$ can be obtained explicitly thanks to the
integrability of the equation (see \cite{APS,Burtsev85,Miz15}).  In this
paper, we will use the Lyapunov-Schmidt method to find resonant
eigemodes of the linearized operator.

To prove non-existence of unstable modes for the linearized operator
around small line solitary waves, we make use of the KP-II
approximation of the the linearized operator of \eqref{eq:BL} on long
length scales and make use of the transverse linear stability of line
solitons for the KP-II equation.  For $1$-dimensional long wave
models, non-existence of unstable modes for the linearized operator
around solitary waves has been proved by utilizing spectral stability
of KdV solitons.  See e.g. \cite{FP3,Miz13,MW96,PS2,PW97} and
\cite{MPQ13} for the $1$-dimensional Benney-Luke equation.  We expect
that the KP-II approximation of the linearized operator is useful to
other $2$-dimensional long wave models such as KP-BBM and Boussinesq
systems with no surface tension (see e.g. \cite{CCD10}).
\par
%Our plan of the present paper is as follows.
%In Section~\ref{sec:statement},
%In Section~\ref{sec:resonant},
%In Section~\ref{sec:free operator},
%In Section~\ref{sec:spectrum},
%In Section~\ref{sec:linear-stability},
%In Section~\ref{sec:bifurcation}, we prove Theorem~\ref{thm:1}.
%In Section~\ref{sec:dynamics}, we prove 
\par
Now let us introduce several notations. 
For an operator $A$, we denote by $\sigma(A)$ the spectrum
and by $D(A)$ and $R(A)$ the domain and the range of the
operator $A$, respectively.
For Banach spaces $V$ and $W$, let $B(V,W)$ be the space of all
linear continuous operators from $V$ to $W$ and
$\|T\|_{B(V,W)}=\sup_{\|x\|_V=1}\|Tu\|_W$ for $T\in B(V,W)$.
We abbreviate $B(V,V)$ as $B(V)$.
For $f\in \mathcal{S}(\R^n)$ and $m\in \mathcal{S}'(\R^n)$, let 
\begin{gather*}
(\mathcal{F}f)(\xi)=\hat{f}(\xi)
=(2\pi)^{-n/2}\int_{\R^n}f(x)e^{-ix\xi}\,dx\,,\\
(\mathcal{F}^{-1}f)(x)=\check{f}(x)=\hat{f}(-x)\,,
\end{gather*}
and $(m(D)f)(x)=(2\pi)^{-n/2}(\check{m}*f)(x)$.
We denote $\la f,g\ra$ by
$$\la f,g\ra=\sum_{j=1}^m \int_\R f_j(x)\overline{g_j(x)}\,dx$$
for $\C^m$-valued functions $f=(f_1,\cdots,f_m)$ and
$g=(g_1,\cdots,g_m)$.

Let $L^2_\a(\R^2)=L^2(\R^2;e^{2\a x}dxdy)$, $L^2_\a(\R)=L^2(\R;e^{2\a x}\,dx)$
and let $H^k_\a(\R^2)$ and $H^k_\a(\R)$ be Hilbert spaces with the norms
\begin{gather*}
\|u\|_{H^k_\a(\R^2)}=\left(\|\pd_x^ku\|_{L^2_\a(\R^2)}^2
+\|\pd_y^ku\|_{L^2_\a(\R^2)}^2+\|u\|_{L^2_\a(\R^2)}^2\right)^{1/2}\,,
\\
\|u\|_{H^k_\a(\R)}=\left(\|\pd_x^ku\|_{L^2_\a(\R)}^2+\|u\|_{L^2_\a(\R)}^2\right)^{1/2}\,.
\end{gather*}
We use $a\lesssim b$ and $a=O(b)$ to mean that there exists a
positive constant such that $a\le Cb$. 
Various constants will be simply denoted
by $C$ and $C_i$ ($i\in\mathbb{N}$) in the course of the
calculations. We denote $\la x\ra=\sqrt{1+x^2}$ for $x\in\R$.
\bigskip

\section{Statement of the result}
\label{sec:statement}
Since \eqref{eq:BL} is isotropic and translation invariant,
we may assume $\theta=\gamma=0$ in \eqref{eq:lsw} without loss of generality.
Let $\Psi=\pd_t\Phi$, $A=I-a\Delta$ and $B=I-b\Delta$. Then 
in the moving coordinate $z=x-ct$, the Benney-Luke equation \eqref{eq:BL}
can be rewritten as
\begin{equation}
  \label{eq:BL1}
\left\{
  \begin{aligned}
& \pd_t\Phi=c\pd_z\Phi+\Psi\,,\\
& \pd_t\Psi=c\pd_z\Psi+B^{-1}A\Delta \Phi
-B^{-1}(\Psi\Delta\Phi+2\nabla\Phi\cdot\nabla\Psi)\,,
\end{aligned}\right.
\end{equation}
Let $r_c(z)=-cq_c(z)$.
Linearizing \eqref{eq:BL1} around $(\Phi,\Psi)=(\varphi_c(z), r_c(z))$,
we have
\begin{gather}
  \label{eq:linear}
\pd_t  \begin{pmatrix}\Phi\\ \Psi\end{pmatrix}
=\mL \begin{pmatrix}\Phi\\ \Psi\end{pmatrix}\,,
\\ \notag
\mL=\mL_0+V\,,\quad 
\mL_0=
\begin{pmatrix}  c\pd_z & 1 \\ B^{-1}A\Delta & c\pd_z\end{pmatrix}\,,
\\ \notag
\\\
V=-B^{-1}\begin{pmatrix}  0 & 0 \\ v_{1,c} & v_{2,c}\end{pmatrix}\,,
\quad v_{1,c}=2r_c'(z)\pd_z+r_c(z)\Delta\,,\quad
v_{2,c}=2q_c(z)\pd_z+q_c'(z)\,.
\end{gather}
We study linear stability of \eqref{eq:linear} in a weighted space
$X:=H^1_\a(\R^2)\times L^2_\a(\R^2)$.
Let $\mL(\eta)u(z)=e^{-iy\eta}\mL(e^{iy\eta}u(z))$ for $\eta\in\R$. Note that
$V$ is independent of $y$. For each small $\eta\ne0$, the operator
$\mL(\eta)$ has two stable eigenvalues.
\begin{theorem}
\label{thm:1}
Let $0<a<b$ and $k\in\N$. Fix $c>1$ and $\a\in(0,\a_c)$.
Then there exist a positive constant $\eta_0$,
$\lambda(\eta)\in C^\infty([-\eta_0,\eta_0])$,
$$\zeta(\cdot,\eta)\in
C^\infty([-\eta_0,\eta_0];H^k_\a(\R)\times H^{k-1}_\a(\R))\,,\quad
\zeta^*(\cdot,\eta)\in
C^\infty([-\eta_0,\eta_0];H^k_{-\a}(\R)\times H^{k-1}_{-\a}(\R))$$
such that
\begin{gather} \notag
\mL(\eta) \zeta(z,\pm\eta)=\lambda(\pm\eta)\zeta(z,\pm\eta)\,,
\quad 
\mL(\eta)^* \zeta^*(z,\pm\eta)=\lambda(\mp\eta)\zeta^*(z,\pm\eta)\,,
\\  \label{eq:lambda-asymp}
\lambda(\eta)=i\lambda_1\eta -\lambda_2\eta^2+O(\eta^3)\,,
\\  \label{eq:zeta-asymp}
\zeta(\cdot,\eta)=\zeta_1+i\lambda_1\eta \zeta_2+O(\eta^2)
\quad\text{in $H^k_\a(\R)\times H^{k-1}_\a(\R)$,}
\\  \label{eq:zeta*-asymp}
\zeta^*(\cdot,\eta)=\zeta_2^*-i\lambda_1\eta \zeta_1^*+O(\eta^2)
\quad\text{in $H^k_{-\a}(\R)\times H^{k-1}_{-\a}(\R)$,}
\\  \label{eq:zeta-parity}
\overline{\lambda(\eta)}=\lambda(-\eta)\,,\enskip
\overline{\zeta(z,\eta)}=\zeta(z,-\eta)\,,\enskip \overline{\zeta^*(z,\eta)}=\zeta^*(z,-\eta)
\quad\text{for $\eta\in[-\eta_0,\eta_0]$ and $z\in\R$,}
\end{gather}
where $\lambda_1$ and $\lambda_2$ are positive constants,
$A_0=1-a\pd_z^2$, $B_0=1-b\pd_z^2$ and
\begin{gather*}
\zeta_1=
\begin{pmatrix}  q_c\\ r_c'\end{pmatrix}\,,
\quad
\zeta_2=
\begin{pmatrix} \int_z^\infty \pd_cq_c \\ -\pd_cr_c \end{pmatrix}\,,\\
\zeta_1^*=c
\begin{pmatrix}
-B_0\pd_cr_c-2q_c\pd_cq_c-q_c'\int_{-\infty}^z \pd_cq_c
\\ B_0\int_{-\infty}^z \pd_cq_c
\end{pmatrix}
\,,\quad
\zeta_2^*=\begin{pmatrix}  A_0q_c'\\ -B_0r_c\end{pmatrix}\,.
\end{gather*}
\end{theorem}
\begin{remark}
We remark that $\mL(0)$ is a linearized operator of the $1$-dimensional
Benney-Luke equation around $\varphi_c(z)$ and
$\zeta_1$ and $\zeta_2$ belong to the generalized kernel of $\mL(0)$.
More precisely,
  \begin{gather*}
\mL(0)\zeta_1=0\,,\quad
\mL(0)\zeta_2=\zeta_1\,,\quad \mL(0)^*\zeta_1^*=\zeta_2^*\,,\quad
\mL(0)^*\zeta_2^*=0\,,\\
\ker_g(\mL(0))=\spann\{\zeta_1,\zeta_2\}\,,
\quad \ker_g(\mL(0))=\spann\{\zeta_1^*,\zeta_2^*\}\,.
  \end{gather*}
The eigenvalue $\lambda=0$ for $\mL(0)$
splits into two stable eigenvalues $\lambda(\pm\eta)$ for $\mL(\eta)$
with $\eta\ne0$.
\par
In the exponentially weighted space $L^2_\a(\R)$,
the value $\lambda=0$ is an isolated eigenvalue of $\mL(0)$ and there
exists a $\beta>0$ such that 
$$\sigma(\mL(0))\setminus\{0\}\subset \{\lambda\in\C\mid \Re\lambda \le -\beta\}$$ 
provided $c>1$ and $c$ is sufficiently close to $1$.
See Lemma~2.1, Theorem~2.3 and Appendix~B in \cite{MPQ13}.
\end{remark}
\begin{remark}
  We expect that $\zeta_k(\cdot,\eta)$ and $\zeta_k^+(\cdot,\eta)$ ($k=1$, $2$) do not belong to $L^2(\R)$ as is the same with
continuous resonant modes for the KP-II equation. This is a reason why we study spectral stability of $\mL$ in
the exponentially weighted space $X$.
\end{remark}
We will prove Theorem~\ref{thm:1} by using the Lyapunov Schmidt method
in Section~\ref{sec:bifurcation}.
\par
Let $\mathcal{P}(\eta_0)$ be the spectral projection onto the subspace
corresponding to the continuous eigenvalues
$\{\lambda(\eta)\}_{-\eta_0\le \eta \le\eta_0}$ and 
$\mathcal{Q}(\eta_0)=I-\mathcal{P}(\eta_0)$.
Let $Z=\mathcal{Q}(\eta_0)(H^1_\a(\R^2)\times L^2_\a(\R^2))$.
If $\mL$ is spectrally stable, then $e^{t\mL}|_{Z}$ is exponentially stable.
\begin{corollary}
  \label{thm:linear-stability}
Let $0<a<b$, $c>1$ and $\a\in(0,\a_c)$. Consider the operator $\mL$
in the space $X=H^1_\a(\R^2)\times L^2_\a(\R^2)$.
Assume that there exist positive
constants $\beta$ and $\eta_0$ such that
\begin{equation}
  \label{ass:H}
\sigma(\mL|_Z)  \subset \{\lambda\mid \Re\lambda\le -\beta\}\,,
\tag{H}
\end{equation}
where $\mL|_Z$ is the restriction of the operator $\mL$ on $Z$.
Then for any $\beta'<\beta$, there exists a positive constant $C$ such that
\begin{equation}
  \label{eq:decay}
\|e^{t\mL}\mathcal{Q}(\eta_0)\|_{B(X)}
\le Ce^{-\beta' t}\quad\text{for any  $t\ge0$.}  
\end{equation}
\end{corollary}
The semigroup estimate \eqref{eq:decay} follows from the assumption
\eqref{ass:H} and the Geahart-Pr\"{u}ss theorem \cite{Ge78,Pr85} which
tells us that the boundedness of $C^0$-semigroup in a Hilbert space is
equivalent to the uniform boundedness of the resolvent operator on the
right half plane.  See also \cite{Herbst,How}.
\par
Time evolution of the continuous eigenmodes
$\{e^{t\lambda(\eta)}g(z,\eta)\}_{-\eta_0\le \eta\le \eta_0}$ can be considered
as a linear approximation of non-uniform phase shifts of modulating
line solitary waves.
For the KP-II equation, modulations of the local amplitude and the angle of
the local phase shift of a line soliton are described by a system of Burgers'
equations  (see \cite[Theorems~1.4 and 1.5]{Miz15}).
In this paper, we find the first order asymptotics of solutions for the
linearized equation \eqref{eq:linear} is described by a wave equation
with a diffraction term and it tends to a constant multiple of the
$x$-derivative of the line solitary wave as $t\to\infty$.
\begin{theorem}
  \label{thm:linear-dynamics}
Let $0<a<b$, $c>1$, $\a$ be as in Theorem~\ref{thm:linear-stability} and 
$(\Phi_0,\Psi_0)\in H^2_\a(\R^2)\times H^1_\a(\R^2)$. 
Assume \eqref{ass:H}. Then a solution of \eqref{eq:linear} with
$(\Phi(0),\pd_t\Phi(0))=(\Phi_0,\Psi_0)$ satisfies
\begin{equation*}
\left\|\begin{pmatrix}\pd_z\Phi(t,z,y) \\ \pd_t\Phi(t,z,y) \end{pmatrix}
-(H_t*W_t*f)(y)
\begin{pmatrix}  q_c'(z)\\ r_c'(z)\end{pmatrix}
\right\|_{L^2_\a(\R_z)L^\infty(\R_y)}
=O(t^{-1/4})\quad\text{as $t\to\infty$,}
\end{equation*}
where $f(y)=\left\la cB_0\Psi_0-A_0\pd_z\Phi_0, q_c\right\ra$,
$H_t(y)=(4\pi\lambda_2t)^{-1/2}e^{-y^2/4\lambda_2t}$,
$\kappa_1=\frac{\lambda_1}{2}\frac{d}{dc}E(q_c,r_c)$ and
$W_t(y)=(2\kappa_1)^{-1}$ for $y\in [-\lambda_1 t, \lambda_1 t]$
and $W_t(y)= 0$ otherwise.
\end{theorem}
We remark that if $f(y)$ is well localized and $\int_\R f(y)\,dy\ne0$,
then $H_t*W_t*f(y)\simeq (2\kappa_1)^{-1}\int_\R f(y)\,dy$ on any
compact intervals in $y$ as $t\to\infty$.  The first order asymptotics
of solutions to \eqref{eq:linear} suggests that the local phase shift
of line solitary waves propagate mostly at constant speed toward $y=\pm\infty$.
\par
If $c$ is close to $1$, then the assumption \eqref{ass:H} is valid
and the spectrum of $\mL$ is similar to that of the linearized KP-II operator
around a line soliton.
To utilize the spectral property of the linearized operators of
the KP-II equation around $1$-line solitons,
we introduce the scaled parameters and variables
\begin{equation}
  \label{eq:KP2-scale1}
\lambda=\eps^3\Lambda\,, \quad c^2=1+\eps^2\,, 
\quad \hat{z}=\eps z\,,\quad \hat{y}=\eps^2 y\,,\quad
\xi=\eps\hxi\,,\quad \eta=\eps^2\heta\,,  
\end{equation}
and translate the solitary wave profile $q_c(x)$ as
\begin{equation}
  \label{eq:KP2-scale2}
q_c(z)=\eps^2\theta_{\eps}(\hat{z})\,,\quad
\theta_{\eps}(\hat{z})=\frac{1}{c}
\sech^2\left(\frac{\hat{\a}_\eps\hat{z}}{2}\right)\,,
\quad \hat{\a}_\eps=\frac{1}{\sqrt{bc^2-a}}\,.
\end{equation}
Let
\begin{gather*}
\hat{\a}_0=(b-a)^{-1/2}\,, \quad
\theta_0(\hat{z})=\sech^2(\frac{\hat{\a}_0}{2}\hat{z})\,,\\
\mL_{KP}=-\frac{1}{2}\{(b-a)\pd_{\hat{z}}^3-\pd_{\hat{z}}
+\pd_{\hat{z}}^{-1}\pd_{\hat{y}}^2 +3\pd_{\hat{z}}(\theta_0\cdot)\}\,.   
\end{gather*}
We remark that the operator $\mL_{KP}$  
is the linearization of the KP-II equation 
\begin{equation}
  \label{eq:KPII}
2\pd_tu+(b-a)\pd_x^3u+\pd_x^{-1}\pd_y^2u+\frac{3}{2}\pd_x(u^2)=0
\end{equation}
around its line soliton solution $\theta_0(x-t)$.
The linearized operator $\mL_{KP}$ has continuous eigenvalues 
$\lambda_{KP}(\eta)=\frac{i\eta}{\sqrt{3}}\sqrt{1+i\gamma_1\eta}$
which has to do with dynamics of modulating
line solitons (see \cite{Burtsev85,Miz15} and Section~\ref{subsec:KP}).
\par
In the low frequency regime,
we can deduce the eigenvalue problem
\begin{equation}
  \label{eq:BLevp}
  \mathcal{L}  \begin{pmatrix}    u \\ v  \end{pmatrix}
=\lambda \begin{pmatrix}    u \\ v  \end{pmatrix}
\end{equation}
to $\mathcal{L}_{KP}\pd_{\hat{z}}u= \Lambda \pd_{\hat{z}}u$
provided $\eps$ is sufficiently small.
More precisely, we have the following.
\begin{theorem} \label{thm:spectrum}
Let $c=\sqrt{1+\eps^2}$, $\a=\hat{\a}\eps$ and $\hat{\a} \in (0,\hat{\a}_0/2)$.
Then there exist positive constants
$\eps_0$, $\eta_0$, $\hat{\beta}$ and a smooth function
$\lambda_\eps(\eta)$ such that if $\eps\in(0,\eps_0)$, then 
\begin{gather}
  \label{eq:spec-L}
\sigma(\mL)\setminus \{\lambda_\eps(\eta)\mid
\eta\in[-\eps^2\eta_0,\eps^2\eta_0]\}
\subset \{\lambda\in \C\mid \Re\lambda \le -\hat{\beta}\eps^3\}\,,
\\  \label{eq:asymp-ev}
\lim_{\eps\downarrow0}|\eps^{-3}\lambda_\eps(\eps^2\eta)-\lambda_{KP}(\eta)|
=O(\eta^3)
\quad\text{for $\eta\in[-\eta_0,\eta_0]$,}
\\ \label{eq:decay-estimate}
\|e^{t\mL}\mathcal{Q}(\eps^2\eta_0)\|_{B(X)}
\le K e^{-\hat{\beta}\eps^3 t}\quad\text{for any  $t\ge0$,}  
\end{gather}
where $K$ is a constant that does not depend on $t$.
\end{theorem}
\bigskip

\section{Resonant modes of the linearized operator}
\label{sec:resonant}
In this section, we will prove the existence of resonant continuous
eigenvalues of $\mL$ near $\lambda=0$ and show that 
the resonant eigenvalues and resonant eigenmodes for $\mL$ are similar to 
those for the linearized KP-II operator $\mL_{KP}$
provided line solitary waves are small.

\subsection{Spectral stability in the KP-II scaling regime}
\label{subsec:KP}
First, we recall some spectral properties of the linearized KP-II
equation around $1$-line solitons.  Let us consider the eigenvalue
problem of the linearized operator of \eqref{eq:KPII} around
$\theta_0$.  Let
\begin{gather*}
\mL_{KP,0}=-\frac{1}{2}\{(b-a)\pd_z^3-\pd_z+\pd_z^{-1}\pd_y^2\}\,,
\quad
\mL_{KP}=\mL_{KP,0}-\frac{3}{2}\pd_z(\theta_0\cdot)\,,\\
\mL_{KP}(\eta)=-\frac{1}{2}\pd_z\{(b-a)\pd_z^2-1+3\theta_0\}
+\frac{\eta^2}{2}\pd_z^{-1}\,.
\end{gather*}
Formally, we have $\mL_{KP}(u(z)e^{iy\eta})=e^{iy\eta}(\mL_{KP}(\eta)u)(z)$.
The operator $\mL_{KP,0}$ is spectrally stable
in exponentially weighted spaces.
\begin{lemma}
\label{lem:free-KP}
Let $\hat{\a}\in (0,\hat{\a}_0)$ and 
$\hat{\beta}_0=\frac{\hat{\a}}{2}\{1-(b-a)\hat{\a}^2\}$. Then
\begin{equation}
  \label{eq:KP-res1}
 \|(\Lambda-\mL_{KP,0})^{-1}\|_{B(L^2_{\hat{\a}}(\R^2))}
\le (\Re\Lambda+\hat{\beta}_0)^{-1}
\quad \text{for $\Lambda$ satisfying $\Re\Lambda> -\hat{\beta}_0$.}  
\end{equation}
Moreover, there exists a positive constant $C$ such that
if  $\Re\Lambda> -\hat{\beta}_0$,
\begin{gather}
\label{eq:KP-res2}
 \|\pd_z^j(\Lambda-\mL_{KP,0})^{-1}\|_{B(L^2_{\hat{\a}}(\R^2))}
\le C\left(\Re\Lambda+\frac{\hat{\beta}_0}{2}\right)^{-1+\frac{j}{2}}
\quad \text{for $j=1,2$,}
\\
  \label{eq:KP-sector}
 \|(\Lambda-\mL_{KP,0})^{-1}\|_{B(L^2_{\hat{\a}}(\R^2))}
\le C\left|\Lambda+\frac{\hat{\beta}_0}{2}\right|^{-2/3}\,.
\end{gather}
\end{lemma}
\begin{proof}
By the Plancherel theorem,
\begin{equation}
  \label{eq:wPl}
\|g\|_{L^2_\a(\R^2)}^2=\int_{\R^2}e^{2\a x}|g(x,y)|^2\,dxdy
=\int_{\R^2}|\hat{g}(\xi+i\a,\eta)|^2\,d\xi d\eta  
\end{equation}
for any $g\in C_0(\R^2)$ and
an operator $m(D):=\frac{1}{2\pi}\check{m}*f$
is bounded on $L^2_\a(\R^2)$ if and only if
\begin{equation}
  \label{eq:fm0}
 \|m(D)\|_{B(L^2_\a(\R^2))}=\sup_{(\xi,\eta)\in\R^2}
\left|m(\xi+i\a,\eta)\right|
<\infty\,.
\end{equation}
\par
 Suppose $f\in L^2_\a(\R^2)$ and that $u$ is a solution of
$$ (\Lambda-\mL_{KP,0})u=f\,.$$
Then 
$$ \hat{u}(\xi,\eta)=
\frac{\hat{f}(\xi,\eta)}{\Lambda-\mL_{KP,0}(\xi,\eta)}\,,$$
where $\mL_{KP,0}(\xi,\eta)= \frac{i}{2}\{(b-a)\xi^3+\xi-\xi^{-1}\eta^2\}$.
Since
\begin{align*}
\Re \mL_{KP,0}(\xi+i\hat{\a},\eta)=& 
-\frac{1}{2}  \left\{
3(b-a)\hat{\a}\xi^2+2\hat{\beta}_0+\frac{\hat{\a}\eta^2}{\xi^2+\hat{\a}^2}
 \right\}\ge -\hat{\beta}_0\,,
\end{align*}
it follows from \eqref{eq:fm0} that for $j=0,1,2$ and $\Lambda$ with
$\Re\Lambda>-\hat{\beta}_0$,
\begin{align*}
\|\pd_z^j(\Lambda-\mL_{KP,0})^{-1}\|_{B(L^2_{\hat{\a}}(\R^2)} =&
\sup_{(\xi,\eta)\in\R^2}
\frac{|\xi+i\hat{\a}|^j}{|\Lambda-\mL_{KP,0}(\xi+i\hat{\a},\eta)|} \,.
\end{align*}
Thus we have \eqref{eq:KP-res1} and \eqref{eq:KP-res2}.
Moreover, we have \eqref{eq:KP-sector} because
$|\Im \mL_{KP,0}(\xi+i\hat{\a},\eta)|\lesssim \{-\Re \mL_{KP,0}(\xi+i\hat{\a},\eta)\}^{3/2}$.
\end{proof}
Let $\gamma_1=4\sqrt{(b-a)/3}$, $\hat{x}=\frac{\hat{\a}_0}{2}x$
and 
\begin{gather*}
\lambda_{KP}(\eta)=\frac{i\eta}{\sqrt{3}}\sqrt{1+i\gamma_1\eta}\,,\\
g_0(x,\eta)=\frac{2(b-a)}{\gamma_1\sqrt{1+i\gamma_1\eta}}
\pd_x^2\left(e^{-\sqrt{1+i\gamma_1\eta}\hat{x}}\sech\hat{x}\right)\,,\\
g_0^*(x,\eta)=\frac{i}{\eta}
\pd_x\left(e^{\sqrt{1-i\gamma_1\eta}\hat{x}}\sech\hat{x}\right)\,.
\end{gather*}
Using Lemma~2.1 in \cite{Miz15} and the change of variable
\begin{equation*}
x\mapsto \frac{\hat{\a}_0}{2}x\,,\quad y\mapsto \frac{1}{\gamma_1}y\,,\quad
\eta \mapsto \gamma_1\eta\,,
\end{equation*}
we have for $\eta\in\R\setminus\{0\}$,
\begin{gather*}
 \mL_{KP}(\eta)g_{0}(x,\pm\eta)=\lambda_{KP}(\pm\eta)g_{0}(x,\pm\eta)\,,\\
 \mL_{KP}(\eta)^*g_{0}^*(x,\pm\eta)=\lambda_{KP}(\mp\eta)g_{0}^*(x,\pm\eta)\,,\\
\int_\R g_{0}(x,\eta)\overline{g_{0}^*(x,\eta)}\,dx=1\,,\quad
\int_\R g_{0}(x,\eta)\overline{g_{0}^*(x,-\eta)}\,dx=0\,.
\end{gather*}
To  resolve the singularity of $g_{0}(x,\eta)$ and the degeneracy of
$g_0^*(x,\eta)$ at $\eta=0$, we decompose them into
their real parts and imaginary parts. Let
\begin{gather*}
g_{0,1}(x,\eta)=g_0(x,\eta)+g_0(x,-\eta)\,,\quad
g_{0,2}(x,\eta)=\frac{1}{i\eta}\{g_0(x,\eta)-g_0(x,-\eta)\}\,,\\
g_{0,1}^*(x,\eta)=\frac12\{g_0^*(x,\eta)+g_0^*(x,-\eta)\}\,,\quad
g_{0,2}^*(x,\eta)=\frac{\eta}{2i}\{g_0^*(x,\eta)-g_0^*(x,-\eta)\}\,.
\end{gather*}
Then 
$$\int_\R g_{0,j}(x,\eta)\overline{g_{0,k}^*(x,\eta)}\,dx=\delta_{jk}
\quad\text{for $j$, $k=1$, $2$.}$$
Moreover, we see that
$g_{0,k}(x,\eta)$ and $g_{0,k}^*(x,\eta)$ are even in $\eta$ and that 
for $k=1$, $2$ and $\hat{\a}\in(0,\hat{\a}_0)$, 
\begin{gather}
\label{eq:g0eta-0}
\|g_{0,k}(\cdot,\eta)-g_{0,k}(\cdot,0)\|_{L^2_{\hat{\a}}}+
\|g_{0,k}^*(\cdot,\eta)-g_{0,k}^*(\cdot,0)\|_{L^2_{-\hat{\a}}}=O(\eta^2)\,,
\\
g_{0,1}(x,0)=-\frac{\sqrt{3}}{2}\theta_0'(x)\,, \quad
g_{0,2}(x,0)=\theta_0(x)+\left(\frac{x}{2}+\hat{\a}_0^{-1}\right)\theta_0'(x)\,,
\\
g_{0,1}^*(x,0)=\frac{\hat{\a}_0}{2\sqrt{3}}
\int_{-\infty}^x (x_1\theta_0'(x_1)+2\theta_0(x_1))\,dx_1\,,
\quad g_{0,2}^*(x,0)=\frac{\hat{\a}_0}{2}\theta_0(x)\,.
\end{gather}

Let $\mathcal{P}_{KP}(\eta_0)$ be the spectral projection to resonant modes
$\{g_0(x,\pm\eta)e^{iy\eta}\}_{-\eta_0\le \eta\le \eta_0}$ defined by
\begin{gather*}
\mathcal{P}_{KP}(\eta_0)f(x,y)=\frac{1}{\sqrt{2\pi}}\sum_{k=1,\,2}
\int_{-\eta_0}^{\eta_0}a_{0,k}(\eta)g_{0,k}(x,\eta)e^{iy\eta}\,d\eta\,,\\
 a_{0,k}(\eta)= \int_\R (\mF_yf)(x,\eta)\cdot g_{0,k}^*(x,\eta)\,dx\,,
\end{gather*}
and let $\mathcal{Q}_{KP}(\eta_0)=I-\mathcal{P}_{KP}(\eta_0)$.
By Lemma~3.1 in \cite{Miz15}, the operator $\mathcal{P}_{KP}(\eta_0)$ and $\mathcal{Q}_{KP}(\eta_0)$ are
bounded on $L^2_{\hat{\a}}(\R^2)$ for $\hat{\a}\in (0,\hat{\a}_0)$.
Moreover, we have the following.

\begin{proposition}
  \label{prop:KPII}
Let $\hat{\a}\in (0,\hat{\a}_0)$ and $\eta_*$ be a positive number satisfying 
$\frac{\hat{\a}}{2}(\Re\sqrt{1+i\gamma\eta_*}-1)=\hat{\a}$.
For any $\eta_0\in(0,\eta_*)$, there exists a positive number $b$ such that
$$\sup_{\Re\Lambda\ge -b}\|(\Lambda-\mL_{KP})^{-1}\mathcal{Q}_{KP}(\eta_0)\|_{B(L^2_{\hat{\a}}(\R^2))}<\infty\,.$$
\end{proposition}

\begin{proof}
By Proposition~3.2 in \cite{Miz15}, there exist positive constants
$b_1$ and $C$ such that
$$ \|e^{t\mL_{KP}}\mathcal{Q}_{KP}(\eta_0)\|_{B(L^2_{\hat{\a}}(\R^2))}
\le Ce^{-b_1t}\,.$$
If $\Re\Lambda\ge -b>-b_1$, then
\begin{align*}
  \|(\Lambda-\mL_{KP})^{-1}\mathcal{Q}_{KP}(\eta_0)\|_{B(L^2_{\hat{\a}})}
\le & \int_0^\infty 
\|e^{-\Lambda t}e^{t\mL_{KP}}\mathcal{Q}_{KP}(\eta_0)\|_{B(L^2_{\hat{\a}}(\R^2))}\,dt
 \lesssim  \frac{1}{b_1-b}\,.
\end{align*}
\end{proof}

\subsection{Resonant modes}
\label{subsec:resonant}
In this subsection, we will prove the existence of continuous resonant modes
of $\mL$ near $\lambda=0$ by using the Lyapunov Schmidt method.
Let 
\begin{gather*}
A(\eta)=1+a\eta^2-a\pd_z^2\,,\quad  B(\eta)=1+b\eta^2-b\pd_z^2\,,\\
\mL_0(\eta)=
\begin{pmatrix}
c\pd_z & 1 \\ B(\eta)^{-1}A(\eta)(\pd_z^2-\eta^2) & c\pd_z  
\end{pmatrix}\,,\\
\mL(\eta)=\mL_0(\eta)+V(\eta)\,,\quad
V(\eta)=-B(\eta)^{-1}
\begin{pmatrix}
  0 & 0 \\ v_{1,c}(\eta) & v_{2,c}(\eta)
\end{pmatrix}\,,
\\
v_{1,c}(\eta)=2r_c'\pd_z+r_c(\pd_z^2-\eta^2)\,,\quad
v_{2,c}(\eta)=2q_c\pd_z+q_c'\,.
\end{gather*}
If $e^{iy\eta}(u_1(z),u_2(z))$ is a solution of \eqref{eq:BLevp}, then
\begin{equation}
  \label{eq:evBL-eta}
 \mL(\eta)\begin{pmatrix}u_1 \\ u_2\end{pmatrix}
=\lambda \begin{pmatrix} u_1 \\ u_2 \end{pmatrix}
\end{equation}
or equivalently,
\begin{gather}
\label{eq:u1}
 \{A(\eta)(\pd_z^2-\eta^2)-(\lambda-c\pd_z)^2B(\eta)\}u_1
-v_{1,c}(\eta)u_1-v_{2,c}(\eta)(\lambda-c\pd_z)u_1=0\,,
\\ \label{eq:u2}
 u_2= (\lambda-c\pd_z)u_1\,.
\end{gather}
We will find solutions of \eqref{eq:evBL-eta} in
$H^1_\a(\R)\times L^2_\a(\R)$ for small $\eta$.  Using the change
of variables \eqref{eq:KP2-scale1} and \eqref{eq:KP2-scale2} and
dropping the hats in the resulting equation, we have
\begin{equation}
  \label{eq:ev-eta}
F(U,\Lambda,\eps,\eta):= 2L_\eps(\eta)U-\Lambda T_1(\eps,\eta)U
+\eps^2\Lambda^2B_\eps(\eta)\pd_z^{-1}U=0\,,
\end{equation}
where $U(z)=\pd_zu_1(z/\eps)$ and 
\begin{gather*}
L_\eps(\eta)= -\frac12\pd_z\{(bc^2-a)\pd_z^2-1+3c\theta_\eps\}
+\frac{\eta^2}{2}T_2(\eps,\eta)\,,\\
T_1(\eps,\eta)=2cB_\eps(\eta)-\eps^2(2\theta_\eps+\theta_\eps'\pd_z^{-1})\,,
\quad  T_2(\eps,\eta)=\{A_\eps(\eta)+\eps^2(bc^2-a)\pd_z^2+c\eps^2\theta_\eps\}\pd_z^{-1}\,,\\
A_\eps(\eta)=1+a\eps^2(\eps^2\eta^2-\pd_z^2)\,,\quad
B_\eps(\eta)=1+b\eps^2(\eps^2\eta^2-\pd_z^2)\,.
\end{gather*}
Let $L_\eps(\eta)$ be an operator on $L^2_{\hat{\a}}(\R)$  with 
$D(L_\eps)=H^3_{\hat{\a}}(\R)$ for an $\hat{\a}\in(0,\hat{\a}_\eps)$
and
$$(\pd_z^{-1}f)(z)=-\int_z^\infty f(z_1)\,dz_1
\quad\text{for $f\in L^2_{\hat{\a}}(\R)$.}$$
We remark that $F(U,\Lambda,0,\eta)=2\mL_{KP}(\eta)U-2\Lambda U$
and the translated eigenvalue problem \eqref{eq:ev-eta} is
similar to the eigenvalue problem of the  KP-II equation
provided $\eps$ is sufficiently small.
For small $\eta\ne0$, \eqref{eq:evBL-eta} has two eigenvalues 
in the vicinity of $0$. 
\par
First, we will find an approximate solution of \eqref{eq:ev-eta}.
Let $U(\eta)=U_0+\eta U_1+\eta^2U_2+O(\eta^3)$, 
$\Lambda(\eta)=i\Lambda_{1,\eps}^0\eta-\Lambda_{2,\eps}^0\eta^2+O(\eta^3)$
and formally equate the powers of $\eta$ in \eqref{eq:ev-eta}. Then
\begin{gather}
  \label{eq:U0}
  L_\eps(0)U_0=0\,,\\
\label{eq:U1}
L_\eps(0)U_1=\frac{i}{2}\Lambda_{1,\eps}^0T_1(\eps,0)U_0\,,\\
\label{eq:U2}
2L_\eps(0)U_2=-\left\{T_2(\eps,0)+\Lambda_{2,\eps}^0T_1(\eps,0)
-\eps^2(\Lambda_{1,\eps}^0)^2B_\eps(0)\pd_z^{-1} \right\}U_0
+i\Lambda_{1,\eps}^0T_1(\eps,0)U_1 \,.
\end{gather}
Let $\theta_{1,\eps}(z)=\pd_cq_c\left(\frac{z}{\eps}\right)$,
$\theta_{\eps,d}(z)=d\theta_\eps(\sqrt{d}z)$
and $\tilde{\theta}_{1,\eps}=2\pd_d\theta_{\eps,d}|_{d=1}$.
By \eqref{eq:qc},
\begin{equation}
  \label{eq:qc-new}
  (bc^2-a)\theta_\eps''-\theta_\eps +\frac{3c}{2}\theta_\eps^2=0\,.
\end{equation}
It follows from \cite[Proposition~2.8]{PW94} that
\begin{gather}
\label{eq:kerLe1}
L_\eps(0)\theta_\eps'=0\,,\quad L_\eps(0)\tilde{\theta}_{1,\eps}=-\theta_\eps'\,,
\\ \label{eq:kerLe2}
L_\eps(0)^*\theta_\eps=0\,, \quad
L_\eps(0)^*\int_{-\infty}^z \tilde{\theta}_{1,\eps}(z_1)\,dz_1=\theta_\eps\,,
\\ \label{eq:gkerLe}
\ker_g(L_\eps(0))=\spann\{\theta_\eps',\tilde{\theta}_{1,\eps}\}\,,
\quad
\ker_g(L_\eps(0)^*)=\spann\left\{\theta_\eps, \int_{-\infty}^z\tilde{\theta}_{1,\eps}\right\}\,,
\end{gather}
where $\ker_g(A)$ denotes the generalized kernel of the operator $A$.
Differentiating \eqref{eq:qc} with respect to $c$ and $x$,
using the change of variables \eqref{eq:KP2-scale1}, \eqref{eq:KP2-scale2}
and dropping the hats in the resulting equation,
we have
\begin{equation}
  \label{eq:kerLe}
 L_\eps(0)\theta_{1,\eps}=-\frac12T_1(\eps,0)\theta_\eps'\,,
\quad 
L_\eps(0)^*\int_{-\infty}^z\theta_{1,\eps}=\frac12 \pd_z^{-1}T_1(\eps,0)\theta_\eps'\,.
\end{equation}
Combining \eqref{eq:U0}, \eqref{eq:U1}, \eqref{eq:kerLe1}, \eqref{eq:kerLe}
and the fact that  $\ker(L_\eps(0))=\spann\{\theta_\eps'\}$, we have
\begin{equation}
  \label{eq:U0-U1}
U_0=\theta_\eps'\,, \quad
U_1=-i\Lambda_{1,\eps}^0\theta_{1,\eps}+C_1\theta_\eps'
\end{equation}
up to the constant multiplicity, where $C_1$ is an arbitrary constant.
\par

Next, we will determine $\Lambda_{1,\eps}^0$.
Multiplying \eqref{eq:U2} by $\theta_\eps$ and substituting \eqref{eq:U0-U1}
into the resulting equation, we have from \eqref{eq:kerLe2}
\begin{align*}
& \left\la T_2(\eps,0)\theta_\eps' +\Lambda_{2,\eps}^0T_1(\eps,0)\theta_\eps'
-\eps^2(\Lambda_{1,\eps}^0)^2B_\eps(0)\theta_\eps, \theta_\eps \right\ra
+i\Lambda_{1,\eps}^0 \left \la T_1(\eps,0)(i\Lambda_{1,\eps}^0 \theta_{1,\eps}-C_1\theta_\eps'),
 \theta_\eps\right\ra
\\=& -2\la U_2,L_\eps(0)^*\theta_\eps\ra=0\,.
\end{align*}
Since $\theta_\eps$ is even and $\theta_\eps'$ and $T_1(\eps,0)\theta_\eps'$ are odd,
we have $\la T_1(\eps,0)\theta_\eps', \theta_\eps \ra=\la \theta_\eps',\theta_\eps\ra=0$
and
\begin{gather}
  \label{eq:Lambda1}
(\Lambda_{1,\eps}^0)^2=\frac{f_1(\eps)}{f_2(\eps)}\,,
\\ \notag
f_1(\eps)= \la T_2(\eps,0)\theta_\eps', \theta_\eps\ra\,,
\quad
f_2(\eps)=\la T_1(\eps,0)\theta_{1,\eps} +\eps^2B_\eps(0)\theta_\eps, \theta_\eps\ra\,.
\end{gather}
By \eqref{eq:qc-new} and the fact that
$(T_1(\eps,0)\pd_z)^*\theta_\eps=-T_1(\eps,0)\theta_\eps'=
-c^{-1}\pd_z\{(A_\eps(0)+c^2B_\eps(0)\}\theta_\eps$,
we have
\begin{gather*}
f_1(\eps)=\frac{1+2c^2}{3}\la \theta_\eps,\theta_\eps\ra
+\frac{\eps^2}{3}(4a-bc^2)\la \theta_\eps',\theta_\eps'\ra\,,
\\
f_2(\eps)=\frac{1}{c}\la \{A_\eps(0)+c^2B_\eps(0)\}\theta_\eps,\theta_{1,\eps}\ra
+\eps^2\la B_\eps(0)\theta_\eps,\theta_\eps\ra\,.
\end{gather*}
Since
\begin{equation}
  \label{eq:theta-approx}
\|\theta_\eps-\theta_0\|_{H^k_\a(\R)\cap H^k_{-\a}(\R)}
+\|\theta_{1,\eps}-2\theta_0-z\theta_0'\|_{H^k(\R)\cap H^k_{-\a}(\R)}=O(\eps^2)
\quad\text{for any $k\ge0$,}
\end{equation}
we have $\Lambda_{1,\eps}^0=\pm \frac{1}{\sqrt{3}}+O(\eps^2)$.

Now we will use the Lyapunov Schmidt method to prove existence of 
solutions to \eqref{eq:ev-eta} satisfying
$(U(\eta),\Lambda(\eta))\simeq (\theta_\eps'-i\eta\Lambda_{1,\eps}^0\theta_{1,\eps},
i\eta\Lambda_{1,\eps}^0)$.
\begin{lemma}
  \label{lem:construction}
Let $\hat{\a}\in(0,\hat{\a}_0/2)$.
There exist positive constants $\eps_0$ and $\eta_0$ such that  \eqref{eq:ev-eta}
has a solution $(U_\eps(\eta),\Lambda_\eps(\eta))$ satisfying
for any $\eta\in [-\eta_0,\eta_0]$ and $k\ge0$,
\begin{gather}
\label{eq:f-Ue}
\sup_{\eps\in(0,\eps_0)} \left\| U_\eps(\eta)-\theta_\eps'+\Lambda_\eps(\eta)\theta_{1,\eps}
\right\|_{H^k_{\hat{\a}}(\R)}=O(\eta^2)\,,
\\
\label{eq:f-La}
\sup_{\eps\in(0,\eps_0)} \left|\Lambda_\eps(\eta)
-i\Lambda_{1,\eps}^0\eta+\Lambda_{2,\eps}^0\eta^2\right|=O(\eta^3)\,,
\end{gather}
where $\Lambda_{1,\eps}^0$ and $\Lambda_{2,\eps}^0$ are constants satisfying
$\Lambda_{1,\eps}^0=\frac{1}{\sqrt{3}}+O(\eps^2)$ and $\Lambda_{2,\eps}^0=
\frac{2}{3\hat{\a}_0}+O(\eps^2)$.
Moreover, 
\begin{equation}
  \label{eq:f-cc}
\overline{U_\eps(\eta)}=U_\eps(-\eta)\,,\quad
\overline{\Lambda_\eps(\eta)}=\Lambda_\eps(-\eta)
\quad\text{for $\eta\in[-\eta_0,\eta_0]$,}  
\end{equation}
and the mapping 
$[-\eta_0,\eta_0]\ni \eta\mapsto (U_\eps(\eta),\Lambda_\eps(\eta))\in H^k_{\hat{\a}}(\R)\times \R$
is smooth for any $k\ge0$.
\end{lemma}
\begin{proof}
 Let $\Lambda(\eta)=i\eta\Lambda_1(\eta)$ and
 \begin{equation}
   \label{eq:defU}
U(\eta)=\theta_\eps'
-\{i\eta\Lambda_1(\eta)-\eta^2\gamma(\eta)\}\theta_{1,\eps}
+\eta^2\widetilde{U}(\eta)\,, \quad
\widetilde{U}(\eta)\perp 
\theta_\eps\,,\; \int_{-\infty}^z\tilde{\theta}_{1,\eps}(z_1)\,dz_1\,.   
 \end{equation}
Then \eqref{eq:ev-eta} is translated into
\begin{equation}
\label{eq:ev-eta2}
2\widetilde{L}_\eps(\eta)\widetilde{U}+G_1(\gamma,\Lambda_1,\eps,\eta)
-i\eta G_2(\gamma,\Lambda_1,\eps,\eta)=0\,,
\end{equation}
where
\begin{align*}
& \widetilde{L}_\eps(\eta)= 
L_\eps(\eta)-\frac{i}{2}\eta\Lambda_1(\eta)T_1(\eps,\eta)-\frac{\eps^2}{2}\eta^2\Lambda_1(\eta)^2B_\eps(\eta)\pd_z^{-1}\,,
\\ &
G_1(\gamma,\Lambda_1,\eps,\eta)=T_2(\eps,\eta)\theta_\eps'+2\gamma L_\eps(\eta)\theta_{1,\eps}
-\Lambda_1^2\{T_1(\eps,\eta)\theta_{1,\eps}+\eps^2B_\eps(\eta)\theta_\eps\}\,,
\\ &
G_2(\gamma,\Lambda_1,\eps,\eta)=2bc\eps^4\Lambda_1\theta_\eps'
+\Lambda_1\{T_2(\eps,\eta)+\gamma T_1(\eps,\eta)\}\theta_{1,\eps}
\\ & +\eps^2\Lambda_1^2(\Lambda_1+i\gamma\eta)
B_\eps(\eta)\int_z^\infty \theta_{1,\eps}(z_1)\,dz_1\,.
\end{align*}
Here we use \eqref{eq:kerLe} and the fact that
$\{T_1(\eps,\eta)-T_1(\eps,0)\}\theta_\eps'=2bc\eps^4\eta^2\theta_\eps'$.
\par
Let $P_\eps: L^2_{\hat{\a}}\to \ker_g(L_\eps(0))$ be the spectral projection associated
with $L_\eps(0)$ and let $Q_\eps=I-P_\eps(0)$.
Since $\widetilde{U}\in Q_\eps L^2_{\hat{\a}}(\R)$, we can translate \eqref{eq:ev-eta2} into
\begin{gather}
  \label{eq:ev-eta3}
  2\widehat{L}_\eps(\eta)\widetilde{U}+Q_\eps G_1(\gamma,\Lambda_1,\eps,\eta)
-i\eta Q_\eps G_2(\gamma,\Lambda_1,\eps,\eta)=0\,,
\\ \label{eq:ev-eta4}
F_1(\gamma,\Lambda_1,\eps,\eta):=
\left\la G_1(\gamma,\Lambda_1,\eps,\eta)-i\eta G_2(\gamma,\Lambda_1,\eps,\eta)
+2\{\widetilde{L}_\eps(\eta)-L_\eps(0)\}\widetilde{U},\theta_\eps\right\ra\,,
\\ \label{eq:ev-eta5}
F_2(\gamma,\Lambda_1,\eps,\eta):=\left\la G_1(\gamma,\Lambda_1,\eps,\eta)-i\eta G_2(\gamma,\Lambda_1,\eps,\eta)
+2\{\widetilde{L}_\eps(\eta)-L_\eps(0)\}\widetilde{U},\int_{-\infty}^z\tilde{\theta}_{1,\eps}\right\ra\,,
\end{gather}
where $\widehat{L}_\eps(\eta)=Q_\eps \widetilde{L}_\eps(\eta)Q_\eps$.
Let $k_1$ be a positive number such that
\begin{multline*}
 \sup_{\eps\in(0,\eps_0]\,,\,\eta\in[-\eta_0,\eta_0]}
\bigl(\|T_1(\eps,\eta)\|_{B(H^2_{\hat{\a}}(\R),L^2_{\hat{\a}}(\R))}
+\|T_2(\eps,\eta)\|_{B(H^1_{\hat{\a}}(\R),L^2_{\hat{\a}}(\R))}
\\ +\|B_\eps(\eta)\pd_z^{-1}\|_{B(H^1_{\hat{\a}}(\R),L^2_{\hat{\a}}(\R))}
\bigr)\le k_1\,.  
\end{multline*}
Suppose $\sup_{\eta\in[-\eta_0,\eta_0]}
\left(|\Lambda_1(\eta)|+|\gamma(\eta)|\right)\le k_2$ for a $k_2>0$.
Since 
$\|Q_\eps L_\eps(0)^{-1}Q_\eps\|_{B(L^2_{\hat{\a}}(\R), H^3_{\hat{\a}}(\R))}$ is uniformly bounded in $\eps\in(0,\eps_0)$ and
$$\|\widehat{L}_\eps(\eta)-Q_\eps L_\eps(0)Q_\eps\|_{B(H^2_{\hat{\a}}(\R),L^2_{\hat{\a}}(\R))}\lesssim 
\eta^2k_1(1+k_2^2\eps^2)+\eta k_1k_2\,,$$
we see that $\widehat{L}_\eps(\eta)^{-1}:Q_\eps L^2_{\hat{\a}}(\R)\to Q_\eps H^3_{\hat{\a}}(\R)$
is uniformly bounded in $\eps\in(0,\eps_0)$ and $\eta\in [-\eta_0,\eta_0]$
provided $\eps_0$ and $\eta_0$ are sufficiently small.
Thus there exists a positive constant $C_1$ such that
$$\sup_{\eps\in(0,\eps_0]\,,\,\eta\in[-\eta_0,\eta_0]}\|\widetilde{U}(\eta)\|_{H^3_{\hat{\a}}(\R)}
\le C_1\{(1+k_2)^2+\eps_0^2\eta_0k_2^3\}\,.$$
Let
\begin{gather}
  \label{eq:gamma0}
\gamma_\eps^0=\frac{f_3(\eps)}{f_4(\eps)}\,,\\
\notag
f_3(\eps)=
(\Lambda_{1,\eps}^0)^2 
\left\la T_1(\eps,0)\theta_{1,\eps}+\eps^2B_\eps(0)\theta_\eps, \int_{-\infty}^z \tilde{\theta}_{1,\eps}(z_1)\,dz_1\right\ra
-
\left\la T_2(\eps,0)\theta_\eps', \int_{-\infty}^z \tilde{\theta}_{1,\eps}(z_1)\,dz_1\right\ra\,,
\\ \notag
f_4(\eps)=2 \left\la L_\eps(0)\theta_{1,\eps},\int_{-\infty}^z \tilde{\theta}_{1,\eps}(z_1)\,dz_1\right\ra\,.
\end{gather}
By \eqref{eq:kerLe2} and \eqref{eq:theta-approx},
$f_4(\eps)=3\la \theta_0,\theta_0\ra+O(\eps^2)$.
Using \eqref{eq:Lambda1}, \eqref{eq:theta-approx} and the fact that
$(\Lambda_{1,\eps}^0)^2=\frac13+O(\eps^2)$ and
\begin{equation}
  \label{eq:theta-approx2}
\|\tilde{\theta}_{1,\eps}-2\theta_0-z\theta_0'\|_{H^k_{\hat{\a}}(\R)\cap H^k_{-\hat{\a}}(\R)}=O(\eps^2)
\quad\text{for any $k\ge0$,}
\end{equation}
we have
\begin{align*}
f_3(\eps)=&   \left\la \frac{2}{3}\theta_{1,\eps}-\theta_\eps,
\int_{-\infty}^z \tilde{\theta}_{1,\eps}\right\ra+O(\eps^2)
= -\frac{1}{6}\|\theta_0\|_{L^1(\R)}^2+O(\eps^2)\,.
\end{align*}
Thus we have
$$ \gamma_\eps^0
=-\frac{1}{18}\frac{\| \theta_0(z)\|_{L^1(\R)}^2}{\la \theta_0,\theta_0\ra}
+O(\eps^2)
=-\frac{1}{3\hat{\a}_0}+O(\eps^2)\,.
$$
In view of \eqref{eq:Lambda1} and \eqref{eq:gamma0},
\begin{align*}
  F_1(\widetilde{U}_0,\gamma_\eps^0,\Lambda_{1,\eps}^0,\eps,0)
=& \left\la G_1(\gamma,\Lambda_{1,\eps}^0,\eps,0), \theta_\eps\right\ra
\\=& \la T_2(\eps,0)\theta_\eps', \theta_\eps\ra
-(\Lambda_{1,\eps}^0)^2\la T_1(\eps,0)\theta_{1,\eps}+\eps^2B_\eps(0)\theta_\eps,\theta_\eps\ra
\\=& 0\,,
\end{align*}
where $\widetilde{U}_0=\widetilde{U}(0)$.
\begin{align*}
    F_2(\widetilde{U}_0,\gamma_\eps^0,\Lambda_{1,\eps}^0,\eps,0)
=& \left\la G_1(\gamma_\eps^0,\Lambda_{1,\eps}^0,\eps,0),
\int_{-\infty}^z \tilde{\theta}_{1,\eps}(z_1)\,dz_1\right\ra = 0\,,
\end{align*}
Next, we compute the Fr\'echet derivative of $(F_1,F_2)$ at
$\mathcal{U}_0=(\widetilde{U}_0,\gamma_\eps^0,\Lambda_{1,\eps}^0,\eps,0)$.
By \eqref{eq:kerLe2}, \eqref{eq:kerLe}, \eqref{eq:theta-approx}
and \eqref{eq:theta-approx2},
\begin{align*}
& \pd_{\gamma}F_1(\mathcal{U}_0)=2\la L_\eps(0)\theta_{1,\eps},\theta_\eps\ra=0\,,\\
& \pd_{\Lambda_1}F_1(\mathcal{U}_0)
=-2\Lambda_{1,\eps}^0\la T_1(\eps,0)\theta_{1,\eps}+\eps^2B_\eps(0)\theta_\eps,\theta_\eps\ra
=-6\Lambda_{1,\eps}^0\la \theta_0,\theta_0\ra+O(\eps^2)\,,
\\ & \pd_{\gamma}F_2(\mathcal{U}_0)
=2\la L_\eps(0)\theta_{1,\eps},\int_{-\infty}^z\tilde{\theta}_{1,\eps}\ra
=2\la \theta_{1,\eps},\theta_\eps\ra=3\la \theta_0,\theta_0\ra+O(\eps^2)\,,
\\ & \pd_{\Lambda_1}F_2(\mathcal{U}_0)
= -2\Lambda_{1,\eps}^0
\left\la T_1(\eps,0)\theta_{1,\eps}+\eps^2B_\eps(0)\theta_\eps,
     \int_{-\infty}^z \tilde{\theta}_{1,\eps}\right\ra
=-2\Lambda_{1,\eps}^0\|\theta_0\|_{L^1}^2+O(\eps^2)\,.
\end{align*}
and $D_{(\gamma,\Lambda_1)}(F_1,F_2)(\mathcal{U}_0)=
  \begin{pmatrix}\pd_{\gamma}F_1(\mathcal{U}_0) & \pd_{\Lambda_1}F_1(\mathcal{U}_0)
\\ \pd_{\gamma}F_2(\mathcal{U}_0) & \pd_{\Lambda_1}F_2(\mathcal{U}_0)
\end{pmatrix}$ is invertible.
Thus by the implicit function theorem, there exists a smooth curve
$(\gamma_\eps(\eta),\Lambda_{1,\eps}(\eta))$ around $\eta=0$ satisfying
\begin{equation}
\label{eq:La2}
\gamma_\eps(0)=\gamma_\eps^0\,,\quad
\Lambda_{1,\eps}(0)=\Lambda_{1,\eps}^0\,,\quad
\Lambda_{1,\eps}'(0)=
-\frac{\pd_\eta F_1(\mathcal{U}_0)}{\pd_{\Lambda_1}F_1(\mathcal{U}_0)}
=:i\Lambda_{2,\eps}^0\,.
\end{equation}
Since 
\begin{align*}
G_2(\gamma_\eps(0),\Lambda_{1,\eps}(0),\eps,0)
=& \Lambda_{1,\eps}(0)\{T_2(\eps,0)+\gamma_\eps(0)T_1(\eps,0)\}\theta_{1,\eps}+O(\eps^2)
\\=& \Lambda_{1,\eps}^0\left\{-\int_z^\infty \theta_{1,\eps}(z_1)\,dz_1
+2\gamma_\eps(0)\theta_{1,\eps}\right\}+O(\eps^2)
\quad\text{in $L^2_{\hat{\a}}(\R)$,}
\end{align*}
we have
\begin{align*}
  \pd_\eta F_1(\widetilde{U}_0,\gamma_\eps^0,\Lambda_{1,\eps}^0,\eps,0)
=& 
-i \left\la G_2(\widetilde{U}_0,\gamma_\eps^0,\Lambda_{1,\eps}^0,\eps,0),\theta_\eps\right\ra\,,
\\ =&
i\Lambda_{1,\eps}^0\left\{\frac12\|\theta_0\|_{L^1(\R)}^2-3\gamma_\eps^0\la \theta_0,\theta_0\ra\right\}
+O(\eps^2)
\\=& \frac{2i}{3}\Lambda_{1,\eps}^0\|\theta_0\|_{L^1(\R)}^2\,,
\end{align*}
and $\Lambda_{2,\eps}^0= \frac{1}{9}\|\theta_0\|_{L^1(\R)}^2
\|\theta_0\|_{L^2(\R)}^{-2}+O(\eps^2)=2/(3\hat{\a})+O(\eps^2)$.
\par
Letting $\Lambda_\eps(\eta)=i\eta\Lambda_{1,\eps}(\eta)$ and
\begin{align*}
U_\eps(\eta)=& \theta_\eps'-\{\Lambda_\eps(\eta)-\eta^2\gamma_\eps(\eta)\}\theta_{1,\eps}
\\ & 
-\frac{\eta^2}{2}\widehat{L}_\eps(\eta)^{-1}Q_\eps
\{G_1(\gamma_\eps(\eta),\Lambda_{1,\eps}(\eta),\eps,\eta)
-i\eta G_2(\gamma_\eps(\eta),\Lambda_{1,\eps}(\eta),\eps,\eta)\}\,,
\end{align*}
we have \eqref{eq:f-Ue} and \eqref{eq:f-cc} because
$\overline{\widetilde{L}_\eps(\eta)}=\widetilde{L}_\eps(-\eta)$
and $\overline{F_j(\gamma,\Lambda,\eta,\eps)}
=F_j(\overline{\gamma},\overline{\Lambda},-\eta,\eps)$ for
$j=1$, $2$.
Thus we complete the proof.
\end{proof}

\par
 
\begin{lemma}
  \label{lem:resonance1}
Let $c$, $\hat{\a}$, $\eps_0$ and $\eta_0$
be as in Lemma~\ref{lem:construction}.
For any $\eps\in(0,\eps_0)$ and  $\eta\in[-\eps^2\eta_0,\eps^2\eta_0]$, let
$\lambda(\eta)=\eps^3\Lambda_\eps(\eps^{-2}\eta)$,
$u(z,\eta)={}^t\!(u_1(z,\eta),u_2(z,\eta))$,
$v(z,\eta)={}^t\!(v_1(z,\eta),v_2(z,\eta))$  and
\begin{align*}
u_1(z,\eta)=& \pd_z^{-1}U_\eps(\eps z,\eps^{-2}\eta)\,,
\\
u_2(z,\eta)=& -c\eps U_\eps(\eps z,\eps^{-2}\eta)
+\lambda(\eta)(\pd_z^{-1}U_\eps)(\eps z,\eps^{-2}\eta)\,,
\\
v_1(z,\eta)=& (\lambda(-\eta)+c\pd_z)B(\eta)
\int_{-\infty}^{\eps z}U_\eps(-z_1,-\eps^{-2}\eta)\,dz_1
\\ & -(2q_c\pd_z+q_c')
\int_{-\infty}^{\eps z} U_\eps(-z_1,-\eps^{-2}\eta)\,dz_1\,,
\\
v_2(z,\eta)=& B(\eta)\int_{-\infty}^{\eps z} U_\eps(-z_1,-\eps^{-2}\eta)\,dz_1\,.
\end{align*}
Then
\begin{gather}
\label{eq:Lu,Lv}
\mL(\eta)u(\cdot,\eta)=\lambda(\eta)u(\cdot,\eta)\,, \quad
\mL(\eta)^*v(\cdot,\eta)=\lambda(-\eta)v(\cdot,\eta)\,,
\\ \label{eq:cc} 
\overline{\lambda(\eta)}=\lambda(-\eta)\,,\quad
\overline{u(z,\eta)}=u(z,-\eta)\,,\quad
\overline{v(z,\eta)}=v(z,-\eta)\,,
\\
\label{eq:gg*}
\la u(x,\eta),v(x,-\eta)\ra=0
\quad\text{for $\eta\in [-\eps^2\eta_0,\eps^2\eta_0]\setminus\{0\}$.}
\end{gather}
Moreover,  for any $k\in\N$, the mappings
$[-\eps^2\eta_0,\eps^2\eta_0]\ni\eta \mapsto u\left(\eps^{-1}\cdot,\eta\right)
\in H^k_{\hat{\a}}(\R)\times H^{k-1}_{\hat{\a}}(\R)$ 
and $[-\eps^2\eta_0,\eps^2\eta_0]\ni\eta\mapsto
v\left(\eps^{-1}\cdot,\eta\right) \in
H^k_{-\hat{\a}}(\R)\times H^{k-1}_{-\hat{\a}}(\R)$ are smooth.
\end{lemma}
\begin{proof}
By \eqref{eq:u1},\eqref{eq:u2} and the definition of $U_\eps(\eta)$,
we see that $u(z,\eta)$ is a solution of \eqref{eq:evBL-eta} with $\lambda=\lambda(\eta)$.
The mappings $\eta\mapsto u(\eps^{-1}\cdot,\eta)$ and $v(\eps^{-1}\cdot,\eta)$
are smooth thanks to the smoothness of $U_\eps(\eta)$
and \eqref{eq:cc} follows from \eqref{eq:f-cc}.
\par
Suppose $\mL(\eta)^*\begin{pmatrix}v_1 \\ v_2 \end{pmatrix}
=\lambda(-\eta)\begin{pmatrix}v_1 \\ v_2 \end{pmatrix}$ and $\tv_2=B(\eta)^{-1}v_2$. Then
\begin{gather}
\label{eq:tv1}
v_1=(\lambda(-\eta)+c\pd_z)B(\eta)\tv_2  +v_{2,c}(\eta)^*\tv_2\,,
\\
\label{eq:tv2}
\{A(\eta)(\pd_z^2-\eta^2)
-(\lambda(-\eta)+c\pd_z)^2B(\eta)\}\tv_2
-\{v_{1,c}(\eta)^*+(\lambda(-\eta)+c\pd_z)v_{2,c}(\eta)^*\}\tv_2=0\,.
\end{gather}
Formally, we have $v_{2,c}(\eta)^*=-v_{2,c}(\eta)$ and 
$v_{1,c}(\eta)^*+c\pd_zv_{2,c}(\eta)^*=v_{1,c}(\eta)-cv_{2,c}(\eta)\pd_z$.
Using the change of variable $z\mapsto -z$ and the fact that
$q_c$ is an even function, we see that $\tv_2(-z)$ satisfies \eqref{eq:u1}
with $\lambda=\lambda(-\eta)$ and that
$$\tv_2(z,\eta)=\int_{-\infty}^{\eps z} U_\eps(-z_1,-\eps^{-2}\eta)\,dz_1$$
is a solution of \eqref{eq:tv2}.
Thus we prove $\mL(\eta)^*v(\cdot,\eta)=\lambda(-\eta)v(\cdot,\eta)$.
We have \eqref{eq:gg*} from \eqref{eq:Lu,Lv} since
$\overline{\lambda(\eta)}\ne \lambda(\eta)$ for
$\eta\in [-\eps^2\eta_0,\eps^2\eta_0]\setminus\{0\}$.
Thus we complete the proof.
\end{proof}
Let
\begin{gather*}
g(z,\eta)=\frac{\sqrt{3}}{2}
\left(1+i\frac{\Re\la u(\cdot,\eta),v(\cdot,\eta)\ra}
{\Im \la u(\cdot,\eta),v(\cdot,\eta)\ra}
\right)
\begin{pmatrix}
u_1(z,\eta)\\ u_2(z,\eta)
\end{pmatrix}\,,
\\
g^*(z,\eta)=-\frac{\hat{\a}_0}{4}
\begin{pmatrix}  v_1(z,\eta) \\ v_2(z,\eta)\end{pmatrix}\,,
\end{gather*}
By \eqref{eq:cc} and \eqref{eq:gg*},
\begin{gather}
\label{eq:cc2}
\overline{g(z,\eta)}=g(z,-\eta)\,,\quad
\overline{g^*(z,\eta)}=g^*(z,-\eta)\,,\\
  \label{eq:gg*1}
\la g(\cdot,\eta),g^*(\cdot,-\eta)\ra=0\quad\text{and}\quad
\Re \la g(\cdot,\eta),g^*(\cdot,\eta)\ra=0
\quad\text{for $\eta\in[-\eps^2\eta_0,\eps^2\eta_0]$.}
\end{gather}
To resolve the degeneracy of the subspace
$\spann\{g(\cdot,\eta),g(\cdot,-\eta)\}$ at $\eta=0$, we introduce
\begin{gather}
\label{eq:def-gk}
  g_1(z,\eta)=\frac{1}{2}\{g(z,\eta)+g(z,-\eta)\}\,,
\quad
g_2(z,\eta)=
\frac{1}{2i\kappa(\eta)}\{g(z,\eta)-g(z,-\eta)\}
\,,\\
\label{eq:def-gk*}
g_1^*(z,\eta)=\frac{i}{2\kappa(\eta)}\{g^*(z,\eta)-g^*(z,-\eta)\}\,,
\quad g_2^*(z,\eta)=\frac{1}{2}\{g^*(z,\eta)+g^*(z,-\eta)\}\,,
\end{gather}
where $\kappa(\eta)=\frac12\Im\la g(\cdot,\eta),g^*(\cdot,\eta)\ra$.
By \eqref{eq:cc2} and \eqref{eq:gg*1}, we have
\begin{equation}
\label{eq:gg*2}
  \la g_i(\cdot,\eta),g_j^*(\cdot,\eta)\ra =\delta_{ij}
\quad\text{for $i$, $j=1$, $2$.}
\end{equation}
The profiles of $g_k(z,\eta)$ and $g_k^*(z,\eta)$ for small line solitary waves
are as follows.
\begin{corollary}
  \label{lem:resonance2}
Let $c$, $\hat{\a}$, $\eps_0$ and $\eta_0$
be as in Lemma~\ref{lem:construction}.
For every $k\ge0$, there exists a positive constant $C$
such that for $\eta\in [-\eps^2\eta_0,\eps^2\eta_0]$ and $\eps\in(0,\eps_0]$,
\begin{gather}
\label{eq:g1}
 \left\|
   \begin{pmatrix}
  1 & 0 \\ 0 & \eps^{-1}
   \end{pmatrix}
\left\{g_1(\eps^{-1}\cdot,\eps^2\eta) -\frac{\sqrt{3}}{2}
\begin{pmatrix}\theta_\eps \\ -\eps \theta_\eps' \end{pmatrix}
\right\} \right\|_{H^k_{\hat{\a}}(\R)\times H^{k-1}_{\hat{\a}}(\R)}
\le C(\eps^2+\eta^2)\,,
\\ \label{eq:g2}
\left\| \begin{pmatrix} 1 & 0 \\ 0 & \eps^{-1} \end{pmatrix}
\left\{ g_2\left(\eps^{-1}\cdot,\eps^2\eta\right)
-\frac{1}{2}
\begin{pmatrix}
\int_z^\infty \theta_{1,\eps}-2\hat{\a}_0^{-1}\theta_\eps
\\
\eps(\theta_{1,\eps}+2\hat{\a}_0^{-1}\theta_\eps')
\end{pmatrix}\right\}
\right\|_{H^k_{\hat{\a}}(\R)\times H^{k-1}_{\hat{\a}}(\R)}
\le C(\eps^2+\eta^2)\,,
\\ \label{eq:g1*}
\left\| \begin{pmatrix}  \eps^{-1} & 0 \\ 0 & 1 \end{pmatrix}
\left\{ g_1^*(\eps^{-1}\cdot,\eps^2\eta)
-\frac{\hat{\a}_0}{4\sqrt{3}}
\begin{pmatrix}
\eps \theta_{1,\eps} \\ \int_{-\infty}^z \theta_{1,\eps}
\end{pmatrix}\right\}
\right\|_{H^k_{-\hat{\a}}(\R)\times H^{k-1}_{-\hat{\a}}(\R)} \le C(\eps^2+\eta^2)\,,
\\ \label{eq:g2*}
 \left\| \begin{pmatrix}  \eps^{-1} & 0 \\ 0 & 1 \end{pmatrix}
\left\{ g_2^*(\eps^{-1}\cdot,\eps^2\eta) 
-\frac{\hat{\a}_0}{4}
\begin{pmatrix} \eps \theta_\eps' \\ \theta_\eps \end{pmatrix}\right\}
\right\|_{H^k_{-\hat{\a}}(\R)\times H^{k-1}_{-\hat{\a}} }\le C(\eps^2+\eta^2)\,.
\end{gather}
\end{corollary}
\begin{proof}
First, we expand $\la u(\cdot,\eps^2\eta), v(\cdot,\eps^2\eta)\ra$
into powers of $\eta$ up to the second order.
By the definitions of $u(z,\eta)$ and $v(z,\eta)$,
  \begin{align*}
 \la u(\cdot,\eps^2\eta), v(\cdot,\eps^2\eta)\ra
=& 2\lambda(\eps^2\eta)
\la u_1(\cdot,\eps^2\eta), v_2(\cdot,\eps^2\eta)\ra
-2c\la \pd_zu_1(\cdot,\eps^2\eta), v_2(\cdot,\eps^2\eta)\ra
\\ &-\left\la u_1(\cdot,\eps^2\eta), 2\eps q_cU_\eps(-\eps\cdot,-\eta)
+q_c'\int_{-\infty}^{\eps\cdot} U_\eps(-z_1,-\eta))\,dz_1\right\ra\,.
  \end{align*}
By \eqref{eq:defU} and \eqref{eq:La2},
$$u_1(\eps^{-1}z,\eps^2\eta)=\theta_\eps
+\left\{i\eta\Lambda_{1,\eps}^0
-\eta^2(\gamma_\eps^0+\Lambda_{2,\eps}^0)\right\}
\int_z^\infty\theta_{1,\eps}
-\eta^2\int_z^\infty\widetilde{U}_0(z_1)\,dz_1+O(\eta^3)
\quad\text{in $H^1_{\hat{\a}}(\R)$,}$$
and 
\begin{align*}
v_2(\eps^{-1}z,\eps^2\eta)= &
B_\eps(\eta)\left[
-\theta_\eps
+\left\{i\eta\Lambda_{1,\eps}^0+
\eta^2(\gamma_\eps^0+\Lambda_{2,\eps}^0)\right\}
\int_{-\infty}^z\theta_{1,\eps}
\right]
\\ & +\eta^2\int_{-\infty}^z B_\eps(0)\widetilde{U}_0(-z_1)\,dz_1
+O(\eta^3)
\quad\text{in $H^1_{-\hat{\a}}(\R)$.}
 \end{align*}
Using the fact that $\theta_\eps$ and $\theta_{1,\eps}$ are even, we have
\begin{align*}
\eps\la u_1(\cdot,\eps^2\eta), v_2(\cdot,\eps^2\eta)\ra 
=& 
-\la B_\eps(0)\theta_\eps,\theta_\eps\ra
-i\eta\Lambda_{1,\eps}^0\left(\int_\R B_\eps(0)\theta_\eps\right)\left( \int_\R \theta_{1,\eps}\right)+O(\eta^2)\,,
\end{align*}
\begin{align*}
\la \pd_zu_1(\cdot,\eps^2\eta), v_2(\cdot,\eps^2\eta)\ra 
=& \left\{2i\eta\Lambda_{1,\eps}^0-2\eta^2(\gamma_\eps^0+\Lambda_{2,\eps}^0)\right\}
\la B_\eps(0)\theta_\eps,\theta_{1,\eps}\ra
\\ & -\eta^2\left\{
(\Lambda_{1,\eps}^0)^2\la B_\eps(0)\theta_{1,\eps},\int_{-\infty}^z \theta_{1,\eps}\ra
+2\left\la \widetilde{U}_0,B_\eps(0)\theta_\eps\right\ra\right\} +O(\eta^3)\,,
\end{align*}
\begin{align*}
 \left\la u_1(\cdot,\eps^2\eta), 2\eps q_cU_\eps(-\eps\cdot,-\eta)
+q_c'\int_{-\infty}^{\eps z} U_\eps(-z_1,-\eta)\,dz_1\right\ra
=&
-3i\eps^2\eta\Lambda_{1,\eps}^0\la \theta_\eps^2,\theta_{1,\eps}\ra+O(\eps^2\eta^2)\,.
\end{align*}
In the last line, we use \eqref{eq:KP2-scale2}.
Since $\widetilde{U}(\eta) \perp \theta_\eps$ and
$\|B_\eps(0)\theta_\eps-\theta_\eps\|_{L^2_{-\hat{\a}}}=O(\eps^2)$, we have
$\la \widetilde{U}_0, B_\eps(0)\theta_\eps \ra= O(\eps^2)$.
Combining the above with \eqref{eq:theta-approx} and the fact that
$\lambda(\eps^2\eta)=\eps^3\{i\eta\Lambda_{1,\eps}^0+O(\eta^2)\}$,
we have
\begin{align*}
\Im \la u(\cdot,\eps^2\eta), v(\cdot,\eps^2\eta)\ra
=&
-2c\Im \la \pd_zu_1(\cdot,\eps^2\eta), v_2(\cdot,\eps^2\eta)\ra
+O(\eps^2\eta+\eta^3)
\\ =&
-4\eta \Lambda_{1,\eps}^0\la B_\eps(0)\theta_\eps,\theta_{1,\eps}\ra
+O(\eps^2\eta+\eta^3)
\\=& \left\{-\frac{16}{\sqrt{3}\hat{\a}_0}+O(\eps^2)\right\}\eta
+O(\eta^3)\,,
\end{align*}
\begin{align*}
\Re \la u(\cdot,\eps^2\eta), & v(\cdot,\eps^2\eta)\ra
=
-2c\Re \la \pd_zu_1(\cdot,\eps^2\eta), v_2(\cdot,\eps^2\eta)\ra  
+O(\eps^2\eta^2)
\\ =& 2\eta^2\left\{
(\Lambda_{1,\eps}^0)^2\la \theta_{1,\eps}, B_\eps(0)\int_{-\infty}^z\theta_{1,\eps}\ra
+2(\gamma_\eps^0+\Lambda_{2,\eps}^0)\la B_\eps(0)\theta_\eps,\theta_{1,\eps}\ra
+O(\eps^2+\eta^2)\right\}
\\=& \frac{32}{3\hat{\a}_0^2}\eta^2+O(\eps^2\eta^2+\eta^4)\,.
\end{align*}
Note that $\Re \la u(\cdot,\eps^2\eta), v(\cdot,\eps^2\eta)\ra$
is even in $\eta$ thanks to \eqref{eq:cc}.
Thus we have
$$
\frac{\Re \la u(\cdot,\eps^2\eta), v(\cdot,\eps^2\eta)\ra}
{\Im \la u(\cdot,\eps^2\eta), v(\cdot,\eps^2\eta)\ra}
=-\frac{2\eta}{\sqrt{3}\hat{\a}_0}+O(\eps^2\eta)\,,
$$
\begin{align*}
\la g(\cdot,\eps^2\eta), g^*(\cdot,\eps^2\eta)\ra
=& -\frac{\sqrt{3}\hat{\a}_0}{8}i
\Im\la u(\cdot,\eps^2\eta), v(\cdot,\eps^2\eta)\ra
\left\{1+\left(
\frac{\Re \la u(\cdot,\eps^2\eta), v(\cdot,\eps^2\eta)\ra}
{\Im \la u(\cdot,\eps^2\eta), v(\cdot,\eps^2\eta)\ra}\right)^2\right\}
\\=& 2i\eta\{1+O(\eps^2+\eta^2)\}\,,
\end{align*}
and \eqref{eq:g1}--\eqref{eq:g2*} follow immediately from the definitions
of $g_k$ and $g_k^*$ ($k=1$, $2$).
\end{proof}
\begin{remark}
\label{rem:L}
In view of \eqref{eq:cc2}, we have
\begin{gather*}
  \mL(\eta)g_1(\cdot,\eta)
=\Re\lambda(\eta)g_1(\cdot,\eta)-\kappa(\eta)\Im\lambda(\eta)g_2(\cdot,\eta)\,,
\\
\mL(\eta)g_2(\cdot,\eta)
=\frac{\Im\lambda(\eta)}{\kappa(\eta)}g_1(\cdot,\eta)
+\Re\lambda(\eta)g_2(\cdot,\eta)\,,\\
\mL(\eta)^*g_1^*(\cdot,\eta)
=\Re\lambda(\eta)g_1^*(\cdot,\eta)
+\frac{\Im\lambda(\eta)}{\kappa(\eta)}g_2^*(\cdot,\eta)\,,\\
\mL(\eta)^*g_2^*(\cdot,\eta)
=-\kappa(\eta)\Im\lambda(\eta)g_1^*(\cdot,\eta)
+\Re\lambda(\eta)g_2^*(\cdot,\eta)\,.
  \end{gather*}
\end{remark}
Now we define a spectral projection to resonant modes.
Let  $\mathcal{P}(\eta_0)$ be an operator defined by
\begin{gather*}
\mathcal{P}(\eta_0)f(z,y)=\frac{1}{\sqrt{2\pi}}\sum_{k=1,\,2}
\int_{-\eta_0}^{\eta_0}c_k(\eta)g_k(z,\eta)e^{iy\eta}\,d\eta
\,, \\
 c_k(\eta)=\int_\R (\mF_yf)(z,\eta)\cdot g_k^*(z,\eta)\,dz
\end{gather*}
for $f\in X$ and let $\mathcal{Q}(\eta_0)=I-\mathcal{P}(\eta_0)$.
Using Corollary~\ref{lem:resonance2}, we can prove that $\mathcal{P}(\eta_0)$
and $\mathcal{Q}(\eta_0)$ are spectral projections associated with $\mL$
in exactly the same way with \cite[Lemma~3.1]{Miz15}.
\begin{lemma}
  \label{lem:p0}
Let $c=\sqrt{1+\eps^2}$ and $\a\in(0,\hat{\a}_0/2)$. Then there exist positive constants
$\eps_0$ and $\eta_1$ such that for any $\eps\in(0,\eps_0)$ and $\eta_0\in[0,\eta_1]$,
\begin{enumerate}
\item \label{it:p01}
$\|\mathcal{P}(\eps^2\eta_0)f\|_X \le C\|f\|_X$ for any $f\in X$,
where $C$ is a positive constant depending only on $\a$, $\eps$ and $\eta_1$,
\item \label{it:p03}
$\mL \mathcal{P}(\eps^2\eta_0)f=\mathcal{P}(\eps^2\eta_0)\mL f$ for any $f\in D(\mL)$,
\item \label{it:p04}
$\mathcal{P}(\eps^2\eta_0)^2=\mathcal{P}(\eps^2\eta_0)$ on $X$,
\item \label{it:p05}
$e^{t\mL}\mathcal{P}(\eps^2\eta_0)=\mathcal{P}(\eps^2\eta_0)e^{t\mL}$ on $X$.
\end{enumerate}
\end{lemma}
\bigskip

\section{Properties of the free operator $\mL_0$}
\label{sec:free operator}
In this section, we investigate properties of the linearized operator 
$\mL_0$ in $X$.
To begin with, we investigate the spectrum of $\mL_0$.
\begin{lemma}
  \label{lem:spectrum-free}
Let $\a_c'=\sqrt{\frac{c-1}{bc-a}}$.
Suppose $0<a<b$, $c>1$ and $\a\in(0,\a_c')$. Then
\begin{equation*}
\sigma(\mL_0(D))\subset
\left\{\lambda \in\C \mid \Re\lambda <-\frac{\a}{2}(c-1)\right\}\,.
\end{equation*}
\end{lemma}
By \eqref{eq:fm0}, the operator
$\begin{pmatrix}
m_{11}(D) & m_{12}(D) \\ m_{21}(D) & m_{22}(D)
\end{pmatrix}$
is bounded on $X$ if and only if 
\begin{equation}
  \label{eq:fm}
\sum_{i,j=1,2}(1+\xi^2+\eta^2)^{(j-i)/2}|m_{ij}(\xi+i\a,\eta)|<\infty\,.
\end{equation}
\par

The symbol of the operator $\mL_0$ is
$$
\mL_0(\xi,\eta)=\begin{pmatrix}
ic\xi & 1 \\  -(\xi^2+\eta^2)S(\xi,\eta)^2 & ic\xi \end{pmatrix}\,,
\quad
S(\xi,\eta)=\sqrt{\frac{1+a(\xi^2+\eta^2)}{1+b(\xi^2+\eta^2)}}\,,
$$
and we observe
$L_0(\xi,\eta)P(\xi,\eta)
=\diag(\lambda_+(\xi,\eta),\lambda_-(\xi,\eta))P(\xi,\eta)$,
where
\begin{gather}
\lambda_\pm(\xi,\eta)=ic\xi\pm i\mu(\xi,\eta)S(\xi,\eta)\,,
\quad \mu(\xi,\eta)=\xi\sqrt{1+\xi^{-2}\eta^2}\,,
\notag \\
 \label{eq:P-def}
P(\xi,\eta)=
\begin{pmatrix} -i\mu(\xi,\eta)^{-1} &  i\mu(\xi,\eta)^{-1}
\\ S(\xi,\eta) & S(\xi,\eta) \end{pmatrix}\,.  
\end{gather}
\par
To investigate properties of the resolvent operator $(\lambda-\mL_0)^{-1}$,
we need the following.
\begin{claim}
  \label{cl:mu}
Suppose $0<a<b$ and $\a>0$. Then
\begin{gather}
  \label{eq:im-mu}
0\le\Im\mu(\xi+i\a,\eta)\le \Im\mu(\xi+i\a,0)=\a\quad\text{for $\xi\in\R$,}
\\
\label{eq:sign-mu}
\xi\Re\mu(\xi+i\alpha,\eta)>0\,, \quad \Im\mu(\xi+i\a,\eta) >0
\quad\text{for $\xi\ne0$.}
\end{gather}
\end{claim}
\begin{claim}
  \label{cl:spectrum-free1}
Suppose $0<a<b$ and $0<\a<\a_c$. Then
\begin{gather}
\label{eq:cl-sf0}
\Re S(\xi+i\a,\eta)>0
\quad\text{for $(\xi,\eta)\in\R^2$,}
\,,\\
\label{eq:cl-sf1}
\xi\Im S(\xi+i\a,\eta)<0
\quad\text{for $\xi\in\R\setminus\{0\}$ and $\eta\in\R$,}
\\ \label{eq:cl-sf2}
\sqrt{\mathstrut{\frac{a}{b}}}<|S(\xi+i\a,\eta)| <S(i\a,0)<c
\quad\text{for $(\xi,\eta)\in\R^2\setminus\{(0,0)\}$,}
\\ \label{eq:cl-sf3}
|S(\xi+i\a,\eta)|<1-\frac{b-a}{2}
\frac{\xi^2+\eta^2-\a^2}{1+b(\xi^2+\eta^2-\a^2)}
\quad\text{for $(\xi,\eta)\in\R^2$.}
\end{gather}
\end{claim}
\begin{claim}
\label{cl:spectrum-free2}
Suppose $0<a<b$, $c>1$ and $\a\in(0,\a_c')$. Then for $(\xi,\eta)\in\R^2$,
\begin{gather}
  \label{eq:cl-sf4}
-2\alpha c < \Re\lambda_+(\xi+i\a,\eta)\le -\alpha  c\,,\\
\label{eq:cl-sf5'}
  \Re\lambda_-(\xi+i\a,\eta)\le 
-\a\left\{c-1+\frac{b-a}{2}\frac{\xi^2+\eta^2-\a^2}{1+b(\xi^2+\eta^2-\a^2)}
\right\}\,,\\
\label{eq:cl-sf5}
-\alpha c \le \Re\lambda_-(\xi+i\a,\eta)\le -\frac{\alpha}{2}(c-1)\,.
\end{gather}
\end{claim}

\begin{proof}[Proof of Claim~\ref{cl:mu}]
Since 
$$\mu(\xi+i\alpha,\eta)=(\xi+i\alpha)\sqrt{1+\frac{\eta^2}{(\xi+i\alpha)^2}}
=\mbox{sgn}(\xi)\sqrt{(\xi+i\alpha)^2+\eta^2}\,,$$
we have \eqref{eq:sign-mu}.
\par
Since $\Im\mu(i\a,\eta)=\sqrt{\a^2-\eta^2}$ for $\eta\in[-\a,\a]$
and $\Im\mu(i\a,\eta)=0$ for $\eta\in \R$ satisfying $|\eta|>\a$,
we have \eqref{eq:im-mu} for $\xi=0$.
Let $s=\eta^2$, $\gamma_1(\xi,s)=\Re\mu(\xi+i\a,\eta)$ and
$\gamma_2(\xi,s)=\Im\mu(\xi+i\a,\eta)$.
To prove \eqref{eq:im-mu}, it suffices to show that
$\gamma_2(\xi,s)$ is monotone decreasing in $s$ when $\xi\ne0$.
Differentiating
\begin{equation}
  \label{eq:g1-g2}
\gamma_1^2-\gamma_2^2=\xi^2-\a^2+s\quad\text{and}\quad\gamma_1\gamma_2=\a\xi
\end{equation}
with respect to $s$, we have
\begin{equation}
  \label{eq:g2_s}
\pd_s\gamma_2=-\frac{\gamma_2}{2(\gamma_1^2+\gamma_2^2)}\,.  
\end{equation}
Combining \eqref{eq:g2_s} with \eqref{eq:sign-mu}, we have $\pd_s\gamma_2<0$.
Thus we prove \eqref{eq:im-mu}.

\end{proof}
\begin{proof}[Proof of Claim~\ref{cl:spectrum-free1}]
We observe
  \begin{equation}
\label{eq:cl-sfpf1}
\begin{split}
S(\xi+i\a,\eta)^2=& \frac{1+a(\xi^2+\eta^2-\a^2)+2ia\a\xi}
{1+b(\xi^2+\eta^2-\a^2)+2ib\a\xi}
\\=& \frac{a}{b}+\left(1-\frac{a}{b}\right)
\frac{1}{1-b\a^2+b(\xi^2+\eta^2)+2ib\a\xi}\,.  
\end{split}
  \end{equation}
Since $0<a<b$ and $1-b\a^2>0$ for $\a\in(0,\a_c)$,
it follows from \eqref{eq:cl-sfpf1} that
\begin{gather}
\label{eq:cl-sf6a}
|S(\xi+i\a,\eta)|^2 \ge \Re S(\xi+i\a,\eta)^2
>\frac{a}{b}>0
\quad\text{for $(\xi,\eta)\in\R^2$,}
\\ 
\label{eq:cl-sf6b}  
\xi\Im S(\xi+i\a,\eta)^2<0
\quad\text{for $\xi\in\R\setminus\{0\}$ and $\eta\in\R$.}
\end{gather}
By \eqref{eq:cl-sf6a}, we have the first part of \eqref{eq:cl-sf2} and
\eqref{eq:cl-sf0} because $\Re S(i\a,0)=\sqrt{\frac{1-a\a^2}{1-b\a^2}}>0$
and $S(\xi+i\a,\eta)$ is continuous in $(\xi,\eta)\in\R^2$.
Eq.~\eqref{eq:cl-sf1} follows from \eqref{eq:cl-sf0} and \eqref{eq:cl-sf6b}.
\par
We have $c>S(i\a,0)$ for $\a\in(0,\a_c)$.
By \eqref{eq:cl-sfpf1} and the triangle inequality,
\begin{equation}
  \label{eq:cl-sfpf2}
\begin{split}
  |S(\xi+i\a,\eta)|^2 \le & 
\frac{a}{b}+\left(1-\frac{a}{b}\right)
\frac{1}{1+b(\xi^2+\eta^2-\a^2)}
\\=& \frac{1+a(\xi^2+\eta^2-\a^2)}{1+b(\xi^2+\eta^2-\a^2)}
\le  \frac{1-a\a^2}{1-b\a^2}=S(i\a,0)^2\,,
\end{split}  
\end{equation}
and $|S(\xi+i\a,\eta)|=S(i\a,0)$ if and only if $\xi=\eta=0$.
Thus we have the second part of \eqref{eq:cl-sf2}.
Furthermore, we have \eqref{eq:cl-sf3} from \eqref{eq:cl-sfpf2} since
$|S|\le (|S|^2+1)/2$. Thus we complete the proof.
\end{proof}

Using Claim~\ref{cl:spectrum-free1}, we will estimate the upper and
lower bounds of $\lambda_\pm(\xi+i\a,\eta)$.
\begin{proof}[Proof of Claim~\ref{cl:spectrum-free2}]
First, we will show
\begin{equation}
  \label{eq:muS>0}
 \Im\left(\mu(\xi+i\a,\eta)S(\xi+i\a,\eta)\right)\ge0
\quad\text{for $(\xi,\eta)\in\R^2$.}
\end{equation}
We see that $\mu(\xi+i\a,\eta)S(\xi+i\a,\eta)$ is a real number
if and only if $\xi=0$ and $|\eta|\ge\a$ since
\begin{align*}
\Im\left\{\mu(\xi+i\a,\eta)S(\xi+i\alpha,\eta)\right\}^2
= 2\alpha \xi
\left[\frac{a}{b}+\left(1-\frac{a}{b}\right)\frac{1}
{|1+b(\xi+i\alpha)^2+b\eta^2|^2}\right]
\end{align*}
and 
$$
\mu(i\a,\eta)^2S(i\a,\eta)^2=
(\eta^2-\a^2)\frac{1-b\a^2+\a\eta^2}{1-b\a^2+b\eta^2}\,.$$
Thanks to the continuity of $\Im(\mu S)(\xi+i\a,\eta)$ on $\R^2$ and
the fact that $\Im(\mu S)(i\a,0)>0$, we have \eqref{eq:muS>0}.
\par
By \eqref{eq:muS>0} and the definition of $\lambda_\pm$,
$$\Re\lambda_+(\xi+i\a,\eta)\le -\a c\,,\quad
\Re\lambda_-(\xi+i\a,\eta) \ge  -\a c\,.
$$
Since $0<\a<\a_c'<\a_c$, it follows from \eqref{eq:im-mu},
\eqref{eq:sign-mu},\eqref{eq:cl-sf1} and \eqref{eq:cl-sf2} that
\begin{align*}
\Re\lambda_+(\xi+i\a,\eta) \ge &
-\a c -\Im\mu(\xi+i\a,\eta)\Re S(\xi+i\a,\eta)
\\ > & -2\a c\,,
\end{align*}
\begin{equation}
  \label{eq:lambda-}
\begin{split}
\Re\lambda_-(\xi+i\a,\eta)
\le & -\a c +\Im\mu(\xi+i\a,\eta)\Re S(\xi+i\a,\eta)\,.
\end{split} 
\end{equation}
Combining \eqref{eq:lambda-} with \eqref{eq:im-mu} and \eqref{eq:cl-sf3},
we have  \eqref{eq:cl-sf5'}.
Since $x/(1+bx)$ is increasing on $[-\a^2,\infty)$ and $c>S(i\a,0)^2$
for $\a\in(0,\a_c')$,
\begin{align*}
\Re\lambda_-(\xi+i\a,\eta)\le &
-\a\left\{c-1-\frac{b-a}{2}\frac{\a^2}{1-b\a^2}\right\}
\\=& -\a\left(c-\frac12-\frac12S(i\a)^2\right)
<  -\frac{c-1}{2}\,.
\end{align*}
Thus we complete the proof.
\end{proof}

Now we are in position to prove Lemma~\ref{lem:spectrum-free}.
\begin{proof}[Proof of Lemma~\ref{lem:spectrum-free}]

If $\lambda\ne \lambda_\pm(\xi,\eta)$, 
\begin{equation}
  \label{eq:resl0}
(\lambda-\mL_0(\xi,\eta))^{-1}
=\frac{1}{(\lambda-\lambda_+(\xi,\eta))(\lambda-\lambda_-(\xi,\eta))}
\begin{pmatrix}
\lambda-ic\xi & 1 \\ -(\mu S)^2(\xi,\eta) & \lambda-ic\xi
\end{pmatrix}\,.
\end{equation}
Since $2ic\xi=\lambda_++\lambda_-$ and $2i\mu S=\lambda_+-\lambda_-$,
\begin{equation}
  \label{eq:cxi-muS}
\begin{split}
& \frac{2(\lambda-ci\xi)}
{(\lambda-\lambda_+(\xi,\eta))(\lambda-\lambda_-(\xi,\eta))}
=
\frac{1}{\lambda-\lambda_+(\xi,\eta)}
+\frac{1}{\lambda-\lambda_-(\xi,\eta)}\,,
\\ &
\frac{2i\mu(\xi,\eta)S(\xi,\eta)}
{(\lambda-\lambda(\xi,\eta))(\lambda-\lambda_-(\xi,\eta))}
=\frac{1}{\lambda-\lambda_+(\xi,\eta)}
-\frac{1}{\lambda-\lambda_-(\xi,\eta)}\,.
\end{split}
\end{equation}
In view of \eqref{eq:fm}, \eqref{eq:cl-sf2}, \eqref{eq:resl0} and \eqref{eq:cxi-muS},
the operator $\lambda-\mL_0$ has a bounded inverse on $X$ if
\begin{equation}
  \label{eq:resolvent-cond}
\sup_{(\xi,\eta)\in\R^2}|\lambda-\lambda_\pm(\xi+i\a,\eta)|^{-1}<\infty\,.
\end{equation}
Thus we have
\begin{equation}
  \label{eq:L0-ess}
\sigma(\mL_0(D))=\overline{\left\{\lambda_\pm(\xi+i\a,\eta)\mid
(\xi,\eta)\in\R^2\right\}}\,,
\end{equation}
 and Lemma~\ref{lem:spectrum-free} follows immediately from
\eqref{eq:cl-sf4}, \eqref{eq:cl-sf5} and \eqref{eq:L0-ess}.
\end{proof}

To prove the boundedness of $(\lambda-\mL)^{-1}$ restricted on
$\mathcal{Q}(\eta_0)X$ for a small $\eta_0>0$,
the estimate \eqref{eq:cl-sf5} in Claim~\ref{cl:spectrum-free2}
is insufficient.
To have a better estimate on $(\lambda-\lambda_-(D))^{-1}$,
we will estimate $\lambda_-(\xi,\eta)$ in the high frequency regime,
the middle frequency regime and in the low frequency regime, separately.
Let $\delta=\eps^{1/20}$,  $K=\delta^{-3}$ and
\begin{align*}
A_{high}=&\{(\xi,\eta)\in\R^2\mid \text{$|\xi|\ge \delta$
or $|\eta|\ge \delta|\xi+i\a|$}\}\,,\\  
A_{\xi,m}=&\{(\xi,\eta)\in\R^2\mid K\eps\le|\xi|\le \delta,\,
|\eta|\le\delta|\xi+i\a|\}\,,\\  
A_{\eta,m}=&\{(\xi,\eta)\in\R^2\mid |\xi|\le K\eps,\,
K\eps|\xi+i\a|\le |\eta|\le\delta|\xi+i\a|\}\,,\\
A_{low}=&\{(\xi,\eta)\in\R^2\mid |\xi|\le K\eps,\,
|\eta|\le K\eps|\xi+i\a|\}\,,\\  
\widetilde{A}_{low}=&\{(\xi,\eta)\in\R^2\mid |\xi|\le K\eps,\,
|\eta|\le K(K+\hat{\a})\eps^2\}\,.
\end{align*}
Obviously, we have $\R^2=A_{high}\cup A_{\xi,m}\cup A_{\eta,m}\cup A_{low}$
and $A_{low}\subset \widetilde{A}_{low}$.
Suppose $c=\sqrt{1+\eps^2}$ and that $\eps$ is a small positive number.
In the low frequency regime $A_{low}$,
\begin{gather*}
i\mu(D) \sim \eps \pd_{\hat{z}}
+\frac{\eps^3}{2}\pd_{\hat{z}}^{-1}\pd_{\hat{y}}^2\,,
\enskip S(D)\sim I+\frac{b-a}{2}\pd_{\hat{z}}^2\,,\enskip
\lambda_-(D)\sim \eps^3\mL_{KP,0}(D_{\hat{z}},D_{\hat{y}})\,,
\enskip \lambda_+(D) \sim 2\eps \pd_{\hat{z}}\,,
\end{gather*}
where $\hat{z}=\eps z$, $\hat{y}=\eps^2y$ and
$\mL_{KP,0}(D_{\hat{z}},D_{\hat{y}})
=-\frac{1}{2}\{(b-a)\pd_{\hat{z}}^3-\pd_{\hat{z}}+\pd_{\hat{z}}^{-1}\pd_{\hat{y}}^2\}$.
More precisely, we have the following.
\begin{lemma}
  \label{lem:lm-}
Let $c=\sqrt{1+\eps^2}$, $\a=\eps\hat{\a}$ and $\hat{\a}_\eps=1/\sqrt{bc^2-a}$. 
Let $\xi=\eps\hxi$, $\eta=\eps^2\heta$.
Suppose $\hat{\a}\in(0,\hat{\a}_\eps)$.
Then there exist positive constants $\eps_0$ and $C$ such
that for $\eps\in(0,\eps_0)$,
\begin{equation}
  \label{eq:lmm-low}
\lambda_-(\xi+i\a,\eta)=\frac{i\eps^3}{2}(\hxi+i\hat{\a})
\left\{1+(b-a)(\hxi+i\hat{\a})^2-\frac{\heta^2}{(\hxi+i\hat{\a})^2}
+O(K^4\eps^2)\right\}
\end{equation}
for $(\xi,\eta)\in A_{low}$,
\begin{equation}
\label{eq:lmm-low2}
\lambda_-(\xi+i\a,\eta)=\frac{i\eps^3}{2}(\hxi+i\hat{\a})
\left\{1+(b-a)(\hxi+i\hat{\a})^2-\frac{\heta^2}{(\hxi+i\hat{\a})^2}
+O(K^8\eps^2)\right\}
\end{equation}
for $(\xi,\eta)\in \widetilde{A}_{low}$,
\begin{align}
& \label{eq:lm-xi-m}
\Re\lambda_-(\xi+i\a,\eta)\le -\frac{\hat{\a}\eps^3}{4}
\left\{1+(b-a)\hxi^2\right\}
\quad\text{for $(\xi,\eta)\in A_{\xi,m}$,}
\\ &   \label{eq:lm-eta-m}
\Re\lambda_-(\xi+i\a,\eta)\le -\frac{\a\eps^3}{4}
\frac{\heta^2}{\hxi^2+\hat{\a}^2}
\quad\text{for $(\xi,\eta)\in A_{\eta,m}$,}
\\ &   \label{eq:high}
\Re\lambda_-(\xi+i\a,\eta)\le -C\delta^2\eps
\quad\text{for $(\xi,\eta)\in A_{high}$.}
\end{align}
\end{lemma}
\begin{proof}[Proof of Lemma~\ref{lem:lm-}]
If $(\xi,\eta)\in A_{low}$, then
\begin{equation}
  \label{eq:Al}
|\hxi|\le K\,, \quad|\heta|/|\hxi+i\hat{\a}|\le K\,,
\end{equation}
\begin{equation}
  \label{eq:mu-low}
\begin{split}
\mu(\xi+i\a,\eta)=& \eps(\hxi+i\hat{\a})\sqrt{1+\frac{\eps^2\heta^2}
{(\hxi+i\hat{\a})^2}}
\\=& \eps(\hxi+i\hat{\a})
\left\{1+\frac{\eps^2}{2}\frac{\heta^2}{(\hxi+i\hat{\a})^2}
+O(K^4\eps^4)\right\}\,,
\end{split}  
\end{equation}
\begin{equation}
  \label{eq:S-low}
\begin{split}
S(\xi+i\a,\eta)=&\sqrt{
1+\frac{(a-b)\left\{(\xi+i\a)^2+\eta^2\right\}}
{1+b\left\{(\xi+i\a)^2+\eta^2\right\}}}
\\=& 1+\frac{a-b}{2}\eps^2(\hxi+i\hat{\a})^2+O(\eps^4K^4)\,.
\end{split}  
\end{equation}
Combining \eqref{eq:Al}--\eqref{eq:S-low} and the fact that
$c=1+\frac{\eps^2}{2}+O(\eps^4)$, we have
\eqref{eq:lmm-low}.
If $(\xi,\eta)\in \widetilde{A}_{low}$, then
$|\hxi|\le K$ and $|\heta|/|\hxi+i\hat{\a}|\le K(K+\hat{\a})/\hat{\a}$
and we can prove \eqref{eq:lmm-low2} in exactly the same way.
\par
Suppose $(\xi,\eta)\in A_{\xi,m}$.
Then $\xi=O(\delta)$, $\a/\xi=O(K^{-1})$ and $\eta/\xi=O(\delta)$.
By \eqref{eq:cl-sf5'},
\begin{align*}
\Re\lambda_-(\xi+i\a,\eta)\le &
-\a \left\{c-1+\frac{b-a}{2}\frac{\xi^2+\eta^2 -\a^2}
{1+b(\xi^2+\eta^2 -\a^2)}\right\}
\\=& -\a\left\{\frac{\eps^2}{2}+O(\eps^4)
+\frac{b-a}{2}\left(1+O(\delta^2+K^{-2})\right)\xi^2\right\}\,.
\end{align*}
Thus we have \eqref{eq:lm-xi-m} provided $\eps_0$, $\delta$ and $K^{-1}$
are sufficiently small.
\par
Let $(\xi,\eta)\in A_{\eta,m}$.
By \eqref{eq:im-mu}, \eqref{eq:cl-sf2} and \eqref{eq:lambda-},
\begin{equation}
  \label{eq:lambda-est2}
\begin{split}
  \Re\lambda_-(\xi+i\a,\eta)\le & -\a c+
\Im \mu(\xi+i\a,\eta)\Re S(\xi+i\a,\eta)
\\ \le & -c\left\{\a-\Im \mu(\xi+i\a,\eta)\right\}\,.
\end{split}  
\end{equation}
Since
\begin{align*}
\mu(\xi+i\a,\eta)=& 
(\xi+i\a)\sqrt{1+\frac{\eta^2}{(\xi+i\a)^2}}
\\=& \eps(\hxi+i\hat{\a})
\left\{1+\frac{\eps^2\heta^2}{2(\hxi+i\hat{\a})^2}
\left(1+O(\delta^2)\right)\right\}\,,
\end{align*}
\begin{align*}
\Im \mu(\xi+i\a,\eta)= \eps\hat{\a}
-\frac{\eps^3\hat{\a}\heta^2}{2(\hxi^2+\hat{\a}^2)}
\left(1+O(\delta^2)\right)\,.
\end{align*}
By \eqref{eq:lambda-est2} and the above,
we have \eqref{eq:lm-eta-m}  provided $\eps_0$, $\delta$ and $K^{-1}$
are sufficiently small.
\par
Finally, we will prove \eqref{eq:high}.
Suppose $(\xi,\eta)\in A_{high}$ and $|\xi|\ge \delta$.
Then there exists a positive constant $C_1$ such that $\xi^2+\eta^2-\a^2\ge  C_1\delta^2$ and it follows from \eqref{eq:cl-sf5'} that
\begin{align*}
\Re\lambda_-(\xi+i\a,\eta)\le &
-\a\left\{c-1+\frac{b-a}{2}\frac{C_1\delta^2}{1+bC_1\delta^2}\right\}
 \lesssim  -\eps\delta^2\,.  
\end{align*}
\par
Suppose $(\xi,\eta)\in A_{high}$ and $|\eta||\xi+i\a|^{-1}\ge \delta$.
By \eqref{eq:im-mu} and \eqref{eq:g2_s},
\begin{equation}
  \label{eq:lambda-est2'}
\Im\mu(\xi+i\a,\eta)=\gamma_2(\xi,s)\le \gamma_2(\xi,\delta^2|\xi+i\a|^2)
\quad\text{if $s=\eta^2\ge \delta^2|\xi+i\a|^2$.}
\end{equation}
If $0\le s \le \delta^2|\xi+i\a|^2$,
$$\gamma_1^2+\gamma_2^2=\left|(\xi+i\a)^2+s\right|\le (1+\delta^2)|\xi+i\a|^2,
$$
and it follows from \eqref{eq:g2_s} that for a $C>0$,
\begin{equation}
  \label{eq:lambda-est2''}
\gamma_2(\xi,\delta^2|\xi+i\a|^2)\le \gamma_2(\xi,0)
\exp\left(-\delta^2/2(1+\delta^2)\right)\le \alpha-C\delta^2\,.
\end{equation}
Substituting \eqref{eq:lambda-est2'} and \eqref{eq:lambda-est2''}
into  \eqref{eq:lambda-est2}, we have \eqref{eq:high}.
Thus we complete the proof.
\end{proof}

Finally, we will estimate operator norms of
$(\lambda-\lambda_\pm(D))^{-1}$ on $L^2(\R^2_\a)$ and its subspaces.
Let $\rho_y$ and $\tilde{\rho}_y$ be functions on $\R$ such that
$\rho_y(\eta)+\tilde{\rho}_y(\eta)=1$ for $\eta\in\R$ and
$$\rho_y(\eta)=
\begin{cases}
1 & \text{ if $|\eta|\le K(K+\hat{\a})\eps^2$,}\\ 
0 & \text{ if $|\eta|\ge K(K+\hat{\a})\eps^2$.}
\end{cases}$$
Let $\rho_z(\xi)$ be the characteristic function
of $\{\xi\in \C \mid |\Re\xi| \le K\eps\}$,
$\tilde{\rho}_z(\xi)=1-\rho_z(\xi)$ and
\begin{gather*}
Y:=\rho_y(D_y)L^2_\a(\R^2)\,,\quad Y_{low}:=\rho_z(D_z)Y\,,
\quad Y_{high}:=\tilde{\rho}_z(D_z)Y,.
\end{gather*}
We remark that $\widetilde{A}_{low}=\supp\rho_z(\xi)\rho_y(\eta)$.
\begin{lemma}
\label{cl:inv-lambda}
Let $c$, $\a$, and $\hat{\a}$ be as in Lemma~\ref{lem:lm-}.
Let $\hat{\beta}\in(0,\frac{\hat{\a}}{8})$ and $\lambda\in\Omega_\eps:=\{\lambda \in\C
\mid \Re\lambda\ge -\hat{\beta}\eps^3\}$.
Then there exist positive constants $C$ and $\eps_0$ such that if $\eps\in(0,\eps_0)$ and
$\lambda\in \Omega_\eps$,
\begin{gather}
\label{eq:l+inv}
\|(\lambda-\lambda_+(D))^{-1}\|_{B(L^2_\a)}\le C\eps^{-1}\,,
\\
\label{eq:l-inv}
\|(\lambda-\lambda_-(D))^{-1}\|_{B(L^2_\alpha)}\le C\eps^{-3}\,,
\\ 
\label{eq:l-inv2}
\|(\lambda-\lambda_-(D))^{-1}\|_{B(Y_{high})}  \le CK^{-2}\eps^{-3}\,,
\\
\label{eq:mu-res2'}
\|B^{-1}\mu(D)(\lambda-\lambda_-(D))^{-1}\|_{B(Y_{high})}
+\|B^{-1}\pd_z(\lambda-\lambda_-(D))^{-1}\|_{B(Y_{high})}
\le CK^{-1}\eps^{-2}\,,
\\
\label{eq:mu-res3}
\|B^{-1}\mu(D)(\lambda-\lambda_-(D))^{-1}\|_{B(Y_{low})}
+\|B^{-1}\pd_z(\lambda-\lambda_-(D))^{-1}\|_{B(Y_{low})} \le C\eps^{-2}\,.
\end{gather}
\end{lemma}
\begin{proof}
By \eqref{eq:cl-sf4} and \eqref{eq:cl-sf5},
\begin{equation}
  \label{eq:lm+est}
  \begin{split}
\inf_{\lambda\in \Omega_\eps\,,\,(\xi,\eta)\in\R^2}|\lambda-\lambda_+(\xi+i\a,\eta)|
\ge & 
\inf_{\lambda\in \Omega_\eps\,,\,(\xi,\eta)\in\R^2}
\Re(\lambda-\lambda_+(\xi+i\a,\eta))
\gtrsim  \eps\,,    
  \end{split}
\end{equation}
\begin{align*}
\inf_{\lambda\in \Omega_\eps\,,\,(\xi,\eta)\in\R^2}
|\lambda-\lambda_-(\xi+i\alpha,\eta)|
\ge &  
\inf_{\lambda\in \Omega_\eps\,,\,(\xi,\eta)\in\R^2}
\left(\Re\lambda +\frac{\alpha}{2}(c-1)\right)
\gtrsim  \eps^3\,.
\end{align*}
Hence it follows from \eqref{eq:fm0} that
\begin{align*}
\|(\lambda-\lambda_+(D))^{-1}\|_{B(L^2_\a)}= &
\sup_{(\xi,\eta)\in\R^2}\frac{1}{|\lambda-\lambda_+(\xi+i\a,\eta)|}
\le C\eps^{-1}\,,
\\
\|(\lambda-\lambda_-(D))^{-1}\|_{B(L^2_\a)}= &
\sup_{(\xi,\eta)\in\R^2}\frac{1}{|\lambda-\lambda_-(\xi+i\a,\eta)|}
 \le C\eps^{-3}\,,
\end{align*}
where $C$ is a positive constants that does not depend on $\eps\in(0,\eps_0)$
and $\lambda\in\Omega_\eps$.
\par
Next, we will show \eqref{eq:l-inv2}. 
Suppose $f\in Y_{high}$. Then
$\supp \hat{f}(\xi+i\a,\eta)\subset \widetilde{A}_{low}^c
 \subset A_{\xi,m}\cup A_{\eta,m}\cup A_{high}$. 
By Lemma~\ref{lem:lm-},
\begin{equation}
\label{eq:lm-est2}
\inf_{\lambda\in \Omega_\eps\,,\,(\xi,\eta)\in A_{\xi,m} \cup A_{\eta,m}\cup A_{high}}
|\lambda-\lambda_-(\xi+i\a,\eta)|\gtrsim K^2\eps^3\,.
\end{equation}
Hence it follows from \eqref{eq:fm0} that
\begin{align*}
  \|(\lambda-\lambda_-(D))^{-1}f\|_{L^2_\a(\R^2)}
\le & \sup_{(\xi,\eta)\not\in \widetilde{A}_{low}}
\frac{1}{|\lambda-\lambda_{-}(\xi+i\a,\eta)|}
\left(\int_{\R^2}|\hat{f}(\xi+i\a,\eta)|^2\,d\xi d\eta\right)^{1/2}
\\ \lesssim & K^{-2}\eps^{-3}\|f\|_{L^2_\a(\R^2)}\,.
\end{align*}
\par
Next, we will prove \eqref{eq:mu-res2'}.
By \eqref{eq:high},
\begin{equation}
  \label{eq:mu-res1a}
\sup_{\lambda\in \Omega_\eps,(\xi,\eta)\in A_{high}}
\frac{|\xi+i\a|+|\mu(\xi+i\a,\eta)|}
{|B(\xi+i\a,\eta)|^{1/2}|\lambda-\lambda_-(\xi+i\a,\eta)|} \lesssim \delta^{-2}\eps^{-1}\,,  
\end{equation}
where $B(\xi,\eta)=1+b(\xi^2+\eta^2)$.
By \eqref{eq:lm-xi-m} and the definition of $A_{\xi,m}$,
\begin{equation}
  \label{eq:mu-res1b}
\frac{|\xi+i\a|+|\mu(\xi+i\a,\eta)|}{|\lambda-\lambda_-(\xi+i\a,\eta)|}
\lesssim  \frac{\sqrt{\xi^2+\eta^2}}{\eps\xi^2}
\lesssim \frac{1}{K\eps^2}
\quad\text{for $(\xi,\eta)\in A_{\xi,m}$ and $\lambda\in\Omega_\eps$.}  
\end{equation}
By \eqref{eq:lm-eta-m} and the fact that 
$|\xi+i\a|+|\mu(\xi+i\a,\eta)|\lesssim  K\eps$ for $(\xi,\eta)\in A_{\eta,m}$,
\begin{equation}
  \label{eq:mu-res1c}
\frac{|\xi+i\a|+|\mu(\xi+i\a,\eta)|}{|\lambda-\lambda_-(\xi+i\a,\eta)|}
\lesssim K|\xi+i\a|^2\eta^{-2} \lesssim  \frac{1}{K\eps^2}
\quad\text{for $(\xi,\eta)\in A_{\eta,m}$ and $\lambda\in\Omega_\eps$.}  
\end{equation}
Combining \eqref{eq:mu-res1a}--\eqref{eq:mu-res1c} with
\begin{equation}
  \label{eq:lowB}
|B(\xi+i\a,\eta)|\ge 1-b\a^2>0\,,
\end{equation}
we have \eqref{eq:mu-res2'}.
\par
Finally, we will prove \eqref{eq:mu-res3}.
By \eqref{eq:lmm-low2},  we have for $\lambda\in \Omega_\eps$ and
$(\xi,\eta)\in \widetilde{A}_{low}$,
\begin{align*}
&  |\lambda-\lambda_-(\xi+i\a,\eta)|
\\ \ge & \eps^3\left[
-\hat{\beta}+\frac{\hat{\a}}{2}\{
1-(b-a)\hat{\a}^2+3(b-a)\hxi^2+\frac{\heta^2}{\hxi^2+\hat{\a}^2}\}
\right]+O(K^9\eps^5)
\\ \gtrsim & \eps^3(1+\hxi^2)\,.
\end{align*}
%Here we use $K^9\eps^2\ll1$.
\begin{align*}
\sup_{\lambda\in\Omega_\eps\,,\,(\xi,\eta)\in \widetilde{A}_{low}}
\frac{|\xi+i\a|}{|\lambda-\lambda_-(\xi+i\a,\eta)|}
\lesssim & 
\sup_{\hxi\in[0,K]}\frac{|\hxi|+\hat{\a}}{\eps^2(1+\hxi^2)}=O(\eps^{-2})\,.
\end{align*}
Thus we complete the proof.
\end{proof}
\bigskip

\section{Spectral stability  for small line solitary waves}
\label{sec:spectrum}
In this section, we will prove Theorem~\ref{thm:spectrum}.
For small line solitary waves, the spectrum of the linearized operator
$\mL$ is well approximated by that of $\mL_{KP}$ in the low frequency regime,
while the spectrum of $\mL$ is close to that of the free operator $\mL_0$
in the high-frequency regime. 
We will show that any spectrum of $\mL$ locates
in the stable half plane and is bounded away from the imaginary axis
except for the continuous eigenvalues $\{\lambda_\eps(\eta)\}$.
More precisely, we will prove
\begin{equation}
  \label{eq:ev1-est}
\sup_{\lambda\in\Omega_\eps}\|(\lambda-\mL)^{-1}\mathcal{Q}(\eps^2\eta_0)\|_{B(X)}
<\infty\,.  
\end{equation}
Since the potential part of $\mL$ is independent of $y$, we can
estimate the high frequency part in $y$ and the low frequency part in $y$,
separately.
\subsection{Spectral stability  for high frequencies in $y$}
\label{subsec:high}
First, we will estimate solutions of the resolvent equation
\begin{equation}
  \label{eq:ev1}
  (\lambda-\mL)u=f
\end{equation}
for $f\in \tilde{\rho}_y(D_y)X$.
In the high frequency regime in $y$, the potential term 
$V$ is relatively small compared with $\lambda-\mL_0$.
 
\begin{lemma}
\label{lem:high frequencies}
Let $c$, $\a$, $\hat{\a}$ and $\Omega_\eps$ be as in Lemma~\ref{cl:inv-lambda}.
There exists a positive number $\eps_0$ such that if $\eps\in(0,\eps_0)$ and $\lambda \in\Omega_\eps$,
then 
\begin{equation*}
\sup_{\lambda\in \Omega_\eps}\|\tilde{\rho}_y(D_y)(\lambda-\mL)^{-1}\tilde{\rho}_y(D_y)\|_{B(X)}<\infty\,.
\end{equation*}
\end{lemma}

\begin{proof}[Proof of Lemma~\ref{lem:high frequencies}]
In view of Lemma~\ref{lem:spectrum-free} and the second resolvent formula 
\begin{equation}
  \label{eq:res-formula}
(\lambda-\mL)^{-1}=\{I-(\lambda-\mL_0)^{-1}V\}^{-1}(\lambda-\mL_0)^{-1}\,,
\end{equation}
it suffices to show that
\begin{equation}
  \label{eq:i-rv}
\sup_{\lambda\in \Omega_\eps}\left \|\tilde{\rho}_y(D_y)
(I-(\lambda-\mL_0)^{-1}V)^{-1} \tilde{\rho}_y(D_y)\right\|_{B(X)}<\infty\,.
\end{equation}
By \eqref{eq:resl0},
\begin{align*}
(\lambda-\mL_0)^{-1}V =&
-(\lambda-\lambda_+(D))^{-1}(\lambda-\lambda_-(D))^{-1}B^{-1}
\begin{pmatrix}
v_1 & v_2 \\ (\lambda-c\pd_z)v_1 & (\lambda-c\pd_z)v_2
\end{pmatrix}
\\=:& \begin{pmatrix}
r_{11}(\lambda) & r_{12}(\lambda)\\
r_{21}(\lambda) & r_{22}(\lambda)      
\end{pmatrix}\,.
\end{align*}
\par
First, we will show \eqref{eq:i-rv} admitting
\begin{equation}
  \label{eq:rij-bound}
  \begin{split}
& \sup_{\lambda\in\Omega_\eps}
(\|\tilde{\rho}_y(D_y)r_{11}(\lambda)\tilde{\rho}_y(D_y)\|_{B(H^1_\a)}
+\|\tilde{\rho}_y(D_y)r_{22}(\lambda)\tilde{\rho}_y(D_y)\|_{B(L^2_\a)})
=O(K^{-1})\,,
\\ &
\sup_{\lambda\in\Omega_\eps}\|\tilde{\rho}_y(D_y)r_{12}(\lambda)
\tilde{\rho}_y(D_y)\|_{B(L^2_\a,H^1_\a)}
=O(K^{-1}\eps^{-1}+\delta^{-2})\,,
\\ &
\sup_{\lambda\in\Omega_\eps} 
\|\tilde{\rho}_y(D_y)r_{21}(\lambda)\tilde{\rho}_y(D_y)\|_{B(H^1_\a,L^2_\a)}
=O(\eps\delta^{-2})\,.
  \end{split}
\end{equation}
Let 
\begin{gather*}
B_1(\lambda)=
\begin{pmatrix}
I-r_{11}(\lambda) & -r_{12}(\lambda) \\
O & I-r_{22}(\lambda)
\end{pmatrix}
\,,  \\
B_2(\lambda)=\begin{pmatrix}
I-(I-r_{11}(\lambda))^{-1}r_{12}(\lambda)
(I-r_{22}(\lambda))^{-1}r_{21}(\lambda) & 0 \\
-(I-r_{22}(\lambda))^{-1}r_{21}(\lambda) & I
\end{pmatrix}\,.
\end{gather*}
Then $I-(\lambda-\mL_0)^{-1}V=B_1(\lambda)B_2(\lambda)$.
We see from \eqref{eq:rij-bound} that
$I-r_{ii}(\lambda)$ ($i=1,2$) have bounded inverse and that
\begin{align*}
& \|\tilde{\rho}_y(D_y)r_{12}(\lambda)\tilde{\rho}_y(D_y)\|_{B(L^2_\a,H^1_\a)}
\|\tilde{\rho}_y(D_y)r_{21}(\lambda)\tilde{\rho}_y(D_y)\|_{B(H^1_\a,L^2_\a)}
\\=& O(K^{-1}\delta^{-2}+\eps\delta^{-4})=O(\eps^{1/20})\,.
\end{align*}
Thus we have
$$\sup_{\lambda\in\Omega_\eps}(
\|\tilde{\rho}_y(D_y)B_1(\lambda)^{-1}\tilde{\rho}_y(D_y)
\|_{B(X)}
+\|\tilde{\rho}_y(D_y)B_1(\lambda)^{-1}\tilde{\rho}_y(D_y)
\|_{B(X)})<\infty\,,$$
and
$\tilde{\rho}_y(D_y)(I-(\lambda-\mL_0)^{-1}V)^{-1}\tilde{\rho}_y(D_y)
\in L^\infty(\Omega_\eps;B(X))$.
\par
Now we will start to show \eqref{eq:rij-bound}.
By \eqref{eq:cxi-muS},
\begin{align}
& \label{eq:hf-pf2}
\Delta(\lambda-\lambda_+(D))^{-1}(\lambda-\lambda_-(D))^{-1}
=\frac{i\mu(D)}{2S(D)}\left\{
(\lambda-\lambda_+(D))^{-1}-(\lambda-\lambda_-(D))^{-1}\right\}\,,
\\ & \label{eq:hf-pf1}
(\lambda-c\pd_z)(\lambda-\lambda_+(D))^{-1}(\lambda-\lambda_-(D))^{-1}
=\frac{1}{2}\left\{
(\lambda-\lambda_+(D))^{-1}+(\lambda-\lambda_-(D))^{-1}\right\}\,.  
\end{align}
If $|\eta|\ge K(K+\hat{\a})\eps^2$, then $(\xi,\eta)\in \widetilde{A}_{low}^c
\subset A_{high}\cup A_{\xi,m}\cup A_{\eta,m}$.
Since
\begin{equation}
  \label{eq:v1c-alt}
v_{1,c}=cq_c''-c\Delta(q_c\cdot)\,,  
\end{equation}
it follows from \eqref{eq:hf-pf2}, \eqref{eq:S} and Claim~\ref{cl:qc-size}
in Appendix~\ref{sec:ap1} that
\begin{align*}
\left\|\tilde{\rho}_y(D_y)
(\lambda-\lambda_+(D))^{-1}(\lambda-\lambda_-(D))^{-1}B^{-1}v_{1,c}
 \right\|_{B(H^1_\a)} \lesssim I_1+I_2\,,
\end{align*}
where 
\begin{align*}
I_1=&  \eps^4 \left\|\tilde{\rho}_y(D_y)(\lambda-\lambda_+(D))^{-1}
(\lambda-\lambda_-(D))^{-1}B^{-1} \right\|_{B(L^2_\a)}\,,\\
I_2=&  \eps^2 \sum_{\pm}\left\|
\tilde{\rho}_y(D_y)(\lambda-\lambda_\pm(D))^{-1}B^{-1}
\mu(D) \right\|_{B(L^2_\a)}\,.
\end{align*}
By \eqref{eq:lm+est}, \eqref{eq:lm-est2} and \eqref{eq:lowB},
\begin{equation*}
I_1= \eps^4\sup_{(\xi,\eta)\not\in \widetilde{A}_{low}}
\frac{|B(\xi+i\a,\eta)|^{-1}}{\left|\lambda-\lambda_+(\xi+i\a,\eta)\right|
\left|\lambda-\lambda_-(\xi+i\a,\eta)\right|}=O(K^{-2})\,.
\end{equation*}
By  \eqref{eq:lm+est}, \eqref{eq:mu-res1a}--\eqref{eq:mu-res1c}
and Claim~\ref{cl:prop-B,S},
\begin{align*}
 I_2 \lesssim & \eps+ \eps^2\sup_{(\xi,\eta)\not\in \widetilde{A}_{low}}
\frac{|\mu(\xi+i\a,\eta)|}{|B(\xi+i\a,\eta)|}
\left|\lambda-\lambda_-(\xi+i\a,\eta)\right|^{-1}
 \lesssim  K^{-1}\,.
\end{align*}
Thus we prove 
$$\sup_{\lambda\in \Omega_\eps}
\|\tilde{\rho}_y(D_y)r_{11}(\lambda)\tilde{\rho}_y(D_y)\|_{B(H^1_\a)}
\lesssim I_1+I_2=O(K^{-1})\,.$$
\par
Next, we will estimate $\tilde{\rho}_y(D_y)r_{12}(\lambda)\tilde{\rho}_y(D_y)$.
Since 
\begin{equation}
  \label{eq:v2c-alt}
v_{2,c}=2\pd_z(q_c\cdot)-q_c'\,,  
\end{equation}
we have from Claim~\ref{cl:qc-size}
$$ \|\tilde{\rho}_y(D_y)r_{12}(\lambda)\tilde{\rho}_y(D_y)\|_{B(L^2_\a,H^1_\a)}
\lesssim I_3+I_4\,,$$ where
\begin{align*}
I_3=&  \eps^2 \|\tilde{\rho}_y(D_y)
(\lambda-\lambda_+(D))^{-1}(\lambda-\lambda_-(D))^{-1}\pd_zB^{-1}\|_{B(L^2_\a,H^1_\a)}\,,\\
I_4=&  \eps^3 \|
\tilde{\rho}_y(D_y)(\lambda-\lambda_+(D))^{-1}(\lambda-\lambda_-(D))^{-1}B^{-1}\|_{B(L^2_\a,H^1_\a)}\,.
\end{align*}
By \eqref{eq:fm0},
\begin{align*}
I_3\lesssim &  \sup_{(\xi,\eta)\not\in \widetilde{A}_{low}}\frac{\eps^2|\xi+i\a|}
{|\lambda-\lambda_+(\xi+i\a,\eta)||\lambda-\lambda_-(\xi+\a,\eta)||B(\xi+i\a,\eta)|^{1/2}}\,,
\\
I_4\lesssim & \sup_{(\xi,\eta)\not\in \widetilde{A}_{low}}\frac{\eps^3}
{|\lambda-\lambda_+(\xi+i\a,\eta)||\lambda-\lambda_-(\xi+\a,\eta)|}\,.
\end{align*}
It follows from  \eqref{eq:lm+est} and \eqref{eq:mu-res1a}--\eqref{eq:lowB} that
\begin{align*}
 I_3 \lesssim & 
\sup_{(\xi,\eta)\in A_{high}}
\frac{\eps^2|\xi+i\a|}
{|\lambda-\lambda_+(\xi+i\a,\eta)||\lambda-\lambda_-(\xi+\a,\eta)||B(\xi+i\a,\eta)|^{1/2}}
\\ & +\sup_{(\xi,\eta)\in A_{\xi,m}\cup A_{\eta,m}}
\frac{\eps^2|\xi+i\a|}
{|\lambda-\lambda_+(\xi+i\a,\eta)||\lambda-\lambda_-(\xi+\a,\eta)|}
\\ \lesssim & \delta^{-2}+K^{-1}\eps^{-1}\,.
\end{align*}
By \eqref{eq:lm+est} and \eqref{eq:lm-est2},
$$I_4=O(K^{-2}\eps^{-1})\,.$$
Thus we prove
$$\sup_{\lambda\in\Omega_\eps}
\|\tilde{\rho}_y(D_y)r_{12}(\lambda)\tilde{\rho}_y(D_y)\|_{B(L^2_\a)}
\lesssim K^{-1}\eps^{-1}+\delta^{-2}\,.
$$
\par
Using \eqref{eq:hf-pf1},
we can estimate $r_{21}$ and $r_{22}$ in exactly the same way.
Thus we complete the proof.
\end{proof}

\subsection{Spectral stability  for low  frequencies in $y$}
Now we will estimate solutions of \eqref{eq:ev1} for
$f\in \rho_y(D_y)X$ satisfying the orthogonality condition
\begin{equation}
  \label{eq:orth-f}
\int_\R (\mF_yf)(x,\eta)\cdot\overline{g_k^*(x,\eta)}\,dx=0
\quad\text{for $\eta\in [-\eps^2\eta_0,\eps^2\eta_0]$ and $k=1$, $2$.}
\end{equation}
\par
Let $\tf=(\tf_1,\tf_2)$ and $f=(f_1,f_2)=P(D)\tf$.  To begin with, We
will show that \eqref{eq:orth-f} is reduced to the secular term
condition that $\tf_2$ does not include the resonant modes of the
linearized KP-II operator $\mL_{KP}$ in the limit $\eps\to0$. 
\par
Let $E_\eps: L^2_\a(\R^2)\to L^2_{\hat{\a}}(\R^2)$ be an isomorphism defined by 
$(E_\eps f)(x,y):=\eps^{-3/2}f(\eps^{-1}x,\eps^{-2}y)$ and let
\begin{gather*}
Z=\{f\in \rho_y(D_y)X \mid \mathcal{P}(\eps^2\eta_0)f=0\}\,, \quad
\widetilde{Z}=P(D)^{-1}Z\,,
\\  \overline{Z}=\{(\barf_1,\barf_2)\in Y\times Y
\mid \mathcal{P}_{KP}(\eps^2\eta_0)E_\eps\rho_z(D_z)\barf_2=0\}\,.
\end{gather*}
Note that 
$P(D): Y\times Y \to \rho_y(D_y)X$ is isomorphic for small $\eps>0$ because 
$|\mu(\xi+i\a,\eta)|$ is bounded away from $0$ for $\eta\in \supp \rho_y$.
Let  $\overline{P}(\eta_0)$ be the projection on $L^2(\R^2;\C^2)$ defined by
$$ \overline{P}(\eta_0)\begin{pmatrix}  \tu_1 \\ \tu_2\end{pmatrix}
=
\begin{pmatrix}
 0 \\ \rho_z(D_z)E_\eps^{-1}\mathcal{P}_{KP}(\eta_0)E_\eps\rho_z(D_z)\tu_2
\end{pmatrix}\,.$$
The subspaces $\widetilde{Z}$ and $\overline{Z}$ are isomorphic
provided $\eps$ is small.
\begin{lemma}
  \label{lem:orthogonality}
Let  $\eps_0$ and $\eta_0$ be sufficiently small positive numbers.
Then for $\eps\in(0,\eps_0)$, there exists an operator
$\Pi:\widetilde{Z}\to  \overline{Z}$
such that
\begin{equation*}
\|\Pi-I\|_{B(\widetilde{Z},\overline{Z})}
+\|\Pi^{-1}-I\|_{B(\overline{Z},\widetilde{Z})}
=O(K^{-1})\,.  
\end{equation*}
\end{lemma}
Let 
$W_1=H^1_\a(\R)\times L^2_\a(\R)$, $ W_0=L^2_\a(\R;\C^2)$, 
$W_0^*=L^2_{-\a}(\R;\C^2)$ and $W_1^*=H^{-1}_{-\a}(\R)\times L^2_{-\a}(\R)$.
To prove Lemma~\ref{lem:orthogonality}, we need the following.
\begin{claim}
  \label{cl:P-low}
Let $\hat{\a}\in(0,\hat{\a}_0)$, $\a=\hat{\a}\eps$ and let $\eps_0$ and $\eta_0$ be sufficiently small positive numbers.
If $\eps\in (0,\eps_0)$ and $\eta\in[-\eta_0,\eta_0]$, then
\begin{gather*}
\left\|P(D_z,\eps^2\eta)^{-1}
-\frac12
\begin{pmatrix}
\pd_z & S^{-1}(D_z,\eps^2\eta) \\ -\pd_z  & S^{-1}(D_z,\eps^2\eta)
\end{pmatrix}\right\|_{B(W_1,W_0)}=O(\eps^2\eta^2)\,,
\\
\left\|P^*(D_z,\eps^2\eta)-\begin{pmatrix}
(\pd_z^*)^{-1} & \overline{S}^{-1}(D_z,\eps^2\eta) \\ -(\pd_z^*)^{-1} & \overline{S}^{-1}(D_z,\eps^2\eta)
\end{pmatrix}
\right\|_{B(W_1^*,W_0^*)}=O(\eps^2\eta^2)\,,  
\end{gather*}
where $(\pd_z^*)^{-1}f(z)=-\int_{-\infty}^z f(z_1)\,dz_1$.
\end{claim}
\begin{proof}
In view of \eqref{eq:inv-pd}, \eqref{eq:pdon3} and their proofs,
\begin{equation}
  \label{eq:inv-mu2}
\left\|i\mu(D_z,\eps^2\eta)-\pd_z\right\|_{B(W_1)}
+\left\|i\bar{\mu}(D_z,\eps^2\eta)^{-1}-(\pd_z^*)^{-1}\right\|_{B(W_1^*)}
=O(\eps^2\eta^2)\,.
\end{equation}
Since
$$P(\xi,\eta)^{-1}=\frac12
\begin{pmatrix}
i\mu(\xi,\eta) & S(\xi,\eta)^{-1} \\ -i\mu(\xi,\eta) & S(\xi,\eta)^{-1}  
\end{pmatrix}
\,, \quad
P^*(\xi,\eta)=
\begin{pmatrix}
i\overline{\mu(\xi,\eta)}^{-1} & \overline{S(\xi,\eta)}\\
-i\overline{\mu(\xi,\eta)}^{-1} & \overline{S(\xi,\eta)}
\end{pmatrix}\,,
$$
Claim~\ref{cl:P-low} follows from \eqref{eq:inv-mu2}.
\end{proof}

\begin{proof}[Proof of Lemma~\ref{lem:orthogonality}]
Let $\Pi \tu=\tu-\overline{P}(\eta_0)\tu$ for $\tu\in \widetilde{Z}$. 
To prove Lemma~\ref{lem:orthogonality}, it suffices to show
  \begin{equation}
\label{eq:Proj-KPapprox}
\left\|P(D)\mathcal{P}(\eps^2\eta_0)P(D)^{-1}-\overline{P}(\eta_0)
\right\|_Y=O(K^{-1})\,.
  \end{equation}
See e.g. \cite[Chapter I, Section~4.6]{Kato}.
\par
First, we will show
\begin{equation}
  \label{eq:orth-lowfr}
 \|\mathcal{P}(\eps^2\eta_0)\tilde{\rho}_z(D_z)f\|_{L^2_\a(\R^2)}+ \|\tilde{\rho}_z(D_z)\mathcal{P}(\eps^2\eta_0)f\|_{L^2_\a(\R^2)}
\lesssim K^{-1}\|f\|_{L^2_\a(\R^2)}\,. 
\end{equation}
Let
$$\tc_k(\eta)=\int \tilde{\rho}_z(D_z)\mF_yf(z,\eta)\cdot g_k^*(z,\eta)\,dz\,.$$
Then 
$$\mathcal{P}(\eps^2\eta_0)\tilde{\rho}_z(D_z)f=
\frac{1}{\sqrt{2\pi}}\sum_{k=1,2}\int_{-\eps^2\eta_0}^{\eps^2\eta_0}
\tc_k(\eta)g_k(z,\eta)e^{iy\eta}\,d\eta\,.$$
Since $\|\pd_z^{-1}\tilde{\rho}_z(D_z)\|_{B(L^2_\a(\R))}\le (K\eps)^{-1}$ and
$\sup_{\eta\in[-\eps^2\eta_0,\eps^2\eta_0]}\|\pd_zg_k^*(\cdot,\eta)\|_{L^2_{-\a}(\R)}=O(\eps)$
by Corollary~\ref{lem:resonance2},
\begin{align*}
 |\tc_k(\eta)|=& 
\left| \int \pd_z^{-1}\tilde{\rho}_z(D_z)(\mF_yf)(z,\eta)\cdot
     \pd_zg_k^*(z,\eta)\,dz\right|
\\ \lesssim & K^{-1}\|\mF_yf(\cdot,\eta)\|_{L^2_\a(\R_z)}\,.
\end{align*}
Hence it follows from the Plancherel theorem  and the above that
\begin{align*}
\|\mathcal{P}(\eps^2\eta_0)\tilde{\rho}_z(D_z)f\|_{L^2_\a(\R^2)} 
\lesssim &  
\sum_{k=1,2} \left\|\|c_k(\eta)g_k(x,\eta)\|_{L^2_\a(\R_z)}\right\|_{L^2(-\eps^2\eta_0\le \eta \le\eps^2\eta_0)}
\\ \lesssim & K^{-1}\|f\|_{L^2_\a(\R^2)}\,.
\end{align*}
Similarly, we have $\|\tilde{\rho}_z(D_z)\mathcal{P}(\eps^2\eta_0)f\|_{L^2_\a(\R)^2}\lesssim K^{-1}\|f\|_{L^2_\a(\R^2)}$.
Thus we prove \eqref{eq:orth-lowfr}.
\par
Next, we will show $\rho_z(D_z)P(\eps^2\eta_0)\rho_z(D_z)\simeq\rho_z(D_z)E_\eps^{-1}P_{KP}(\eta_0)E_\eps\rho_z(D_z)$.
By the fact that  $\rho_z(D_z)$ is bounded on $L^2_\a(\R^2)$
and $\|f(\cdot)\|_{L^2_\a(\R)}=\eps^{-1/2}\|f(\eps^{-1}\cdot)\|_{L^2_{\hat{\a}(\R)}}$,
\begin{equation}
  \label{eq:Proj-approx1}
\begin{split}
&  \left\|\rho_z(\eps D_z)\left\{
\eps^{-1}P(\eps D_z,\eps^2\eta)^{-1}g_k(\eps^{-1}\cdot,\eps^2\eta)
- \begin{pmatrix} 0 \\ g_{0,k}(\cdot,\eta)\end{pmatrix}
\right\}\right\|_{L^2_{\hat{\a}}(\R)}
\\ \le  & II_1+II_2+II_3+II_4=O(K^{-2}\eps+\eta^2)\,,
\end{split}  
\end{equation}
where
\begin{align*}
II_1=& \eps^{-3/2}\left\|\left\{P(\eps D_z,\eps^2\eta)^{-1}
-\frac12
\begin{pmatrix}
\pd_z & S^{-1}(D_z,\eps^2\eta) \\ -\pd_z & S^{-1}(D_z,\eps^2\eta)
\end{pmatrix}\right\}g_k(\cdot,\eps^2\eta)\right\|_{L^2_\a}\,,
\\
II_2=& \frac12\eps^{-3/2}\left\|(S^{-1}(D_z,\eps^2\eta)-I)
\begin{pmatrix}0 & 1 \\ 0 & 1\end{pmatrix}
g_k(\cdot,\eps^2\eta)\right\|_{L^2_\a}\,,
\\
II_3=& \eps^{-1}\left\|
\frac12\begin{pmatrix}\pd_z & I \\ -\pd_z & I\end{pmatrix}
g_k\bigl(\eps^{-1}\cdot,\eps^2\eta\bigr)
-
\begin{pmatrix} 0 \\ g_{0,k}(\cdot,0) \end{pmatrix}
\right\|_{L^2_{\hat{\a}}(\R)}\,,
\\ 
II_4= & \|g_{0,k}(\cdot,\eta)-g_{0,k}(\cdot,0)\|_{L^2_{\hat{\a}}(\R)}\,.
\end{align*}
Indeed, it follows from Corollary~\ref{lem:resonance2} that
$II_3=O(\eps^2+\eta^2)$ and that for $k=1$, $2$,
$$\|g_k(\cdot,\eps^2\eta)\|_{L^2_\a(\R)}=O(\eps^{-1/2})\,,
\quad \left\|\begin{pmatrix} 0 & 1 \\ 0 & 1\end{pmatrix}
g_k(\cdot,\eta)\right\|_{L^2_\a(\R)}=O(\eps^{1/2})\,.$$
Combining the above with Claim~\ref{cl:P-low} and \eqref{eq:binv2},
we have $II_1=O(\eta^2)$ and $II_2=O(K^{-2}\eps)$ and
we have $II_4=O(\eta^2)$ from \eqref{eq:g0eta-0}.
We can prove 
\begin{equation}
\label{eq:Proj-approx2}
 \left\|\rho_z(\eps D_z)\left\{
P^*(\eps D_z,\eps^2\eta)g_k^*(\eps^{-1}\cdot,\eps^2\eta)
- \begin{pmatrix}
0 \\ g_{0,k}^*(\cdot,\eta)
  \end{pmatrix}\right\}\right\|_{L^2_{-\hat{\a}}(\R)} 
= O(\eta^2+K^{-2}\eps)
\end{equation}
in the same way.
\par
Since 
\begin{align*}
& P(D_z,\eta)^{-1}\mathcal{P}(\eps^2\eta_0)f\\
=&\sum_{k=1,2}
(2\pi)^{-1/2} \int_{-\eps^2\eta_0}^{\eps^2\eta_0} \la \mF_y\tf(\cdot,\eta), P(D_z,\eta)^*g_k^*(\cdot,\eta)\ra
P(D_z,\eta)^{-1}g_k(x,\eta) e^{iy\eta}\,d\eta
\end{align*}
for $f=P(D)\tf$, we have from \eqref{eq:Proj-approx1}
and \eqref{eq:Proj-approx2} that
\begin{equation}
  \label{eq:Proj-KPapprox2}
\begin{split}
& \left\|\rho_z(D_z)\left\{P(D)^{-1} \mathcal{P}(\eps^2\eta_0)P(D)
-\begin{pmatrix}0 & 0 \\
0 & \rho_z(D_z)E_\eps^{-1}\mathcal{P}_{KP}(\eta_0)E_\eps\rho_z(D_z)
\end{pmatrix}\right\}\rho_z(D_z)\right\|_{B(Y)}
\\=& O(\eta_0^2+K^{-1})\,.
\end{split}  
\end{equation}
\par
Finally, we will prove that for a $\tau_0>0$,
\begin{equation}
  \label{eq:KPproj-cutoff}
\|\tilde{\rho_z}(\eps D_z)\mathcal{P}_{KP}(\eta_0)\|_{B(L^2_{\hat{\a}})}
+\|\mathcal{P}_{KP}(\eta_0)\tilde{\rho_z}(\eps D_z)\|_{B(L^2_{\hat{\a}})}
=O(e^{-\tau_0K})\,.
\end{equation}
Since $\tilde{g}_0(z,\eta):=e^{\hat{\a}z}g_0(z,\eta)$ and $\tilde{g}_0^*(z,\eta)=e^{-\hat{\a}z}g_0^*(z,\eta)$ are analytic on $\{z\in\C\mid |\Im z|<\hat{\a}_0\}$ and
$\sup_{\tau\in [-\tau_0,\tau_0]}(\|\tilde{g_0}(z+i\tau,\eta)\|_{L^1(\R_z)}+\|\tilde{g}_0^*(z+i\tau,\eta)\|_{L^1(\R_z)})<\infty$ for any $\tau_0\in[0,\hat{\a}_0)$
and $\eta\in[-\eta_0,\eta_0]$,
it follows from the Paley-Wiener theorem that there exists a $C_{\tau_0}$
for any $\tau_0\in[0,\hat{\a})$ such that
\begin{equation}
  \label{eq:Proj-approx3}
 \sup_{\eta\in[-\eta_0,\eta_0]}\left(
|\mF_z\tilde{g}_0(\xi,\eta)|+|\mF_z\tilde{g}_0^*(\xi,\eta)|\right)
\le C_{\tau_0}e^{-\tau_0 |\xi|}\,.
\end{equation}
By \eqref{eq:KPproj-cutoff} and the definition of $\mathcal{P}_{KP}(\eta_0)$,
we have \eqref{eq:KPproj-cutoff}.
Combining \eqref{eq:orth-lowfr}, \eqref{eq:Proj-KPapprox2} and
\eqref{eq:KPproj-cutoff}, we have \eqref{eq:Proj-KPapprox}.
Thus we complete the proof.
\end{proof}

Next, we will show that $(\lambda-\mL)^{-1}|_Z$ is uniformly bounded in $\lambda\in\Omega_\eps$.
\begin{lemma}
  \label{lem:low-frequencies}
Let $c$, $\a$ and $\eps_0$ be as in Lemma~\ref{lem:high frequencies}.
Then there exists a positive constant $C$ such that
\begin{equation*}
\sup_{\lambda\in \Omega_\eps}\|(\lambda-\mL)^{-1}f \|_X \le C\|f\|_X
\quad\text{for any $f\in Z$.}
\end{equation*}
\end{lemma}

Let $f\in Z$ and
\begin{equation}
  \label{eq:ftfutu}
\baru={}^t\!(\baru_1,\baru_2):=\Pi P(D)^{-1}u\,,\quad
\barf:={}^t\!(\barf_1,\barf_2)=\Pi P(D)^{-1}f\,.
\end{equation}
Then $\barf \in\overline{Z}$ and \eqref{eq:ev1} is translated into
\begin{equation}
  \label{eq:ev2}
\left\{
\begin{aligned}
& (\lambda-\lambda_+(D)-a_1-\barr_{11})\baru_1-(a_2+\barr_{12})\baru_2=\barf_1\,,\\
& (\lambda-\lambda_-(D)-a_2-\barr_{22})\baru_2-(a_1+\barr_{21})\baru_1= \barf_2\,,
\end{aligned}
\right.
\end{equation}
where 
\begin{gather*}
a_1= \frac{i}{2}B^{-1}S(D)^{-1}v_{1,c}\mu(D)^{-1}
-\frac{1}{2}B^{-1}S(D)^{-1}v_{2,c}S(D)\,,
\\
a_2=-\frac{i}{2}B^{-1}S(D)^{-1}v_{1,c}\mu(D)^{-1}
-\frac{1}{2}B^{-1}S(D)^{-1}v_{2,c}S(D)\,,
\\
\begin{pmatrix}
  \barr_{11} & \barr_{12} \\ \barr_{21} & \barr_{22}
\end{pmatrix}
= \left[\Pi,
  \begin{pmatrix}
    \lambda_+(D)+a_1 & a_2 \\ a_1 & \lambda_-(D)+a_2
  \end{pmatrix}
\right]\Pi^{-1}\,.
\end{gather*}

We decompose $\barf_2$ and $\baru_2$ into the high frequency part
and the low frequency part.
Let $\baru_{2,h}=\tilde{\rho}_z(D_z)\baru_2$, $\baru_{2,\ell}=\rho_z(D_z)\baru_2$, 
$\barf_{2,h}=\tilde{\rho}_z(D_z)\baru_2$ and $\barf_{2,\ell}=\rho_z(D_z)\barf_2$.
Then
\begin{equation*}
\left\{\lambda I-
  \begin{pmatrix}
\lambda_+(D) & 0 & 
\\ 0 & \lambda_-(D) & 0
\\ 0 & 0 & \lambda_-(D)
  \end{pmatrix}-A\right\}
  \begin{pmatrix}
    \baru_1 \\ \baru_{2,h} \\ \baru_{2,\ell}
  \end{pmatrix}
=
\begin{pmatrix}
  \barf_1 \\ \barf_{2,h} \\ \barf_{2,\ell}
\end{pmatrix}\,,
\end{equation*}
where $$A=
\begin{pmatrix}
a_1+\barr_{11} & a_2+\barr_{12} & a_2+\barr_{12}
\\ \tilde{\rho}_z(D_z)(a_1+\barr_{21})
& \tilde{\rho}_z(D_z)(a_2+\barr_{22})\tilde{\rho}_z(D_z)
& \tilde{\rho}_z(D_z)(a_2+\barr_{22})\rho_z(D_z)
\\  \rho_z(D_z)(a_1+\barr_{21}) & \rho_z(D_z)(a_2+\barr_{22})\tilde{\rho}_z(D_z) & \rho_z(D_z)(a_2+\barr_{22})\rho_z(D_z)
\end{pmatrix}\,.$$
To estimate $\baru_{2,h}$ and $\baru_{2,\ell}$, we need the following.
\begin{lemma}
  \label{lem:a2-res-bound}
Let $\hat{\a}\in(0,\hat{\a}_0/2)$,$\a=\hat{\a}\eps$  and $\Omega_\eps$ be
as in Lemma~\ref{cl:inv-lambda}.
There exists an $\eps_0>0$ such that 
\begin{gather}
  \label{eq:A2-b1}
\sup_{\lambda\in\Omega_\eps\;\eps\in(0,\eps_0)}
\left\|a_2(\lambda-\lambda_-(D))^{-1}\rho_z(D_z)\right\|_{B(Y)}<\infty\,,\\
  \label{eq:a2-res1}
\|a_2(\lambda-\lambda_-(D))^{-1}\tilde{\rho}_z(D_z)\|_{B(Y)}=O(K^{-1})\,.
\end{gather}
\end{lemma}

\begin{lemma}
  \label{lem:res-bound}
Let $\hat{\a}\in(0,\hat{\a}_0/2)$ and $\a=\hat{\a}\eps$.
Let $\hat{\beta}$ be a small positive number and
$\Omega_\eps$ be as in Lemma~\ref{cl:inv-lambda}.
There exist positive constants $\eps_0$ and $\eta_0$ such that if
$\eps\in(0,\eps_0)$,
\begin{equation*}
\sup_{\lambda\in\Omega_\eps\,,\;\eps\in(0,\eps_0)}
\left\|\rho_z(D_z)\{I-(a_2+\barr_{22})(\lambda-\lambda_-(D))^{-1}\}^{-1}
\rho_z(D_z) E_\eps^{-1}\mathcal{Q}_{KP}(\eta_0)E_\eps\rho_z(D_z)
\right\|_{B(Y)}<\infty\,.
\end{equation*}
\end{lemma}

\begin{proof}[Proof of Lemma~\ref{lem:a2-res-bound}]
By \eqref{eq:S} and the definition of $a_2$,
\begin{align*}
\|a_2(\lambda-\lambda_-(D))^{-1}\tilde{\rho}_z(D_z)\|_{B(Y)}
\lesssim & \|B^{-1}v_{1,c}\mu(D)^{-1}
(\lambda-\lambda_-(D))^{-1}\tilde{\rho}_z(D_z)\|_{B(Y)}
\\ & +\|B^{-1}v_{2,c}(\lambda-\lambda_-(D))^{-1}\tilde{\rho}_z(D_z)\|_{B(Y)}\,.  
\end{align*}
Since
\begin{equation}
  \label{eq:v1c-alt'}
B^{-1}v_{1,c}\mu(D)^{-1}=
c\left\{(q_c-B^{-1}[B,q_c])B^{-1}\mu(D)\right\}
-2cB^{-1}q_c'\pd_z\mu(D)^{-1}\,,  
\end{equation}
it follows from Claims~\ref{cl:qc-size}--\ref{cl:prop-B,S} and
Lemma~\ref{cl:inv-lambda} that
\begin{align*}
& \|B^{-1}v_{1,c}\mu(D)^{-1}(\lambda-\lambda_-(D))^{-1}\tilde{\rho}_z(D_z)\|_{B(Y)}
\\ \lesssim & \eps^2\|B^{-1}\mu(D)(\lambda-\lambda_-(D))^{-1})\|_{B(Y_{high})}
+\eps^3\|(\lambda-\lambda_-(D))^{-1}\|_{B(Y_{high})}
\\ =O(K^{-1})\,.
\end{align*}
We can prove
\begin{align*}
& \|B^{-1}v_{2,c}(\lambda-\lambda_-(D))^{-1}\tilde{\rho}_z(D_z)\|_{B(Y)}
\\ \lesssim & 
\eps^3\|(\lambda-\lambda_-(D))^{-1})\|_{B(Y_{high})}
+\eps^2\|B^{-1}\pd_z(\lambda-\lambda_-(D))^{-1}\|_{B(Y_{high})}
\\=& O(K^{-1})  
\end{align*}
in the same way. Thus we prove \eqref{eq:a2-res1}.
\par
Next, we will show \eqref{eq:A2-b1}.
As in the proof of \eqref{eq:a2-res1}, we have
\begin{align*}
 \|a_2(\lambda-\lambda_-(D))^{-1}\rho_z(D_z)\|_{B(Y)}
 \lesssim & \eps^2\|\mu(D)(\lambda-\lambda_-(D))^{-1}\|_{B(Y_{low})}
\\ & +\eps^2\|\pd_z(\lambda-\lambda_-(D))^{-1}\|_{B(Y_{low})}
 +\eps^3\|(\lambda-\lambda_-(D))^{-1}\|_{B(Y)}\,.
\end{align*}
Combining the above with Lemma~\ref{cl:inv-lambda}, we have \eqref{eq:A2-b1}.
Thus we complete the proof.
\end{proof}
\begin{proof}[Proof of Lemma~\ref{lem:res-bound}]
To prove  Lemma~\ref{lem:res-bound}, we approximate
$\lambda_-(D)+a_2$ by $\mL_{KP}$ and apply Proposition~\ref{prop:KPII}.
Let $E_\eps: L^2_\a(\R^2)\to L^2_{\hat{\a}}(\R^2)$ be an isomorphism defined by 
$(E_\eps f)(x,y):=\eps^{-3/2}f(x/\eps,y/\eps^2)$,
$a_{2,\eps}=\eps^{-3}E_\eps a_2 E_{\eps}^{-1}$ and 
$\lambda_{-,\eps}(\xi,\eta)=\eps^{-3}\lambda_-(\eps\xi,\eps^2\eta)$.
Then
\begin{equation*}
\left\|\rho_z(D_z)\left\{a_2(\lambda-\lambda_-(D))^{-1}
+\frac{3}{2}E_\eps^{-1}\pd_z (\theta_0\cdot)(\Lambda -\mL_{KP,0})^{-1}E_\eps
\right\}\rho_z(D_z)\right\|_{B(Y)}
\le  III_1+III_2\,,
\end{equation*}
where $\rho_{KP}(\xi,\eta)=\rho_z(\eps\xi)\rho_y(\eps^2\eta)$ and
\begin{align*}
III_1=& \left\|\rho_{KP}(D)\left\{
a_{2,\eps}+\frac{3}{2}\pd_z(\theta_0\cdot)\right\}\rho_{KP}(D)\right
\|_{B(L^2_{\hat{\a}})} \|(\Lambda-\lambda_{-,\eps}(D))^{-1}\|_{B(L^2_{\hat{\a}})}
\\ III_2=&  \frac{3}{2}
\left\|\rho_{KP}(D)\pd_z(\theta_0\cdot)
\{(\Lambda-\mL_{KP,0})^{-1}-(\Lambda-\lambda_{-,\eps}(D))^{-1}\}
\rho_{KP}(D)\right\|_{B(L^2_{\hat{\a}})}\,.
\end{align*}
By \eqref{eq:l-inv} and \eqref{eq:a,l-l}, we have $III_1=O(K^5\eps^2)$.
By \eqref{eq:fm0},
\begin{equation*}
III_2 \lesssim  \sup_{(\xi,\eta)\in \widetilde{A}_{low}}
\frac{(1+|\xi+i\hat{\a}|)
|\lambda_{-,\eps}(\xi+i\hat{\a},\eta)-\mL_{KP,0}(\xi+i\hat{\a},\eta)|}
{|\Lambda-\mL_{KP,0}(\xi+i\hat{\a},\eta)|
|\Lambda-\lambda_{-,\eps}(\xi+i\hat{\a},\eta)|}\,.  
\end{equation*}
Since 
$$|\lambda_{-,\eps}(\xi+i\hat{\a},\eta)-\mL_{KP,0}(\xi+i\hat{\a},\eta)|
=O(K^8\eps^2)$$
by \eqref{eq:lmm-low2} and 
$\sup_{\Re\Lambda \ge -\hat{\beta}/2\,,\;(\xi,\eta)\in\R^2}
(1+|\xi|) |\Lambda-\mL_{KP,0}(\xi+i\hat{\a},\eta)|^{-1}<\infty$
thanks to Lemma~\ref{lem:free-KP},  we have 
\begin{align*}
III_2  \lesssim  K^8\eps^2\,.
\end{align*}
Thus we have
\begin{equation}
  \label{eq:KP-approx}
\left\|
\rho_z(D_z)a_2(\lambda-\lambda_-(D))^{-1}\rho_z(D_z)
+ \frac{3}{2} E_\eps^{-1}
\pd_{\hat{z}}\left\{\theta_0(\Lambda-\mL_{KP,0})^{-1}\right\}E_\eps
\right\|_{B(Y)}=O(K^8\eps^2)\,.
\end{equation}
By Lemma~\ref{cl:inv-lambda} and Claim~\ref{cl:tr-bound}, we have
$$\|\barr_{22}(\lambda-\lambda_-(D))^{-1}\|_Y=O(K^5\eps^2)\,.$$
Combining the above with Proposition~\ref{prop:KPII} and \eqref{eq:KP-approx},
we obtain Lemma~\ref{lem:res-bound}. Thus we complete the proof.
\end{proof}

Now we are in position to prove Lemma~\ref{lem:low-frequencies}.
\begin{proof}[Proof of Lemma~\ref{lem:low-frequencies}]
By Lemma~\ref{cl:inv-lambda}, Claims~\ref{cl:V-bound12} and \ref{cl:tr-bound},
$$\|(\lambda-\lambda_+(D)-a_1-\barr_{11})^{-1}\|_{B(Y)}=O(\eps^{-1})\,,$$
\begin{equation}
  \label{eq:tu1-bound}
\|\baru_1\|_Y\lesssim \eps^{-1}\|\barf_1\|_Y+\eps(\|\baru_{2,h}\|_Y
+\|\baru_{2,\ell}\|_Y)\,.
\end{equation}
\par
Since
\begin{equation*}
\|\rho_z(D_z)(\lambda-\lambda_-(D)-a_2-\barr_{22})^{-1}
\rho_z(D_z)E_\eps^{-1}\mathcal{Q}_{KP}(\eta_0)E_\eps\rho_z(D_z)\|_{B(Y)}
=O(\eps^{-3}) 
\end{equation*} 
by Lemmas~\ref{cl:inv-lambda} and \ref{lem:res-bound},
$$\|\baru_{2,\ell}\|_Y\lesssim 
\eps^{-3}\|\barf_{2,\ell}\|_Y+K(\|\baru_{2,h}\|_Y+\|\baru_1\|_Y)$$
follows from Claims~\ref{cl:V-bound12} and \ref{cl:tr-bound}.
Furthermore, Lemmas~\ref{lem:a2-res-bound}, \ref{lem:res-bound} 
and Claim~\ref{cl:tr-bound} imply 
\begin{gather*}
\|(a_2+\barr_{22})
\rho_z(D_z)(\lambda-\lambda_-(D)-a_2-\barr_{22})^{-1}
\rho_z(D_z)E_\eps^{-1}\mathcal{Q}_{KP}(\eta_0)E_\eps\rho_z(D_z)\|_{B(Y)}=O(1)\,,
\\
\|(a_2+\barr_{22})\baru_{2,\ell}\|_Y\lesssim
\|\barf_{2,\ell}\|_Y+K\eps^3(\|\baru_1\|_Y+\|\baru_{2,h}\|_Y)\,.  
\end{gather*}
By Lemma~\ref{cl:inv-lambda}, Claims~\ref{cl:V-bound12} and \ref{cl:tr-bound},
$$\|\tilde{\rho}_z(D_z)(\lambda-\lambda_-(D_z)-a_2-\barr_{22})^{-1}
\tilde{\rho_z}(D_z)\|_{B(Y)}=O(K^{-2}\eps^{-3})\,,$$
and 
\begin{align*}
 \|\baru_{2,h}\|_Y\lesssim & 
K^{-2}\eps^{-3}(\|\barf_{2,h}\|_Y+\|(a_1+\barr_{21})\baru_1\|_Y
+\|(a_{22}+\barr_{22})\baru_{2,\ell}\|_Y)
\\ \lesssim & 
K^{-2}\eps^{-3}(\|\barf_{2,h}\|_Y+\|\barf_{2,\ell}\|_Y)
+K^{-2}\eps^{-1}\|\baru_1\|_Y+K^{-1}\|\baru_{2,h}\|_Y\,.
\end{align*}
Combining the above, we have
\begin{gather*}
  \|\baru_1\|_Y \lesssim \eps^{-1}\|\barf_1\|_Y
+\eps^{-2}(K^{-1}\|\barf_{2,h}\|_Y+\|\barf_{2,\ell}\|_Y)\,,
\\
\|\baru_{2,h}\|_Y \lesssim K^{-2}\eps^{-2}\|\barf_1\|_Y
+K^{-2}\eps^{-3}(\|\barf_{2,h}\|_Y+\|\barf_{2,\ell}\|_Y)\,,\\
\|\baru_{2,\ell}\|_Y \lesssim K^{-1}\eps^{-2}\|\barf_1\|_Y+\eps^{-3}
(K^{-1}\|\barf_{2,h}+\|\barf_{2,\ell}\|_Y)\,,
\end{gather*}
and
$\sup_{\lambda\in\Omega_\eps}\|\Pi P(D)^{-1}(\lambda-\mL)^{-1}P(D)\Pi^{-1}
\|_{B(\overline{Z})}<\infty$.
Since $\Pi P(D)^{-1}:Z\to\overline{Z}$ is isomorphic, we have
Lemma~\ref{lem:low-frequencies}.
Thus we complete the proof.
\end{proof}
\subsection{Proof of Theorem~\ref{thm:spectrum}}
Now we are in position to prove Theorem~\ref{thm:spectrum}.
Lemmas~\ref{lem:construction}, \ref{lem:resonance1},
\ref{lem:high frequencies} and \ref{lem:low-frequencies}
imply \eqref{eq:spec-L} and \eqref{eq:asymp-ev}.
Taking $\hat{\beta}>0$ smaller if necessary, we see from
Gearhart-Pr\"{u}ss theorem that for small $\eps>0$,
there exists a $K=K(\eps)$ satisfying \eqref{eq:decay-estimate}.
This completes the proof of Theorem~\ref{thm:spectrum}.
\bigskip

\section{Proof of Theorem~\ref{thm:1}}
\label{sec:bifurcation}
In this section, we will show that  the eigenvalue $\lambda=0$
of $\mL(0)$ splits into two stable eigenvalues of $\mL(\eta)$ for small
$\eta\ne0$ without assuming smallness of line solitary waves.
As in Subsection~\ref{subsec:resonant}, we will use Lyapunov Schmidt method.
\par
To begin with, we expand $\mL(\eta)$ as
$\mL(\eta)=\mL(0)+\eta^2\mL_1(\eta)$ with
$$\mL_1(0)=B_0^{-1}
\begin{pmatrix} 0 & 0 \\ I-A_0-B_0^{-1}A_0+r_c  & 0 \end{pmatrix}
+bB_0^{-2} \begin{pmatrix}0 & 0 \\ v_{1,c}(0) & v_{2,c}(0)
\end{pmatrix}\,.$$
We easily see that $\|\mL_1(\eta)\|_{B(H^1_\a(\R)\times L^2_\a(\R))}=O(1)$ as $\eta\to0$.
\par
Using the ansatz
\begin{gather*}
  \lambda(\eta)=i\eta\lambda_1(\eta)\,,\quad
\zeta(\eta)=\zeta_1+\{\lambda(\eta)+\eta^2\gamma(\eta)\}\zeta_2
+\eta^2 z(\eta)\,,
\end{gather*}
we will solve the eigenvalue problem \eqref{eq:evBL-eta}.
Suppose $\mL(\eta)\zeta(\eta)=\lambda\zeta(\eta)$ and
$z(\eta)\perp \zeta_1^*$, $\zeta_2^*$. Then
\begin{gather}
\label{eq:z}
\mathcal{Q}_0(\mathcal{L}(\eta)-i\eta\lambda_1)z(\lambda_1,\gamma,\eta)
+\mathcal{Q}_0G(\lambda_1,\gamma,\eta)=0\,,
\\ \label{eq:F-l-g}
F_k(\lambda_1,\gamma,\eta):=\la G(\lambda_1,\gamma,\eta)
+\eta^2\mL_1(\eta)z(\lambda_1,\gamma,\eta),\zeta_k^*\ra=0
\quad\text{for $k=1$, $2$,}
\\ G(\lambda_1,\gamma,\eta)=
\gamma\left(\zeta_1-i\lambda_1\eta\zeta_2+\eta^2\mL_1(\eta)\zeta_2\right)
+\lambda_1^2\zeta_2+\mL_1(\eta)(\zeta_1+i\lambda_1\eta\zeta_2)
\end{gather}
where $\mathcal{Q}_0:H^1_\a(\R)\times L^2_\a(\R)\to \perp^{}\ker_g(\mL(0))^*)$
is a spectral projection associated with $\mL(0)$.
\par
The operator
$\mL(0):H^2_\a(\R)\times H^1_\a(\R)\to H^1_\a(\R)\times L^2_{\alpha}(\R)$
is a Fredholm operator of index zero.
In fact, we see from Claim~\ref{cl:spectrum-free2}, \eqref{eq:resl0} and \eqref{eq:cxi-muS} with $\lambda=0$ that
$\mL_0(0):H^2_\a(\R)\times H^1_\a(\R)\to H^1_\a(\R)\times L^2_\a(\R))$
has a bounded inverse and $V(0)$ is a compact operator on $H^1_\a(\R)\times L^2_\a(\R)$.
Note that $\lambda_+(D_z,0)^{-1}\in B(L^2_\a(\R),H^1_\a(\R))$  by \eqref{eq:cl-sf2} and the fact that $\pd_z^{-1}\in B(L^2_\a(\R),H^1_\a(\R))$.

Thus there exist positive constants $C$ and $k$ such that if
$|\eta|(|\lambda_1|+\|\mL_1(\eta)\|_{B(H^1_\a(\R)\times L^2_\a(\R))})<k$,
then a solution $z=z(\lambda_1,\gamma,\eta)$ of \eqref{eq:z} satisfies 
$$\|z(\lambda_1,\gamma,\eta)\|_{H^2_\a(\R)\times H^1_\a(\R)}
\le C\|G(\lambda_1,\gamma,\eta)\|_{H^1_\a(\R)\times L^2_\a(\R)}\,.$$
\par
Now we choose constants $\lambda_{1,0}$ and $\gamma_0$ so that
\begin{equation}
  \label{eq:pft21-1}
\begin{split}
& 
F_1(\lambda_{1,0},\gamma_0,0)=\gamma_0\la\zeta_1,\zeta_1^*\ra
+\la \mL_1(0)\zeta_1,\zeta_1^*\ra+\lambda_{1,0}^2\la \zeta_2,\zeta_1^*\ra=0\,,
\\ &  \notag
F_2(\lambda_{1,0},\gamma_0,0)=\la \mL_1(0)\zeta_1,\zeta_2^*\ra
+\lambda_{1,0}^2\la \zeta_2,\zeta_2^*\ra=0\,.
\end{split} 
\end{equation}
By straightforward computations, we have
\begin{equation}
  \label{eq:zeta1-zeta1*}
\la \zeta_1,\zeta_2^*\ra =0\,,\quad
\la \zeta_1,\zeta_1^*\ra = \la \zeta_2,\zeta_2^*\ra
=\frac{1}{2}\frac{d}{dc}E(q_c,r_c) >0\,,  
\end{equation}
\begin{equation}
  \label{eq:zeta2-zeta1*}
\begin{split}
\la \zeta_2,\zeta_1^*\ra
&=-\left(\frac{c}{2}\frac{d}{dc} \int_{\R}q_c^2\,dz
+c\frac{d}{dc} \int_{\R} r_c\,dz\right) \left(\frac{d}{dc}\int_{\R}q_c\,dz\right)\\
&=\frac{16}{3c^4}\frac{bc^4-a}{c^2-1}\frac{a(c^2-1)+(bc^2-a)+2c^4(2bc^2-b-a)}{bc^2-a} >0\,,
\end{split}
\end{equation}
\begin{equation}
\label{eq:mL1-1-1}
\begin{split}
\la \mL_1(0)\zeta_1,\zeta_2^*\ra 
=& \la -A_0q_c-B_0^{-1}A_0q_c+q_c-bB_0^{-1}\pd_z^2(c\frac{3}{2}q_c^2)-cq_c^2,
cq_c\ra
\\=& \left\la\left(-\frac43A_0+\frac{c^2}{3}B_0+1-c^2\right)q_c,cq_c\right\ra
\\= &-\frac{8}{15} \frac{c^2-1}{c}\alpha_c\{ 2c^2(b-a)+3(bc^4-a)\}<0
\end{split}  
\end{equation}
because
\begin{equation}
  \label{eq:qc-alt}
(A_0-c^2B_0)q_c+\frac{3}{2}cq_c^2=0  
\end{equation}
by \eqref{eq:qc} and 
$$\int_\R q_c(x)^2\,dx=\frac{8(c^2-1)^2}{3\a_cc^2}\,,
\quad \int_\R q_c'(x)^2\,dx=\frac{8\a_c(c^2-1)^2}{15c^2}\,.$$
In view of \eqref{eq:zeta1-zeta1*} and \eqref{eq:mL1-1-1}, we have
$\lambda_{1,0}:=\sqrt{\frac{\la \mL_1(0)\zeta_1,\zeta_2^*\ra}{-\la\zeta_2,\zeta_2^*\ra}}>0$.
\par
Since
\begin{align*}
& \pd_{\lambda_1}F_1(\lambda_{1,0},\gamma_0,0)
=2\lambda_{1,0}\la \zeta_2,\zeta_1^*\ra\,,
\quad
\pd_{\gamma}F_1(\lambda_{1,0},\gamma_0,0)=\la \zeta_1,\zeta_1^*\ra\ne0\,,
\\ &
\pd_{\lambda_1}F_2(\lambda_{1,0},\gamma_0,0)
=2\lambda_{1,0}\la \zeta_2,\zeta_2^*\ra\ne 0\,,
\quad \pd_{\gamma}F_2(\lambda_{1,0},\gamma_0,0)=\la \zeta_1,\zeta_2^*\ra=0\,,
\end{align*}
it follows from the implicit function theorem that
there exists an $\eta_0>0$,
$\lambda_1(\eta)$, $\gamma(\eta)\in C^1([-\eta_0,\eta_0])$ such that
$\lambda_1(0)=\lambda_{1,0}$, $\gamma(0)=\gamma_0$ and
$F_k(\lambda_1(\eta),\gamma(\eta),\eta)=0$ for $\eta\in[-\eta_0,\eta_0]$
and $k=1$, $2$.
Moreover, we have
$$\lambda_1'(0)=
-\frac{\pd_\eta F_2(\lambda_{1,0},\gamma_0,0)}
{\pd_{\lambda_1} F_2(\lambda_{1,0},\gamma_0,0)}
=\frac{i}{2}\left(\gamma_0-
\frac{\la\mL_1(0)\zeta_2,\zeta_2^*\ra}{\la \zeta_2,\zeta_2^*\ra}\right)
=:i\lambda_{2,0}\,,$$
and $\lambda(\eta)=i\lambda_{1,0}\eta-\lambda_{2,0}\eta^2+O(\eta^3)$.
Thus we prove \eqref{eq:lambda-asymp} and \eqref{eq:zeta-asymp}.
\par
To obtain the asymptotic expansion of $\zeta^*(\eta)$, let
$\tv_2(z,\eta)=\zeta(-z,-\eta)\cdot {}^t(1,0)$,
where $\cdot$ denotes the inner product in $\C^2$ and
$$
\zeta^*(\eta)=c
\begin{pmatrix}
(\lambda(-\eta)+c\pd_z)B(\eta)\tv_2(z,\eta)+v_{2,c}(\eta)^*\tv_2(z,\eta)
\\ B(\eta)\tv_2(z,\eta)
\end{pmatrix}\,.$$
As in the proof of Lemma~\ref{lem:resonance1},
we have $\mL(\eta)\zeta^*(\eta)=\lambda(-\eta)\zeta^*(\eta)$.
Since
\begin{gather*}
  \lambda(-\eta)=-i\lambda_1\eta-\lambda_2\eta^2+O(\eta^3)\,,\\
\tv_2(\cdot,\eta)=q_c-i\lambda_1\eta\int_{-\infty}^z\pd_cq_c(z_1)\,dz_1
+O(\eta^2)\quad\text{in $H^k_{-\a}(\R)$ for any $k\ge0$,}
\end{gather*}
we have \eqref{eq:zeta*-asymp}. We can show \eqref{eq:zeta-parity} in the same way
as the proof of Lemmas~\ref{lem:construction} and \ref{lem:resonance1}.
\par
Finally, we will prove $\lambda_{2,0}>0$.
By \eqref{eq:zeta1-zeta1*} and the definition of $\lambda_{2,0}$,
\begin{equation}
  \label{lambda2}
\frac{d}{dc}E(q_c,r_c)\lambda_{2,0}=
-\la \mL_1(0)\zeta_1,\zeta_1^*\ra-\la \mL_1(0)\zeta_2,\zeta_2^*\ra
-\lambda_{1,0}^2\la \zeta_2,\zeta_1^*\ra\,.
\end{equation}
We have
\begin{align*}
\mL_1(0)\zeta_2=&  B_0^{-1}\begin{pmatrix} 0 \\
A_0\pd_z^{-1}\pd_cq_c+(A_0-B_0)B_0^{-1}\pd_z^{-1}\pd_cq_c
\end{pmatrix}
\\ & + bB_0^{-2}
\begin{pmatrix} 0 \\
 3c \pd_z(q_c\pd_cq_c)+\frac{3}{2}\pd_z(q_c^2)
\end{pmatrix} 
 +B_0^{-1}\begin{pmatrix} 0 \\ c q_c\pd_z^{-1}\pd_cq_c\end{pmatrix}\,.
\end{align*}
Using the fact that $q_c$ and $\pd_cq_c$ are even,
$q_c'$ is odd and $B^{-1}_0f$ retains the parity of  $f$, we have
\begin{equation} \label{lam-eq2}
  \begin{split}   
\la \mL_1(0)\zeta_2,\zeta_2^*\ra =&
 c\la \pd_z^{-1}\pd_cq_c,q_c\ra
+c^2\la \pd_z^{-1}\pd_cq_c,q_c^2\ra
= \frac{c}{3}(2c^2+1)\la \pd_z^{-1}\pd_cq_c,q_c\ra\,.
  \end{split}
\end{equation}
In the last line, we use $(A_0-c^2B_0)q_c+\frac{3c}{2}q_c^2=0$.
Analogously, we have
\begin{equation}
\label{lam-eq3}
\la \mL_1(0)\zeta_1,\zeta_1^*\ra=
c\la q_c,(\pd_z^{-1})^*\pd_cqc\ra+c^2\la q_c^2,(\pd_z^{-1})^*\pd_cqc\ra
=\frac{c}{3}(2c^2+1)\la \pd_z^{-1}\pd_cq_c,q_c\ra\,,
\end{equation}
where $(\pd_z^{-1})^*f=-\int^z_{-\infty}f(z_1)\,dz_1$.
By integration by parts,
\begin{gather}
\label{Int1}
\la \pd_z^{-1}\pd_cq_c,q_c\ra
=-\frac{1}{4}\frac{d}{dc}\left(\int_\R q_c\,dz\right)^2
=-4\frac{d}{dc}\frac{(c^2-1)(bc^2-a)}{c^2}\,.
\end{gather}
Combining  \eqref{lambda2}--\eqref{Int1} with
\eqref{eq:pft21-1}--\eqref{eq:mL1-1-1},  we have
\begin{equation}
\label{lambda2-3}
\begin{split} 
\lambda_{2,0}\frac{d}{dc}E(q_c,r_c) =&
-\lambda_{1,0}^2\la \zeta_2,\zeta_1^*\ra
-\frac{16}{3} \frac{(1 + 2 c^2)}{c^2} (b c^4-a) \\
&=32\frac{(b\rho^2-a)}{3d(c)}n(c)\,,
\end{split}
\end{equation}
where 
\begin{align}
\label{d-eq} d(c) &=6a^2+(3a^2-9ab)c^2+(6a^2+2b^2-2ab)c^4+(b^2-19ab)c^6+12b^2c^8,\\
\label{n-eq} n(c) &=7 a^2-ba+(4a^2-10ba)c^2+(3b^2+4a^2-7ab)c^4+6b(b-2a)c^6+6b^2c^8.
\end{align}
To show that $n(c)>0$ for all $c>1$ and $b>a>0$, we set $\rho=c^2$ and
differentiate $n(\rho)$ twice to get
$$n'(\rho)=4a^2-10ab+2(3b^2+4a^2-7ab)\rho+18(b^2-2ab)\rho^2+24b^2\rho^3,$$
and
$$n''(\rho)=2(3b^2+4a^2-7ab)+36b^2\rho+72b(b\rho^2-a\rho) > 0,\quad \forall \rho>1.$$
Since $n'(1)=12(b-a)^2+36(b-a)b>0$, $n'(\rho)>0$ for all $\rho>1$. Since $n(1)=15(b-a)^2>0$, thus, $n(\rho)>0$ for all $\rho>1$.
In the same way, to show that $d(\rho)>0$ for all $\rho>1$ and $b>a>0$, we set $\rho=c^2$ and differentiate $d(\rho)$ to obtain
$$d'(\rho)=(3a^2-9ab)+2(6a^2+2b^2-2ab)\rho+3(b^2-19ab)\rho^2+48b^2\rho^3,$$
  and
 $$d''(\rho)=12a^2+4b(b-a)+6b^2\rho+b\rho(-114a+144b\rho)>0, \quad\forall \rho>1.$$ 
 Since $d'(1)=15(b-a)^2+40(b-a)b>0$, $d'(\rho)>0$ for all $\rho>1$. Since $d(1)=15(b-a)^2>0$, thus, $d(\rho)>0$ for all $\rho>1$.\\
Since $\frac{d}{dc}E(q_c,r_c)>0$, $d(c)>0$ and $n(c)>0$ for $c>1$,
we conclude from \eqref{lambda2-3} that $\lambda_{2,0}>0$.
This completes the proof of Theorem~\ref{thm:1}.
\bigskip

\section{Proof of Corollary~\ref{thm:linear-stability}}
\label{sec:linear-stability}
The Gearhart-Pr\"{u}ss theorem \cite{Ge78,Pr85} tells us
the semigroup estimate \eqref{eq:decay} follows from 
uniform boundedness of $(\lambda-\mL)^{-1}\mathcal{Q}(\eta_0)$ in a stable half plane.
Let $\Omega=\{\lambda\mid \Re \lambda\ge -\beta'\}$.
Applying \cite[Corollary~4]{Pr85} to a Hilbert space $\mathcal{Q}(\eta_0)X$,
we have \eqref{eq:decay} provided
$(\lambda-\mL)^{-1}\mathcal{Q}(\eta_0)$ is uniformly bounded in $\Omega$.
Thus to prove Theorem~\ref{thm:linear-stability}, it suffices to show the following.
\begin{lemma}
  \label{lem:G-P-cond}
Let $c>1$ and $\a\in(0,\a_c)$.  Assume \eqref{ass:H} for $\beta\in(0,\a(c-1)/2)$ and 
an $\eta_0>0$. Then for any $\beta'<\beta$,
\begin{equation}
  \label{eq:G-P-cond}
\sup_{\lambda\in\Omega}\|(\lambda-\mL)^{-1}\mathcal{Q}(\eta_0)\|_{B(X)}<\infty\,.
\end{equation}
\end{lemma}
\begin{proof}
By \eqref{ass:H}, the restricted resolvent $(\lambda-\mL)^{-1}
\mathcal{Q}(\eta_0)$ is
uniformly bounded on any compact subset of $\Omega$.
Thus by Lemma~\ref{lem:spectrum-free} and \eqref{eq:res-formula}, we have
\eqref{eq:G-P-cond}  provided 
\begin{equation}
  \label{eq:G-P2}
  \sup_{\lambda\in\Omega,\,|\lambda|\ge K_1}\|(\lambda-\mL_0)^{-1}V\|_{B(X)}\le \frac12
\end{equation}
for sufficiently large $K_1$.  To prove \eqref{eq:G-P2}, we apply the
argument for the $1$-dimensional Benney-Luke equation \cite{MPQ13} for
low frequencies in $y$ and use the argument in \S\ref{subsec:high} for
high frequencies in $y$.
\par
Let $K_2>0$, $\chi$ be the characteristic function of $[-K_2,K_2]$
and $\tilde{\chi}(\eta)=1-\chi(\eta)$ for $\eta\in\R$.
First, we will show that
\begin{equation}
  \label{eq:uniform-conv}
(\lambda-\mL_0)^{-1}V\chi(D_y)=
\begin{pmatrix}
r_{11}(\lambda) & r_{12}(\lambda)\\ r_{21}(\lambda) & r_{22}(\lambda)
\end{pmatrix}\chi(D_y)\to 0
\quad \text{uniformly as $\lambda\to\infty$ with $\lambda\in\Omega$.}
\end{equation}
By \eqref{eq:hf-pf2} and \eqref{eq:v1c-alt}, 
\begin{align*}
r_{11}(\lambda)\chi(D_y)=& 
\frac{ic}{2}S(D)^{-1}
\{(\lambda-\lambda_+(D))^{-1}-(\lambda-\lambda_-(D))^{-1}\}\mu(D)B^{-1}q_c\chi(D_y)
\\ & -c(\lambda-\lambda_+(D))^{-1}(\lambda-\lambda_-(D))^{-1}B^{-1}q_c''\chi(D_y)\,.
\end{align*}
By the Plancherel theorem, 
\begin{equation*}
  \|(\lambda-\lambda_\pm(D))^{-1}\chi(D_y)f\|_{L^2_\a(\R^2)}
=\left\| \frac{\hat{f}(\xi+i\a,\eta)}{\lambda-\lambda_\pm(\xi+i\a,\eta)}
\right\|_{L^2(\R_\xi\times [-K_2,K_2]}\,.
\end{equation*}
In view of \eqref{eq:cl-sf4} and \eqref{eq:cl-sf5},
we have $\lim_{\lambda\in\Omega,\,\lambda\to\infty}
\|(\lambda-\lambda_\pm(D))^{-1}f\|_{L^2_\a(\R^2)}=0$
for any $f\in L^2_\a$ thanks to the dominated convergence theorem.
Thus we prove $(\lambda-\lambda_\pm(D))^{-1}\to 0$ strongly as $\lambda\to\infty$
with $\lambda\in\Omega$.
Since $\mu(D)B^{-1}q_c\chi(D_y)$, $B^{-1}q_c''\chi(D_y):H^1_\a\to H^1_\a$ are compact,
we see that $\lim_{\lambda\in\Omega,\,\lambda\to\infty}\|r_{11}(\lambda)\|_{B(H^1_\a)}=0$ 
as in \cite[p.265]{MPQ13}.
We can prove
$$\lim_{\lambda\in\Omega,\,\lambda\to\infty}
\left(\|r_{12}(\lambda)\chi(D_y)\|_{B(L^2_\a,H^1_\a)}
+\|r_{21}(\lambda)\chi(D_y)\|_{B(H^1_\a,L^2_\a)}
+\|r_{22}(\lambda)\chi(D_y)\|_{B(L^2_\a)}
\right)=0$$
in exactly the same way.
\par
By Lemma~\ref{cl:inv-lambda} and the definition of $r_{ij}(\lambda)$,
\begin{equation}
  \label{eq:r-high}
\begin{split}
& \|r_{ij}(\lambda)\tilde{\chi}(D_y)\|_{B(H^{2-j}_\a,H^{2-i}_\a)}
\\ \lesssim &
\|\tilde{\chi}(D_y)\mu(D)B^{-1}\|_{B(L^2_\a)}
+\|\tilde{\chi}(D_y)B^{-1}\|_{B(L^2_\a)}
 \to 0 \quad\text{as $K_2\to\infty$.}
\end{split}  
\end{equation}
Combining \eqref{eq:uniform-conv} and \eqref{eq:r-high},
we have \eqref{eq:G-P2}.
Thus we complete the proof.
\end{proof}
\bigskip

\section{Proof of Theorem~\ref{thm:linear-dynamics}}
\label{sec:dynamics}
Let 
$$
g(z,\eta)=\left(1+i\frac{\Re \la \zeta(\cdot,\eta),\zeta^*(\cdot,\eta)\ra}{\Im \la \zeta(\cdot,\eta),\zeta^*(\cdot,\eta)\ra}\right)
\zeta(x,\eta)\,,\quad g^*(x,\eta)=\zeta^*(x,\eta)\,,$$
and define $g_k(x,\eta)$ and $g_k^*(x,\eta)$ ($k=1$, $2$)
by \eqref{eq:def-gk} and \eqref{eq:def-gk*}
as in Section~\ref{sec:resonant}.
By \eqref{eq:zeta-parity}, we have for $\eta\in[-\eta_0,\eta_0]$, $z\in\R$ and $k=1$, $2$,
\begin{align*}
& \overline{g(z,\eta)}=g(z,-\eta)\,,\quad \overline{g^*(x,\eta)}=g^*(x,-\eta)\,, 
\\ &
\kappa(\eta):=\frac12\Im\la g(\cdot,\eta), g^*(\cdot,\eta)\ra\quad\text{is odd,}
\end{align*}
and $g_k(z,\eta)$ and $g_k^*(z,\eta)$ are real valued and even in $\eta$. Moreover,
$$\la g_j(\cdot,\eta),g_k^*(\cdot,\eta)=\delta_{jk}\quad\text{for $j$, $k=1$, $2$.}$$
By Theorem~\ref{thm:1} and \eqref{eq:zeta1-zeta1*},
\begin{align*}
  \la \zeta(\cdot,\eta),\zeta^*(\cdot,\eta)\ra =& 
\la \zeta_1,\zeta_2^*\ra+i\lambda_1\eta\{\la \zeta_2,\zeta_2^*\ra+ \la \zeta_1,\zeta_1^*\ra\}+O(\eta^2)
\\=& 2i\kappa_1\eta+O(\eta^2)\,,
\end{align*}
and
\begin{equation}
  \label{eq:kappa-exp}
\begin{split}
\kappa(\eta)=& \frac12\Im\la \zeta(\cdot,\eta),\zeta^*(\cdot,\eta)\ra
\left\{1+\left(\frac{\Re\la \zeta(\cdot,\eta),\zeta^*(\cdot,\eta)}{\Im\la \zeta(\cdot,\eta),\zeta^*(\cdot,\eta)}\right)^2\right\}
\\= & \kappa_1\eta+O(\eta^3)\,.
\end{split}  
\end{equation}
\par

Let $\vec{\Phi}(t)=(\Phi(t),\Psi(t))$ be a solution of
\eqref{eq:linear} with $\vec{\Phi}(0)=(\Phi_0,\Psi_0)$ and
\begin{equation*}
c_k(t,\eta)=\left \la \mF_y\vec{\Phi}(t,\cdot,\eta),g_k^*(\cdot,\eta)\right\ra
\quad\text{for $\eta\in[-\eta_0,\eta_0]$ and $k=1$, $2$.}  
\end{equation*}
Then $$\vec{\Phi}(t)=\frac{1}{\sqrt{2\pi}}\sum_{k=1,2}\int_{-\eta_0}^{\eta_0}
c_k(t,\eta)g_k(z,y)e^{iy\eta}\,d\eta\,.$$
By Remark~\ref{rem:L},
\begin{align*}
\pd_t  \begin{pmatrix}c_1(t,\eta) \\ c_2(t,\eta) \end{pmatrix}
=& \begin{pmatrix}
\la \mL(\eta)\mF_y\vec{\Phi}(t,\cdot,\eta), g_1^*(\cdot,\eta)\ra \\
\la \mL(\eta)\mF_y\vec{\Phi}(t,\cdot,\eta), g_2^*(\cdot,\eta)\ra       
   \end{pmatrix}
= \mathcal{A}(\eta)
 \begin{pmatrix} c_1(t,\eta) \\ c_2(t,\eta) \end{pmatrix}\,,
\end{align*}
where 
$$\mathcal{A}(\eta)=
\begin{pmatrix}
\Re\lambda(\eta) & \frac{\Im\lambda(\eta)}{\kappa(\eta)}
\\ -\kappa(\eta)\Im\lambda(\eta) & \Re\lambda(\eta)  
\end{pmatrix}\,.$$
Let $e(t,\eta)=|\kappa(\eta)c_1(t,\eta)|^2+|c_2(t,\eta)|^2$. Then
$e(t,\eta)=e^{2t\Re\lambda(\eta)}e(0,\eta)$ and
\begin{multline}
  \label{eq:e-est}
 \|\eta^{k+1}c_1(t,\eta)\|_{L^2(-\eta_0,\eta_0)}^2
+\|\eta^kc_2(t,\eta)\|_{L^2(-\eta_0,\eta_0)}^2
\\ \lesssim
  \int_{-\eta_0}^{\eta_0} \eta^{2k}e(t,\eta)\,d\eta
\\  \lesssim 
 (1+t)^{-k}\{\|\eta^{k+1}c_1(0,\eta)\|_{L^2(-\eta_0,\eta_0)}^2
  +\|\eta^kc_2(0,\eta)\|_{L^2(-\eta_0,\eta_0)}^2\}
\\  \lesssim
(1+t)^{-k}(\|\Phi_0\|_{L^2_\a(\R^2)}+\|\Psi_0\|_{L^2_\a(\R^2)})\,.
\end{multline}  
because $\kappa(\eta)=\kappa_1\eta+O(\eta^3)$ with $\kappa_1\ne0$ and
$\Re\lambda(\eta)=-\lambda_2\eta^2+O(\eta^4)$ with $\lambda_2>0$.
Since $\kappa(\eta)$ and $\Im\lambda(\eta)$ are odd and $\Re\lambda(\eta)$
is even, it follows from Theorem~\ref{thm:1} and \eqref{eq:kappa-exp} that
$$\mathcal{A}(\eta)=\mathcal{A}_0(\eta)
+\begin{pmatrix}
  O(\eta^4) & O(\eta^2) \\ O(\eta^4) & O(\eta^4)
\end{pmatrix}\,,
\quad
\mathcal{A}_0(\eta)=
\begin{pmatrix}
  -\lambda_2\eta^2 & \frac{\lambda_1}{\kappa_1}
\\ -\lambda_1\kappa_1\eta^2 & -\lambda_2\eta^2
\end{pmatrix}\,.$$
By the variation of the constants formula,
\begin{align*}
\begin{pmatrix}    c_1(t,\eta) \\ c_2(t,\eta)  \end{pmatrix}
=& e^{t\mathcal{A}_0(\eta)}
\begin{pmatrix}c_1(0,\eta) \\ c_2(0,\eta)  \end{pmatrix}
-\int_0^t e^{(t-s)\mathcal{A}_0(\eta)}
\left(\mathcal{A}(\eta)-\mathcal{A}_0(\eta)\right)
\begin{pmatrix}    c_1(s,\eta) \\ c_2(s,\eta)  \end{pmatrix}
\,ds\,,
\end{align*}
where $e^{t\mathcal{A}_0(\eta)}=e^{-t\lambda_2\eta^2}
\begin{pmatrix}
  \cos t\lambda_1\eta & \frac{\sin t\lambda_1\eta}{\kappa_1\eta}
\\ - \kappa_1\eta\sin t\lambda_1\eta & \cos t\lambda_1\eta
\end{pmatrix}$. Using \eqref{eq:e-est}, we have for $k=0$ and $1$,
\begin{equation}
  \label{eq:pf24-1}
\begin{split}
& \left\|\eta^k\left\{c_1(t,\eta)
-e^{-t\lambda_2\eta^2}\frac{\sin t\lambda_1\eta}{\kappa_1\eta}
c_2(0,\eta) \right\}\right\|_{L^2(-\eta_0,\eta_0)}
\\ \lesssim & 
\|\eta e^{-t\lambda_2\eta^2}c_1(0,\eta)\|_{L^2(-\eta_0,\eta_0)}
+\sum_{j=1,2}
\int_0^t  \| \eta^{4+k-j}e^{-(t-s)\lambda_2\eta^2}c_j(s,\eta)\|_{L^2(-\eta_0,\eta_0)}\,ds
\\ \lesssim &
(1+t)^{-k/2}\|c_1(0,\eta)\|_{L^2(-\eta_0,\eta_0)}
\\ & +\int_0^t(1+t-s)^{-3/4}(1+s)^{-(2k+1)/4}\,ds
\|\eta^{\frac{5}{2}+k-j}e(0,\eta)\|_{L^2(-\eta_0,\eta_0)}
\\ \lesssim & (1+t)^{-k/2}(\|\Phi_0\|_{L^2_\a(\R^2)}+\|\Psi_0\|_{L^2_\a(\R^2)})\,.
\end{split} 
\end{equation}
Since $f(y)=\la \vec{\Phi}(0,\cdot,y), \zeta_2^*\ra$ and
$\|g_2^*(\cdot,\eta)-\zeta_2^*\|_{L^2_{-\a}(\R)}=O(\eta^2)$,  we have
\begin{align*}
|c_2(0,\eta)-\hat{f}(\eta)| \le &
\|\mF_y\vec{\Phi}(0,\cdot,\eta)\|_{L^2_\a(\R)}
\|g_2^*(\cdot,\eta)-\zeta_2^*\|_{L^2_{-\a}(\R)}
\\ \lesssim  & 
\eta^2(\|\mF_y\Phi_0(\cdot,\eta)\|_{L^2_\a(\R)}
+\|\mF_y\Psi_0(\cdot,\eta)\|_{L^2_\a(\R)})\,,
\end{align*}
and 
\begin{equation}
\label{eq:pf24-2}
\left\|e^{-t\lambda_2\eta^2}\frac{\sin t\lambda_1\eta}{\kappa_1\eta}
\{c_2(0,\eta)-\hat{f}(\eta)\}\right\|_{L^2(-\eta_0,\eta_0)}
\lesssim  (1+t)^{-1/2} (\|\Phi_0\|_{L^2_\a(\R^2)}+\|\Psi_0\|_{L^2_\a(\R^2)})\,.
\end{equation}
Combining \eqref{eq:pf24-1} and \eqref{eq:pf24-2} with $\|g_1(\cdot,\eta)-\zeta_1\|_{L^2_\a(\R)}=O(\eta^2)$,
we have for $k=0$ and $1$,
\begin{equation}
\label{eq:pf24-3}
\begin{split}
& \left\|\eta^k\left\{c_1(t,\eta)g_1(\cdot,\eta)-e^{-t\lambda_2\eta^2}
\frac{\sin t\lambda_1\eta}{\kappa_1\eta}\hat{f}(\eta)\zeta_1\right\}
\right\|_{L^2([-\eta_0,\eta_0];L^2_\a(\R_z))}
\\ \lesssim &
\|\eta^{k+2}c_1(t,\eta)\|_{L^2(-\eta_0,\eta_0)}
\sup_{0<|\eta|\le\eta_0}\eta^{-2}\|g_1(\cdot,\eta)-\zeta_1\|_{L^2_\a(\R)}
\\ & +
\left\|\eta^k\left\{c_1(t,\eta)-e^{-t\lambda_2\eta^2}\frac{\sin t\lambda_1\eta}{\kappa_1\eta}
\hat{f}(\eta)\right\}\right\|_{L^2(-\eta_0,\eta_0)}\|\zeta_1\|_{L^2_\a(\R)}
\\ \lesssim & (1+t)^{-k/2}(\|\Phi_0\|_{L^2_\a(\R^2)}+\|\Psi_0\|_{L^2_\a(\R^2)})\,.
\end{split}
\end{equation}
Since $\|\hat{f}\|_{L^2}=\|f\|_{L^2}\lesssim \|\Phi_0\|_{L^2_\a(\R^2)}+\|\Psi\|_{L^2_\a(\R^2)}$,
\begin{equation}
\label{eq:pf24-4}
\left\|e^{-t\lambda_2\eta^2}\frac{\sin t\lambda_1\eta}{\kappa_1\eta}
\hat{f}(\eta)\zeta_1(z)\right\|_{L^2(|\eta|\ge \eta_0;L^2_\a(\R_z))}
\lesssim e^{-t\lambda_2\eta_0^2}(\|\Phi_0\|_{L^2_\a(\R^2)}+\|\Psi_0\|_{L^2_\a(\R^2)})\,.
\end{equation}
Using the Plancherel theorem, \eqref{eq:e-est}, \eqref{eq:pf24-3} and \eqref{eq:pf24-4}, we have
\begin{align*}
& \|\pd_y^j\{
\mathcal{P}(\eta_0)\vec{\Phi}(t)-(H_t*W_t*f)(y)\zeta_1(z)\}\|_X
\\ \lesssim &
\left\|\eta^j\{c_1(t,\eta)g_1(z,\eta)
-e^{-t\lambda_2\eta^2}\frac{\sin t\lambda_1\eta}{\kappa_1\eta}\hat{f}(\eta)
\zeta_1(z)\}\right\|_{L^2([-\eta_0,\eta_0];H^1_\a(\R_z)\times L^2_\a(\R_z))}
\\ & 
+\|\eta^jc_2(t,\eta)\|_{L^2(-\eta_0,\eta_0)}
+\|\eta^je^{-t\lambda_2\eta^2}\frac{\sin t\lambda_1\eta}{\kappa_1\eta}
\hat{f}(\eta)\|_{L^2_\a(|\eta|\ge\eta_0)}
\\ \lesssim & (1+t)^{-j/2}(\|\Phi_0\|_{L^2_\a(\R^2)}+\|\Psi_0\|_{L^2_\a(\R^2)})\,.
\end{align*}
By Theorem~\ref{thm:linear-stability},
$$\|\mathcal{Q}(\eta_0)\vec{\Phi}(t)\|_{H^2_\a(\R^2)\times H^1_\a(\R^2)}
\lesssim e^{-\beta' t}(\|\Phi_0\|_{H^2_\a(\R^2)}+\|\Psi_0\|_{H^1_\a(\R^2)})
\,.$$ Combining the above, we obtain for $j=0$, $1$,
\begin{align*}
& \left\|\diag(\pd_z^j,1)\{
\vec{\Phi}(t,z,y)-(H_t*W_t*f)(y)\zeta_1(z)\}\right\|_{L^2_\a(\R_z)L^\infty(\R_y)}\
\\ \lesssim &
(1+t)^{-1/4}(\|\Psi_0\|_{H^2_\a(\R^2)}+\|\Psi_0\|_{H^1_\a(\R^2)})\,.
\end{align*}
This completes the proof of Theorem~\ref{thm:linear-dynamics}.
\bigskip

\appendix
\section{Miscellaneous estimates of operator norms}
\label{sec:ap1}
In this section, we collect estimates of the norm of operators.
\par
A solitary wave profile $q_c(x)$ is similar to KdV $1$-solitons provided
$c$ is close to $1$.
In view of \eqref{eq:KP2-scale2}, we have the following estimates on derivatives
of $q_c$.
\begin{claim}
\label{cl:qc-size}
Let $c=\sqrt{1+\eps^2}$, $\a=\hat{\a}\eps$ and $\hat{\a}\in(0,\hat{\a}_0/2)$.
There exists positive constants  $\eps_0$ and $C$ such that
$$\|\pd_z^i\pd_c^jq_c\|_{B(L^2_\a(\R))}\le C\eps^{2+i-2j}\,.
\quad\text{for $\eps\in(0,\eps_0)$ and $i$ $j\in\Z_{\ge0}$.}$$
\end{claim}
Next, we collect estimates of $\pd_z$, $\mu(D)$, $S(D)$ and $B^{-1}$.
\begin{claim}
\label{cl:basic}
Let $\hat{\a}>0$ and $\a=\hat{\a}\eps$.
There exists a positive constants $\eps_0$  such that
if $\eps\in(0,\eps_0)$, 
\begin{gather}
  \label{eq:inv-pd}
\|\pd_z^{-1}\|_{B(L^2_\a)} \le \a^{-1}\,,
\\ \label{eq:inv-mu}
\|\mu(D)^{-1}\|_{B(Y)}\le \sqrt{2}\a^{-1}\,, \quad
\|\pd_z\mu(D)^{-1}\|_{B(Y)}\le \sqrt{2}\,,
\\
\label{eq:pdon2}
\|\pd_z\|_{B(Y_{low})} \le (K+\hat{\a})\eps\,,
\quad
\|\mu(D)^j\|_{B(Y_{low})}\le \{2(K+\hat{\a})\eps\}^j
\text{for $j\in\N$,}\\
\label{eq:pdon3}
 \|i\pd_z\mu(D)^{-1}+I\|_{B(Y_{low})}=O(K^4\eps^2)\,.
\end{gather}
\end{claim}
\begin{proof}
By \eqref{eq:fm0},
$$\|\pd_z^{-1}\|_{B(L^2_\a)}= \sup_{\xi\in\R}
\left|\frac{1}{\xi+i\a}\right| \le  \a^{-1}\,,
$$
\begin{align*}
\|\pd_z^j\mu(D)^{-1}\|_{B(Y)}=& 
\sup_{(\xi,\eta)\in\R\times[-K(K+\hat{\a})\eps^2,K(K+\hat{\a})\eps^2]}
\frac{|\xi+i\a|^j}{|\mu(\xi+i\a,\eta)|}\,.
\end{align*}
If $\eta\in[- K^2\eps^2,K^2\eps^2]$ and $\eps$ is sufficiently small,
then $\eta^2\le \a^2/2$ and
\begin{align*}
|\mu(\xi+i\a,\eta)|^4=&  (\xi^2+\a^2-\eta^2)^2+4\xi^2\eta^2
\ge  \frac14(\xi^2+\a^2)^2\,.
\end{align*}
Combining the above,  we have \eqref{eq:inv-mu}.
\par

Since $\supp \hat{f}(\xi+i\a,\eta)\subset\widetilde{A}_{low}$ for $f\in Y_{low}$,
we have \eqref{eq:pdon2} and
\begin{align*}
\|i\pd_z\mu(D)^{-1}+I\|_{B(Y_{low})}
=& \sup_{(\xi,\eta)\in\widetilde{A}_{low}}
\left|\left\{1+\frac{\eta^2}{(\xi+i\a)^2}\right\}^{-1/2}-1\right|
 = O(K^4\eps^2)\,.
\end{align*}
Thus we complete the proof.
\end{proof}
  \begin{claim}
\label{cl:prop-B,S}
Let $\hat{\a}>0$ and $\a=\hat{\a}\eps$.
There exists  positive constants $C$ and $\eps_0$ such that for 
any $\eps\in(0,\eps_0)$,
    \begin{gather}
\label{eq:S}
\|S(D)\|_{B(L^2_\a)}+\|S(D)^{-1}\|_{B(L^2_\a)}\le C\,,\\
      \label{eq:binv1}
\|\pd_z^jB^{-1}\|_{B(L^2_\a,H^{2-j}_\a)}+\|\mu_j(D)^jB^{-1}\|_{B(L^2_\a,H^{2-j}_\a)}
\le C\quad\text{for $j=0$, $1$, $2$,}\\
\label{eq:binv3}
\|[B,\pd_z^jq_c]\|_{B(H^1_\a,L^2_\a)}\le C\eps^{j+3}\,,\\
      \label{eq:binv2}
\|B^{-1}-I\|_{B(Y_{low})}
+\|S(D)-I\|_{B(Y_{low})}+\|S^{-1}(D)-I\|_{B(Y_{low})} \le CK^2\eps^2\,.
    \end{gather}
  \end{claim}
  \begin{proof}
We can prove \eqref{eq:S}--\eqref{eq:binv3} in 
the same way as Lemmas~7.2 and 7.4 in \cite{MPQ13}.
Since $$B(\xi+i\a,\eta)=1+b\{(\xi+i\a)^2+\eta^2\}=1+O(K^2\eps^2)
\quad\text{for $(\xi,\eta)\in \widetilde{A}_{low}$,}$$
we have
\begin{align*}
\|B^{-1}-I\|_{Y_{low}}=& 
\sup_{(\xi,\eta)\in \widetilde{A}_{low}}\left|B^{-1}(\xi+i\a,\eta)-1\right|
= O(K^2\eps^2)\,.
\end{align*}
Similarly, we have 
$\|S(D)-I\|_{B(Y_{low})}+\|S^{-1}(D)-I\|_{B(Y_{low})}=O(K^2\eps^2)$
from \eqref{eq:S-low}.
\end{proof}

Next, we will estimates the operator norms of $a_1$ and $a_2$.
\begin{claim}
  \label{cl:V-bound12}
Let $\hat{\a}\in (0,\hat{\a}_0/2)$ and
$c=\sqrt{1+\eps^2}$. There exists an $\eps_0>0$ such that if
$\eps\in(0,\eps_0)$ and $\a=\eps\hat{\a}$, then
\begin{gather}
\label{eq:a1,a2}
 \|a_i\|_{B(Y)}=O(\eps^2)
\quad\text{for $i=1$, $2$,}
\\ \label{eq:a,l-h}
\| a_i \rho_z(D_z)\|_{B(Y)}+\|\rho_z(D_z)a_i\|_{B(Y)}=O(K\eps^3)
\quad\text{for $i=1$, $2$,}
\\ \label{eq:a,l-l}
\left \|\rho_{KP}(D)\{
a_{2,\eps}+\frac{3}{2}\pd_z(\theta_0\cdot)\}
\rho_{KP}(D)\right\|_{B(L^2_{\hat{\a}}(\R^2))}=O(K^5\eps^2)\,.
\end{gather}
\end{claim}
\begin{proof}
By Claims~\ref{cl:qc-size}--\ref{cl:prop-B,S},
\eqref{eq:v2c-alt} and \eqref{eq:v1c-alt'}, we have
\begin{align*}
& \|B^{-1}v_{1,c}\mu(D)^{-1}\|_{B(L^2_\a)}
+ \|B^{-1}v_{2,c}\|_{B(L^2_\a)}=O(\eps^2)\,,
\\ &
 \|B^{-1}v_{1,c}\mu(D)^{-1}\rho_z(D_z)\|_{B(Y)}
+\|B^{-1}v_{2,c}\rho_z(D_z)\|_{B(L^2_\a)}=O(K\eps^3)\,,
\\ &
 \|\rho_z(D_z)B^{-1}v_{1,c}\mu(D)^{-1}\|_{B(Y)}
+\|\rho_z(D_z)B^{-1}v_{2,c}\|_{B(L^2_\a)}=O(K\eps^3)\,.
\end{align*}  
Combining the above with \eqref{eq:S},
we have \eqref{eq:a1,a2} and \eqref{eq:a,l-h}.
\par
Finally, we will prove \eqref{eq:a,l-l}.
By \eqref{eq:binv2},
\begin{multline*}
  \left\|\rho_z(D_z)\{2a_2+3c(q\pd_z+q_c')\}\rho_z(D_z)\right\|_{B(Y)}
\\ \le  
\left\|\{iv_{1,c}\mu(D)^{-1}-c(q_c\pd_z+2q_c')\}\rho_z(D_z)\right\|_{B(Y)}
\\ +O\left(K^2\eps^2(\|\rho_z(D_z)v_{1,c}\mu(D)^{-1}\rho_z(D_z)\|_{B(Y)}
+\|\rho_z(D_z)v_{2,c}\rho_z(D_z)\|_{B(Y)})\right)\,.
\end{multline*}
Claims~\ref{cl:qc-size} and \ref{cl:basic} imply
$$\|\rho_z(D_z)v_{1,c}\mu(D)^{-1}\rho_z(D_z)\|_{B(Y)}
+\|\rho_z(D_z)v_{2,c}\rho_z(D_z)\|_{B(Y)}=O(K\eps^3)\,,$$
and 
\begin{align*}
& \left\|\{iv_{1,c}\mu(D)^{-1}-c(q_c\pd_z+2q_c')\}\rho_z(D_z)\right\|_{B(Y)}
\\ \lesssim & \|q_c\|_{L^\infty}\|(i\mu(D)-\pd_z)\rho_z(D_z)\|_{B(Y)}
+\|q_c'\|_{L^\infty}\|(i\pd_z\mu(D)^{-1}+I)\rho_z(D_z)\|_{B(Y)}
\\ &+(c-1)\|(q_c\pd_z+2q_c')\rho_z(D_z)\|_{B(Y)}
\\ \lesssim & K^5\eps^5\,.
\end{align*}
In the last inequality, we use the fact that $c=1+O(\eps^2)$.
Combining the above with the fact that
$\|\eps^{-2}q_c(\cdot/\eps)-\theta_0\|_{C^1}=O(\eps^2)$,
we have \eqref{eq:a,l-l}.
Thus we complete the proof.
\end{proof}

\begin{claim}
  \label{cl:tr-bound}
  \begin{gather}
\label{eq:tr-bound1}
 \|\barr_{ij}\|_{B(Y)}\lesssim K\eps^3
\quad\text{for  $i$, $j=1$, $2$,}
\\ \label{eq:tr-bound2}
\|\barr_{22}\|_{B(Y)}\lesssim K^5\eps^5\,.
  \end{gather}
\end{claim}
\begin{proof}
By Lemma~\ref{lem:orthogonality},
\begin{equation*}
\Pi^{-1}=
=\begin{pmatrix}  I & O \\ \eps_{21} & I+\eps_{22} \end{pmatrix}
\end{equation*}
with $\|\eps_{2j}\|_{B(L^2_\a(\R^2))}=O(K^{-1})$ and
for ${}^t\!(\tu_1,\tu_2)\in\widetilde{Z}$ and ${}^t\!(\baru_1,\baru_2)
=\Pi{}^t\!(\tu_1,\tu_2)$,
\begin{align*}
&   \begin{pmatrix}
 \barr_{11} & \barr_{12} \\ \barr_{21} & \barr_{22}
  \end{pmatrix}
\begin{pmatrix} \baru_1 \\ \baru_2 \end{pmatrix}
=
\left[\Pi,
    \begin{pmatrix}
\lambda_+(D)+a_1 & a_2 \\ a_1 &  \lambda_-(D)+a_2
    \end{pmatrix}\right]
\begin{pmatrix} \tu_1 \\ \tu_2 \end{pmatrix}
\\=&
\begin{pmatrix}
-a_2\rho_z(D_z)E_\eps^{-1}\mathcal{P}_{KP}(\eta_0)E_\eps\rho_z(D_z)\tu_2 \\ 
\rho_z(D_z)E_\eps^{-1}\mathcal{P}_{KP}(\eta_0)E_\eps\rho_z(D_z)a_1\tu_1
-[\lambda_-(D)+a_2,
\rho_z(D_z)E_\eps^{-1}\mathcal{P}_{KP}(\eta_0)E_\eps\rho_z(D_z)]\tu_2
\end{pmatrix}\,.
\end{align*}
Combining the above with Claim~\ref{cl:V-bound12}, we have
\eqref{eq:tr-bound1}.
\par
Next, we will prove \eqref{eq:tr-bound2} by using the KP-II approximation
of $\lambda_{-,\eps}(D)+a_{2,\eps}$ in the low frequency regime.
Since 
\begin{align*}
\barr_{22}\baru_2=&
-[\lambda_-(D)+a_2,
\rho_z(D_z)E_\eps^{-1}\mathcal{P}_{KP}(\eta_0)E_\eps\rho_z(D_z)]\tu_2
\\=&
-\eps^3E_\eps^{-1} [\lambda_{-,\eps}(D)+a_{2,\eps},
\rho_z(\eps D_z)\mathcal{P}_{KP}(\eta_0)\rho_z(\eps D_z)]E_\eps\tu_2\,,  
\end{align*}
it follows from \eqref{eq:lmm-low2}, \eqref{eq:Proj-approx3} and
\eqref{eq:a,l-l},
\begin{equation}
  \label{eq:proj-approx1}
  \begin{split}
& \|(\lambda_{-,\eps}(D)+a_{2,\eps})\rho_z(\eps D_z)g_{0,k}(\cdot,\eta)
-\mL_{KP}(\eta)g_{0,k}(\cdot,\eta)\|_{L^2_{\hat{\a}}}
\\ & +\|(\lambda_{-,\eps}(D)+a_{2,\eps})^*\rho_z(\eps D_z)g_{0,k}^*(\cdot,\eta)
-\mL_{KP}(\eta)^*g_{0,k}^*(\cdot,\eta)\|_{L^2_{-\hat{\a}}}
\\=&O(K^8\eps^2)\,.      
  \end{split}
\end{equation}
Since $\mL_{KP}\mathcal{P}_{KP}(\eta_0)=\mathcal{P}_{KP}(\eta_0)\mL_{KP}$,
we have  \eqref{eq:tr-bound2} from \eqref{eq:proj-approx1}. 
\end{proof}
\bigskip

\section*{Acknowledgment}
This research is supported by JSPS KAKENHI Grant Number JP25400174.

\end{document}